\documentclass[a4paper,11pt]{article}
\usepackage{amssymb, latexsym, amsmath, amsthm,enumerate}
\usepackage{tikz}
\usepackage{xcolor}

\title {The Ramsey numbers for trees of order $n$ with maximum degree at least $n-5$ versus the wheel graph of order nine}
\author{Zhi Yee Chng\thanks{School of Mathematics and Statistics, UNSW Sydney, NSW 2052, Sydney, Australia.\newline Email: zhi\_yee.chng@unsw.edu.au.} \and Thomas Britz \thanks{School of Mathematics and Statistics, UNSW Sydney, NSW 2052, Sydney, Australia. Email: britz@unsw.edu.au.} \and Ta Sheng Tan\thanks{Institute of Mathematical Sciences, Faculty of Science, Universiti Malaya, 50603 Kuala Lumpur, Malaysia. Email: tstan@um.edu.my.} \and Kok Bin Wong\thanks{Institute of Mathematical Sciences, Faculty of Science, Universiti Malaya, 50603 Kuala Lumpur, Malaysia.\newline Email: kbwong@um.edu.my.}}
\date{}

\theoremstyle{plain}
\newtheorem{theorem}{Theorem}[section]

\newtheorem{lemma}[theorem]{Lemma}
\newtheorem{corollary}[theorem]{Corollary}
\newtheorem{observation}[theorem]{Observation}

\theoremstyle{definition}

\theoremstyle{remark}

\tikzstyle{vertex} = [draw, circle, scale=0.5]
\tikzstyle{nullvertex} = [draw=white, circle, scale=0.05]
\tikzstyle{cir} = [draw, circle, dotted, scale=5.0]
\tikzstyle{bigcir} = [draw, circle, dotted, scale=10.0]

\begin{document}

\maketitle

\begin{abstract}
The Ramsey numbers $R(T_n,W_8)$ are determined for each tree graph $T_n$ 
of order $n\geq 7$ and maximum degree $\Delta(T_n)$ equal to either $n-4$ or $n-5$.
These numbers indicate strong support for the conjecture, 
due to Chen, Zhang and Zhang and to Hafidh and Baskoro, 
that $R(T_n,W_m) = 2n-1$ for each tree graph $T_n$ 
of order $n\geq m-1$ with $\Delta(T_n)\leq n-m+2$ when $m\geq 4$ is even. 
\end{abstract}

\normalsize

\section{Introduction}

Let $G$ and $H$ be two simple graphs.
The Ramsey number $R(G,H)$ is the smallest integer $n$ such that, 
for any graph of order $n$, 
either it contains $G$ or its complement contains $H$ as a subgraph. 
Chv\'atal and Harary~\cite{ChHa72} proved that $R(G,H)\geq (c(G)-1)(\chi(H)-1)+1$
where $c(G)$ is the largest order of any connected component of $G$ 
and where $\chi(H)$ is the chromatic number of $H$.
For any tree graph $G=T_n$ of order $n$ and the wheel graph $H = W_m$ of order $m+1$
obtained by connecting a vertex to each vertex of the cycle graph $C_m$, 
the Chv\'atal-Harary bound implies that $R(T_n,W_m)\geq 2n-1$ when $m$ is even
and $R(T_n,W_m)\geq 3n-2$ when $m$ is odd.

Chen et al.~{\cite{ChZhZh05c}} and Zhang~\cite{Zh08a} 
showed that $R(P_n,W_m)$ achieves these Chv\'atal-Harary bounds 
for the path graph $T_n=P_n$ of order $n$ when $m$ is odd and $3\leq m\leq n+1$ 
and when $m$ is even and $4\leq m\leq n+1$; see also~\cite{baskoro02,SaBr07}.
Baskoro et al.~\cite{BaSuNaMi02} and Surahmat and Baskoro~\cite{SuBa01}
further proved that $R(T_n,W_m)$ achieves the Chv\'atal-Harary bounds
for $m = 4,5$ and all tree graphs $T_n$ of order $n\geq 3$,
except when $m=4$ and $T_n$ is the star graph $S_n$, 
in which case $R(S_n,W_4) = 2n+1$.
This led Baskoro et al.~\cite{BaSuNaMi02} to conjecture that 
$R(T_n,W_m) = 3m-2$ for all tree graphs $T_n$ of order $n$ when $m\geq 5$ is odd.
The conjecture is true for all sufficiently large $n$, 
according to a result of Burr et al.~\cite{BuErFaRoScGoJa87}.
In contrast, the analogous equality $R(T_n,W_m) = 2n-1$ for even $m\geq 4$ is false
since the star graph $T_n = S_n$ does not achieve this bound,
as the following combined result of Zhang~\cite{Zh08b} and Zhang et al.~\cite{ZhChZh08,ZhChCh09} shows; 
see also~\cite{ChZhZh04a, HaMa18, HaBaAs05, korolova05, LiSc16}.

\begin{theorem}{\rm\cite{Zh08b,ZhChZh08,ZhChCh09}}
\label{thm:R(Sn,W8)}
For $n\geq 5$, 
\[
  R(S_n,W_8)=\begin{cases}
    2n+1 & \text{if $n$ is odd}\,;\\
    2n+2 & \text{if $n$ is even}\,.
  \end{cases}
\]
\end{theorem}

Baskoro et al.~\cite{BaSuNaMi02} therefore conjectured that 
$R(T_n,W_m) = 2n-1$ for all non-star tree graphs $T_n$ of order $n$ when $n\geq 4$ is even.
This conjecture was disproved by Chen, Zhang and Zhang~\cite{ChZhZh04b}
who showed that $R(T_n,W_6)=2n$ for certain non-star tree graphs $T_n$.
Zhang~\cite{Zh08a} further proved the following theorem which shows that 
the conjecture is false when $n$ is small, even for the path graph $P_n$; 
see also~\cite{BaSu05, ChZhZh05c, LiNi14, SaBr07}.

\begin{theorem}{\rm\cite{Zh08a}}
\label{thm:R(Pn,Wm)}
If $m$ is even and $n+2\leq m\leq 2n$, 
then $R(P_n,W_m) = m+n-2$.
\end{theorem}

However, Chen, Zhang and Zhang~\cite{ChZhZh04b} conjectured that $R(T_n,W_m) = 2n-1$ for all tree graphs $T_n$ of order $n\geq m-1$ 
when $m$ is even and the maximum degree $\Delta(T_n)$ ``is not too large";
see also~\cite{ChZhZh05a,ChZhZh05b,ChZhZh06}.
Hafidh and Baskoro~\cite{HaBa21} refined this conjecture by specifying the bound $\Delta(T_n)\leq n-m+2$.
When $n$ is large compared to $m$, $\Delta(T_n)$ is not required to be small; 
indeed, the refined conjecture implies that, for each fixed even integer $m$, 
all but a vanishing proportion of the tree graphs $\{T_n \::\: n\geq m-1\}$ satisfy $R(T_n,W_m)=2n-1$.

For $m=8$, the bound is $\Delta(T_n)\leq n-6$.
There is exactly one tree graph $T_n$ of order $n$ with maximum degree $\Delta(T_n) = n-1$, 
namely the star graph $S_n$; see Theorem~\ref{thm:R(Sn,W8)}.
There is exactly one tree graph $T_n$ of order $n\geq m-1$ with maximum degree $\Delta(T_n) = n-2$:
the graph $S_n(1,1)$ obtained by subdividing an edge of $S_{n-1}$.
More generally, 
let $S_n(\ell,m)$ be the tree graph of order $n$ obtained by 
subdividing $m$ times each of $\ell$ chosen edges of $S_{n-\ell m}$;
see Figure~\ref{fig:Tree example}.

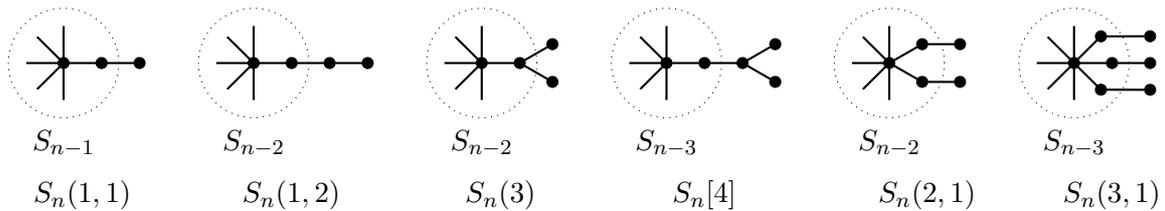
\begin{figure}[ht!]
  \centering
  \begin{tikzpicture}[scale=.5]
    \begin{scope}[shift={(0,0)}]
      \foreach \nn in {1,...,5}{\node[nullvertex] (\nn) at (45+\nn*45:1) {}; \draw[thick] (0,0) -- (\nn);}
      \node[cir] at (0,0)[label=below:$S_{n-1}$,scale=0.8] {};
      \node[vertex,fill=black,scale=.8] (6) at (0,0) {};
      \node[vertex,fill=black,scale=.8] (7) at (1,0) {};
      \node[vertex,fill=black,scale=.8] (8) at (2,0) {};
      \draw[thick] (6)--(7)--(8);
      \node () at (.5,-3.5) {$S_n(1,1)$};
    \end{scope}
    \begin{scope}[shift={(5,0)}]
      \foreach \nn in {1,...,5}{\node[nullvertex] (\nn) at (45+\nn*45:1) {}; \draw[thick] (0,0) -- (\nn);}
      \node[cir] at (0,0)[label=below:$S_{n-2}$,scale=0.8] {};
      \node[vertex,fill=black,scale=.8] (6) at (0,0) {};
      \node[vertex,fill=black,scale=.8] (7) at (1,0) {};
      \node[vertex,fill=black,scale=.8] (8) at (2,0) {};
      \node[vertex,fill=black,scale=.8] (9) at (3,0) {};
      \draw[thick] (6)--(7)--(8)--(9);
      \node () at (1,-3.5) {$S_n(1,2)$};
    \end{scope}
    \begin{scope}[shift={(11,0)}]
      \foreach \nn in {1,...,5}{\node[nullvertex] (\nn) at (45+\nn*45:1) {}; \draw[thick] (0,0) -- (\nn);}
      \node[cir] at (0,0)[label=below:$S_{n-2}$,scale=0.8] {};
      \node[vertex,fill=black,scale=.8] (6) at (0,0) {};
      \node[vertex,fill=black,scale=.8] (7) at (1,0) {};
      \node[vertex,fill=black,scale=.8] (8) at (1.86, .5) {};
      \node[vertex,fill=black,scale=.8] (9) at (1.86,-.5) {};
      \draw[thick] (7)--(9) (6)--(7)--(8);
      \node () at (.5,-3.5) {$S_n(3)$};
    \end{scope}
    \begin{scope}[shift={(15.86,0)}]
      \foreach \nn in {1,...,5}{\node[nullvertex] (\nn) at (45+\nn*45:1) {}; \draw[thick] (0,0) -- (\nn);}
      \node[cir] at (0,0)[label=below:$S_{n-3}$,scale=0.8] {};
      \node[vertex,fill=black,scale=.8]  (6) at (0   ,0  ) {};
      \node[vertex,fill=black,scale=.8]  (7) at (1   ,0  ) {};
      \node[vertex,fill=black,scale=.8]  (8) at (2   ,0  ) {};
      \node[vertex,fill=black,scale=.8]  (9) at (2.86, .5) {};
      \node[vertex,fill=black,scale=.8] (10) at (2.86,-.5) {};
      \draw[thick](8)--(9) (6)--(7)--(8)--(10);
      \node () at (1,-3.5) {$S_n[4]$};
    \end{scope}
    \begin{scope}[shift={(21.72,0)}]
      \foreach \nn in {1,...,5}{\node[nullvertex] (\nn) at (45+\nn*45:1) {}; \draw[thick] (0,0) -- (\nn);}
      \node[cir] at (0,0)[label=below:$S_{n-2}$,scale=0.8] {};
      \node[vertex,fill=black,scale=.8]  (6) at ( 0 ,0) {};
      \node[vertex,fill=black,scale=.8]  (7) at ( 30:1) {};
      \node[vertex,fill=black,scale=.8]  (8) at (-30:1) {};
      \node[vertex,fill=black,scale=.8]  (9) at (1.86, .5) {};
      \node[vertex,fill=black,scale=.8] (10) at (1.86,-.5) {};
      \draw[thick] (9)--(7)--(6)--(8)--(10);
      \node () at (1,-3.5) {$S_n(2,1)$};
    \end{scope}
    \begin{scope}[shift={(26.58,0)}]
      \foreach \nn in {1,...,5}{\node[nullvertex] (\nn) at (45+\nn*45:1) {}; \draw[thick] (0,0) -- (\nn);}
      \node[cir] at (0,0)[label=below:$S_{n-3}$,scale=0.8] {};
      \node[vertex,fill=black,scale=.8] (6) at (0,0) {};
      \foreach \nn in {7,8,9}{\node[vertex,fill=black,scale=.8] (\nn) at (\nn*45:1) {}; \draw[thick] (0,0) -- (\nn);}
      \node[vertex,fill=black,scale=.8] (10) at (2,-0.71) {};
      \node[vertex,fill=black,scale=.8] (11) at (2, 0   ) {};
      \node[vertex,fill=black,scale=.8] (12) at (2, 0.71) {};
      \draw[thick](7)--(10) (8)--(11) (9)--(12);
      \node () at (1,-3.5) {$S_n(3,1)$};
    \end{scope}
    \end{tikzpicture}
  \caption{Examples of $S_n(\ell,m)$, $S_n(\ell)$ and $S_n[\ell]$}\label{fig:Tree example}
\end{figure}

By Theorem~\ref{thm:R(Pn,Wm)}, $R(P_4,W_8) = 10$.
Hafidh and Baskoro~\cite{HaBa21} determined the Ramsey number $R\big(S_n(1,1),W_8\big)$ as follows.

\begin{theorem}{\rm\cite{HaBa21}}
\label{thm:R(Sn(1,1),W8)}
For $n\geq 5$, 
\[
  R\big(S_n(1,1),W_8\big)
  =
  \begin{cases}
    2n+1 & \text{if $n$ is odd}\;,\\
    2n   & \text{if $n$ is even}\,.
  \end{cases}
\]
\end{theorem}

There are exactly $3$ tree graphs $T_n$ of order $n$ with maximum degree $n-3$,
namely $S_n(1,2)$, $S_n(3)$ and $S_n(2,1)$, 
where $S_n(\ell)$ is the tree graph of order $n$ obtained by adding an edge joining 
the centers of two star graphs $S_\ell$ and $S_{n-\ell}$;
see Figure~\ref{fig:Tree example}. 
By Theorem~\ref{thm:R(Pn,Wm)}, $R(P_5,W_8)=11$. 
Hafidh and Baskoro~\cite{HaBa21} determined the Ramsey numbers for the three other graphs as follows.

\begin{theorem}{\rm\cite{HaBa21}}\label{thm:R(Pn-Sn(1,2)-Sn(3)-Sn(2,1),W8)}
For $n\geq 6$, 
\begin{align*}
  R(S_n(1,2),W_8)
 &= \begin{cases}
      2n+1 & \mbox{if $n\equiv 3 \pmod{4}$}\,;\\
      2n   & \mbox{otherwise}
    \end{cases}\\
  R(S_n(3),W_8)
 &=\begin{cases}
    2n-1   & \text{if $n$ is odd and $n\geq 9$}\,;\\
    2n     & \text{otherwise}
  \end{cases}\\
 R(S_n(2,1),W_8)
&=\begin{cases}
    2n-1   & \text{if $n$ is odd}\,;\\
    2n     & \text{otherwise}\,.
  \end{cases}
\end{align*}
\end{theorem}

The purpose of the present paper is to determine the Ramsey numbers $R(T_n,W_8)$ 
for all tree graphs $T_n$ of order $n\geq 6$ with maximal degree $\Delta(T_n) \geq n-5$; 
see Theorems~\ref{thm:ABCDE}, \ref{thm:R(Sn(4)-Sn[4]-Sn(1,3)-TA-TB-TC-Sn(3,1),W8)} and~\ref{thm:R(Sn(1,4)-TDEFGHJKLM)}
in Sections~\ref{sec:n-4} and~\ref{sec:n-5}.
These Ramsey numbers show that the proportion of tree graphs $T_n$ that satisfy the equality $R(T_n,W_8) = 2n-1$ 
quickly grows as the maximal degree $\Delta(T_n)$ decreases.
When $\Delta(T_n) \geq n-2$, no tree graph $T_n$ satisfies the equality.
In contrast when $\Delta(T_n) = n-3$, 
roughly one third of all tree graphs $T_n$ satisfy the equality; 
see Theorem~\ref{thm:R(Pn-Sn(1,2)-Sn(3)-Sn(2,1),W8)}.
When $\Delta(T_n) = n-4$, 
more than 85\%\ of all tree graphs $T_n$ satisfy the equality; 
see Theorems~\ref{thm:ABCDE} and~\ref{thm:R(Sn(4)-Sn[4]-Sn(1,3)-TA-TB-TC-Sn(3,1),W8)}.
And when $\Delta(T_n) = n-5$, roughly 94.7\%\ of all tree graphs $T_n$ satisfy the equality; 
see Theorem~\ref{thm:R(Sn(1,4)-TDEFGHJKLM)}.
These results thereby lend strong support for the conjecture described above by 
Chen, Zhang and Zhang~\cite{ChZhZh04b} and Hafidh and Baskoro~\cite{HaBa21}.

The contents of the present paper are as follows.
Sections~\ref{sec:n-4} and~\ref{sec:n-5} present the main results, 
namely Theorems~\ref{thm:ABCDE}, 
\ref{thm:R(Sn(4)-Sn[4]-Sn(1,3)-TA-TB-TC-Sn(3,1),W8)} 
and~\ref{thm:R(Sn(1,4)-TDEFGHJKLM)} mentioned above.
Section~\ref{sec:auxiliary_results} provides useful auxiliary results
that are used in the proofs of the main results. 
These proofs are presented in 
Sections~\ref{sec:proofofn-4ABCDE}, \ref{sec:proofofn-4} and~\ref{sec:n-5proof}, 
respectively.

\section{The Ramsey numbers $R(T_n,W_8)$ for $\Delta(T_n)=n-4$}
\label{sec:n-4}

This section presents the Ramsey numbers $R(T_n,W_8)$ for all tree graphs $T_n$ of order $n\geq 6$ with $\Delta(T_n) = n-4$. 
For $n=6$, there is just one such graph, namely the path graph $T_6 = P_6$.
Theorem~\ref{thm:R(Pn,Wm)} provides the Ramsey number $R(P_6,W_8) = 12$.
For $n=7$, there are five tree graphs with $\Delta(T_n) = n-4$, 
namely the graphs $A$, $B$, $C$, $D$ and $E$ shown in Figure~\ref{fig:Tree graph of order 7}. 

\begin{figure}[ht!]
  \centering
  \begin{tikzpicture}[scale=.55]
    \begin{scope}[shift={(0,0)}]
      \foreach \nn/\x/\y in {1/0/1, 2/1/1, 3/2/1/, 4/3/1, 5/4/1, 6/5/1, 7/1/0}{%
        \node [vertex,fill=black,scale=.8] (\nn) at (\x,\y) {};}
      \draw (1) -- (2) -- (3) -- (4) -- (5) -- (6) (7) -- (2);
      \node () at (2.5,-1.25) {$A$};
    \end{scope}
    \begin{scope}[shift={(7,0)}]
      \foreach \nn/\x/\y in {1/0/1, 2/1/1, 3/2/1/, 4/3/1, 5/4/1, 6/5/1, 7/2/0}{%
        \node [vertex,fill=black,scale=.8] (\nn) at (\x,\y) {};}
      \draw (1) -- (2) -- (3) -- (4) -- (5) -- (6) (7) -- (3);
      \node () at (2.5,-1.25) {$B$};
    \end{scope}
    \begin{scope}[shift={(14,0)}]
      \foreach \nn/\x/\y in {1/0/1, 2/1/1, 3/2/1/, 4/3/1, 5/4/1, 6/1/0, 7/2/0}{%
        \node [vertex,fill=black,scale=.8] (\nn) at (\x,\y) {};}
      \draw (1) -- (2) -- (3) -- (4) -- (5) (6) -- (2) (7) -- (3);
      \node () at (2,-1.25) {$C$};
    \end{scope}
    \begin{scope}[shift={(20,0)}]
      \foreach \nn/\x/\y in {1/0/1, 2/1/1, 3/2/1/, 4/3/1, 5/4/1, 6/2/0.25, 7/2/-0.5}{%
        \node [vertex,fill=black,scale=.8] (\nn) at (\x,\y) {};}
      \draw (1) -- (2) -- (3) -- (4) -- (5) (3) -- (6) -- (7);
      \node () at (2,-1.25) {$D$};
    \end{scope}
    \begin{scope}[shift={(26,0)}]
      \foreach \nn/\x/\y in {1/0/1, 2/1/1, 3/2/1/, 4/3/1, 5/4/1, 6/1/0, 7/3/0}{%
        \node [vertex,fill=black,scale=.8] (\nn) at (\x,\y) {};}
      \draw (1) -- (2) -- (3) -- (4) -- (5) (6) -- (2) (7) -- (4);
      \node () at (2,-1.25) {$E$};
    \end{scope}
  \end{tikzpicture}
  \caption{Tree graphs of order $7$ with $\Delta(T_n)=n-4$.}\label{fig:Tree graph of order 7}
\end{figure}
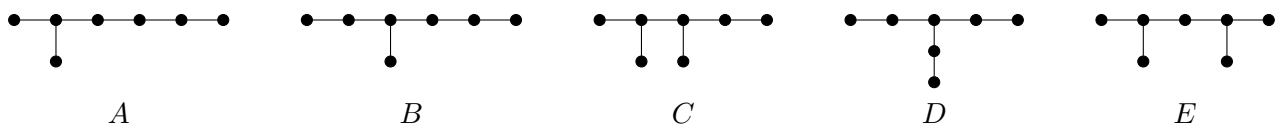

The Ramsey numbers $R(T_n,W_8)$ for these tree graphs are determined as follows.
\begin{theorem}
\label{thm:ABCDE}
$R(T,W_8) = 13$ for each $T\in\{A,B,C\}$, 
$R(D,W_8)=14$ and 
$R(E,W_8)=15$.
\end{theorem}

For $n\geq 8$, there are $7$ tree graphs $T_n$ of order $n$ with $\Delta(T_n) = n-4$,
namely the graphs $S_n(4)$, $S_n[4]$, $S_n(1,3)$, $S_n(3,1)$, $T_A(n)$, $T_B(n)$ and $T_C(n)$ shown in Figures~\ref{fig:Tree example} and~\ref{fig:Tree graph Delta=n-4},
where $S_n[\ell]$ is the tree graph of order $n$ obtained by adding an edge joining the center of $S_{n-\ell}$ to a degree-one vertex of $S_\ell$; 
see Figure~\ref{fig:Tree example}.

\begin{figure}[ht!]
  \centering
  \begin{tikzpicture}[scale=.5]
    \begin{scope}[shift={(0,0)}]
      \foreach \nn in {1,...,5}{\node[nullvertex] (\nn) at (45+\nn*45:1) {}; \draw[thick] (0,0) -- (\nn);}
      \node[cir] at (0,0)[label=below:$S_{n-4}$,scale=0.8] {};
        \node[vertex,fill=black,scale=.8] (6) at (0,0) {};
        \node[vertex,fill=black,scale=.8] (7) at (2,0) {};
        \node[vertex,fill=black,scale=.8] (8) at (3,1) {};
        \node[vertex,fill=black,scale=.8] (9) at (3,-1) {};
        \node[vertex,fill=black,scale=.8] (10) at (4,1) {};
        \draw[thick](7)--(9)
                    (6)--(7)--(8)--(10);
      \node () at (1,-3.5) {$T_A(n)$};
    \end{scope}
    \begin{scope}[shift={(8,0)}]
      \foreach \nn in {1,...,5}{\node[nullvertex] (\nn) at (45+\nn*45:1) {}; \draw[thick] (0,0) -- (\nn);}
      \node[cir] at (0,0)[label=below:$S_{n-5}$,scale=0.8] {};
        \node[vertex,fill=black,scale=.8] (6) at (0,0) {};
        \node[vertex,fill=black,scale=.8] (7) at (2,1) {};
        \node[vertex,fill=black,scale=.8] (8) at (2,-1) {};
        \node[vertex,fill=black,scale=.8] (9) at (3,1) {};
        \node[vertex,fill=black,scale=.8] (10) at (3,-1) {};
        \node[vertex,fill=black,scale=.8] (11) at (4,1) {};
        \draw[thick](6)--(8)--(10)
                    (6)--(7)--(9)--(11);
      \node () at (1,-3.5) {$T_B(n)$};
    \end{scope}
    \begin{scope}[shift={(16,0)}]
      \foreach \nn in {1,...,5}{\node[nullvertex] (\nn) at (45+\nn*45:1) {}; \draw[thick] (0,0) -- (\nn);}
      \node[cir] at (0,0)[label=below:$S_{n-5}$,scale=0.8] {};
        \node[vertex,fill=black,scale=.8] (6) at (0,0) {};
        \node[vertex,fill=black,scale=.8] (7) at (2,1) {};
        \node[vertex,fill=black,scale=.8] (8) at (2,-1) {};
        \node[vertex,fill=black,scale=.8] (9) at (3,1) {};
        \node[vertex,fill=black,scale=.8] (10) at (3,0) {};
        \node[vertex,fill=black,scale=.8] (11) at (3,-2) {};
        \draw[thick](6)--(7)--(9)
                    (6)--(8)--(10)
                    (8)--(11);
      \node () at (1,-3.5) {$T_C(n)$};
    \end{scope}
    \end{tikzpicture}
  \caption{Three tree graphs with $\Delta(T_n)=n-4$.}\label{fig:Tree graph Delta=n-4}
\end{figure}
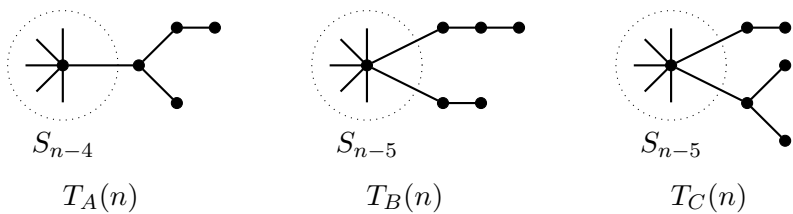

The Ramsey numbers $R(T_n,W_8)$ for these seven tree graphs are determined as follows.
\begin{theorem}
\label{thm:R(Sn(4)-Sn[4]-Sn(1,3)-TA-TB-TC-Sn(3,1),W8)}
If $n\geq 8$, then 
\begin{align*}
  R(S_n(4),W_8)
  &=\begin{cases}
    2n-1 & \text{if $n\geq 9$}\,;\\
    16   & \text{if $n=8$}
  \end{cases}\\
  R(T_n,W_8)
  &=\begin{cases}
    2n-1 & \text{if $n\not\equiv 0 \pmod{4}$}\,;\\
    2n   & \text{otherwise}\\
  \end{cases}\\
  R(T_n',W_8)
  &=2n-1\,,
\end{align*}
for each $T_n\in \{S_n[4],S_n(1,3),T_A(n),T_B(n)\}$ and 
$T_n'\in \{T_C(n),S_n(3,1)\}$.
\end{theorem}

Proofs of Theorems~\ref{thm:ABCDE} and~\ref{thm:R(Sn(4)-Sn[4]-Sn(1,3)-TA-TB-TC-Sn(3,1),W8)} are given in 
Sections~\ref{sec:proofofn-4ABCDE} and~\ref{sec:proofofn-4}.

\section{The Ramsey numbers $R(T_n,W_8)$ for $\Delta(T_n)=n-5$}
\label{sec:n-5}

This section presents the Ramsey numbers $R(T_n,W_8)$ for all tree graphs $T_n$ of order $n\geq 7$ with $\Delta(T_n) = n-5$. 
For $n=7$, there is just one such graph, namely the path graph $T_7 = P_7$.
Theorem~\ref{thm:R(Pn,Wm)} provides the Ramsey number $R(P_7,W_8) = 13$.
For $n\geq 8$, there are $19$ tree graphs $T_n$ of order $n$ with $\Delta(T_n) = n-5$,
namely $S_n(1,4)$, $S_n(5)$, $S_n[5]$, $S_n(2,2)$, $S_n(4,1)$ and the tree graphs 
shown in Figure~\ref{fig:Tree graph Delta=n-5}. 

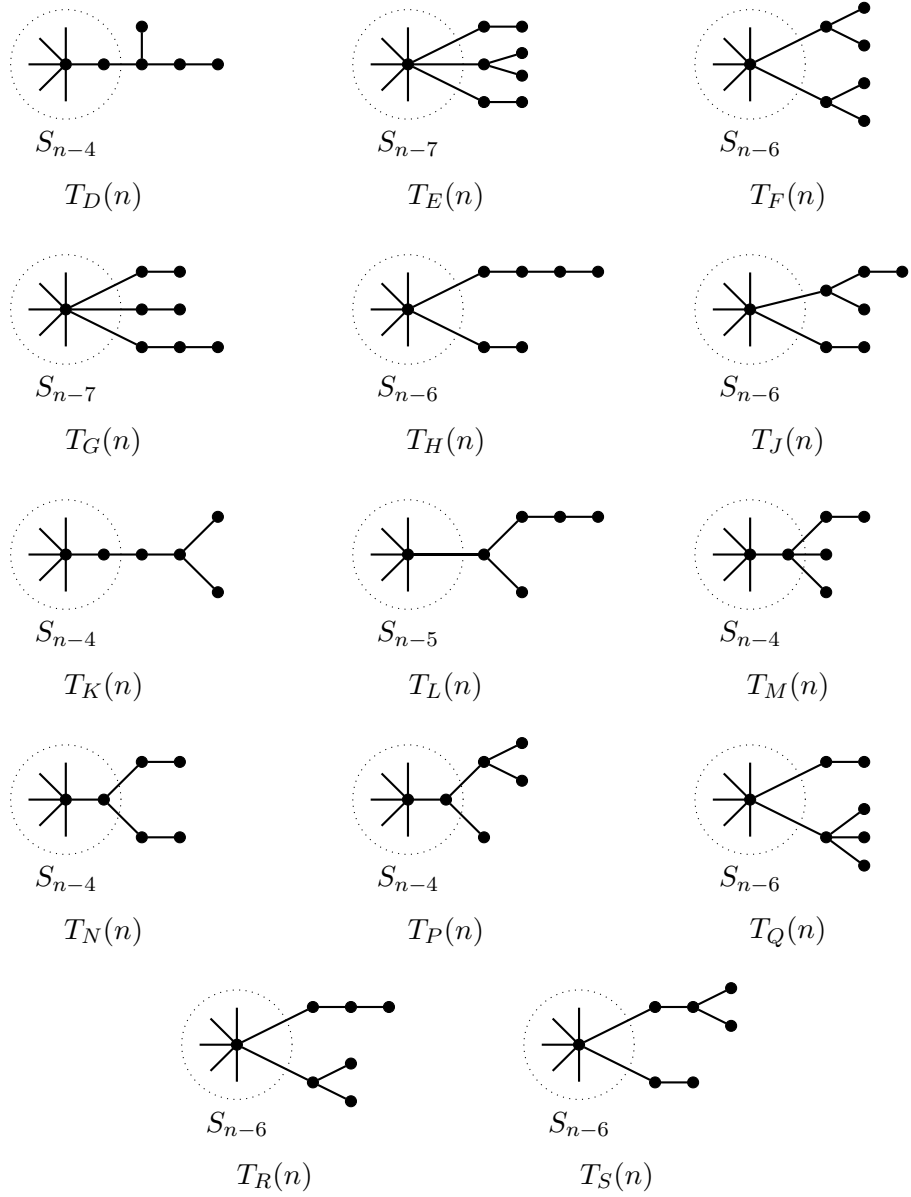
\begin{figure}[ht!]
  \centering
  \begin{tikzpicture}[scale=.5]
    \begin{scope}[shift={( 0,0)}]
      \draw[white] (-1.1,-4) rectangle (5.1,2);
      \foreach \nn in {1,...,5}{\node[nullvertex] (\nn) at (45+\nn*45:1) {}; \draw[thick] (0,0) -- (\nn);}
      \node[cir] at (0,0)[label=below:$S_{n-4}$,scale=0.8] {};
      \node[vertex,fill=black,scale=.8] (6) at (0,0) {};
      \node[vertex,fill=black,scale=.8] (7) at (1,0) {};
      \node[vertex,fill=black,scale=.8] (8) at (2,0) {};
      \node[vertex,fill=black,scale=.8] (9) at (2,1) {};
      \node[vertex,fill=black,scale=.8] (10) at (3,0) {};
      \node[vertex,fill=black,scale=.8] (11) at (4,0) {};
      \draw[thick](8)--(9)
                  (6)--(7)--(8)--(10)--(11);
      \node () at (1,-3.5) {$T_D(n)$};
    \end{scope}
    \begin{scope}[shift={(9,0)}]
      \draw[white] (-1.1,-4) rectangle (5.1,2);
      \foreach \nn in {1,...,5}{\node[nullvertex] (\nn) at (45+\nn*45:1) {}; \draw[thick] (0,0) -- (\nn);}
      \node[cir] at (0,0)[label=below:$S_{n-7}$,scale=0.8] {};
      \node[vertex,fill=black,scale=.8] (6) at (0,0) {};
      \node[vertex,fill=black,scale=.8] (7) at (2,1) {};
      \node[vertex,fill=black,scale=.8] (8) at (3,1) {};
      \node[vertex,fill=black,scale=.8] (9) at (2,0) {};
      \node[vertex,fill=black,scale=.8] (10) at (2,-1) {};
      \node[vertex,fill=black,scale=.8] (11) at (3,-1) {};
      \node[vertex,fill=black,scale=.8] (12) at (3,0.3) {};
      \node[vertex,fill=black,scale=.8] (13) at (3,-0.3) {};
      \draw[thick](11)--(10)--(6)--(7)--(8)
                  (6)--(9)--(12)
                  (9)--(13);
      \node () at (1,-3.5) {$T_E(n)$};
    \end{scope}
    \begin{scope}[shift={(18,0)}]
      \draw[white] (-1.1,-4) rectangle (5.1,2);
      \foreach \nn in {1,...,5}{\node[nullvertex] (\nn) at (45+\nn*45:1) {}; \draw[thick] (0,0) -- (\nn);}
      \node[cir] at (0,0)[label=below:$S_{n-6}$,scale=0.8] {};
      \node[vertex,fill=black,scale=.8] (6) at (0,0) {};
      \node[vertex,fill=black,scale=.8] (7) at (2,1) {};
      \node[vertex,fill=black,scale=.8] (8) at (3,1.5) {};
      \node[vertex,fill=black,scale=.8] (9) at (2,-1) {};
      \node[vertex,fill=black,scale=.8] (10) at (3,-1.5) {};
      \node[vertex,fill=black,scale=.8] (11) at (3,0.5) {};
      \node[vertex,fill=black,scale=.8] (12) at (3,-0.5) {};
      \draw[thick](8)--(7)--(6)--(9)--(10)
                  (7)--(11)
                  (9)--(12);
      \node () at (1,-3.5) {$T_F(n)$};
    \end{scope}
    \begin{scope}[shift={( 0,-6.5)}]
      \draw[white] (-1.1,-4) rectangle (5.1,2);
      \foreach \nn in {1,...,5}{\node[nullvertex] (\nn) at (45+\nn*45:1) {}; \draw[thick] (0,0) -- (\nn);}
      \node[cir] at (0,0)[label=below:$S_{n-7}$,scale=0.8] {};
      \node[vertex,fill=black,scale=.8] (6) at (0,0) {};
      \node[vertex,fill=black,scale=.8] (7) at (2,1) {};
      \node[vertex,fill=black,scale=.8] (8) at (3,1) {};
      \node[vertex,fill=black,scale=.8] (9) at (2,0) {};
      \node[vertex,fill=black,scale=.8] (10) at (3,0) {};
      \node[vertex,fill=black,scale=.8] (11) at (2,-1) {};
      \node[vertex,fill=black,scale=.8] (12) at (3,-1) {};
      \node[vertex,fill=black,scale=.8] (13) at (4,-1) {};
      \draw[thick](10)--(9)--(6)--(7)--(8)
                  (6)--(11)--(12)--(13);
      \node () at (1,-3.5) {$T_G(n)$};
    \end{scope}
    \begin{scope}[shift={(9,-6.5)}]
      \draw[white] (-1.1,-4) rectangle (5.1,2);
      \foreach \nn in {1,...,5}{\node[nullvertex] (\nn) at (45+\nn*45:1) {}; \draw[thick] (0,0) -- (\nn);}
      \node[cir] at (0,0)[label=below:$S_{n-6}$,scale=0.8] {};
      \node[vertex,fill=black,scale=.8] (6) at (0,0) {};
      \node[vertex,fill=black,scale=.8] (7) at (2,-1) {};
      \node[vertex,fill=black,scale=.8] (8) at (3,-1) {};
      \node[vertex,fill=black,scale=.8] (9) at (2,1) {};
      \node[vertex,fill=black,scale=.8] (10) at (3,1) {};
      \node[vertex,fill=black,scale=.8] (11) at (4,1) {};
      \node[vertex,fill=black,scale=.8] (12) at (5,1) {};
      \draw[thick](8)--(7)--(6)--(9)--(10)--(11)--(12);
      \node () at (1,-3.5) {$T_H(n)$};
    \end{scope}
    \begin{scope}[shift={(18,-6.5)}]
      \draw[white] (-1.1,-4) rectangle (5.1,2);
      \foreach \nn in {1,...,5}{\node[nullvertex] (\nn) at (45+\nn*45:1) {}; \draw[thick] (0,0) -- (\nn);}
      \node[cir] at (0,0)[label=below:$S_{n-6}$,scale=0.8] {};
      \node[vertex,fill=black,scale=.8] (6) at (0,0) {};
      \node[vertex,fill=black,scale=.8] (7) at (2,0.5) {};
      \node[vertex,fill=black,scale=.8] (8) at (3,1) {};
      \node[vertex,fill=black,scale=.8] (9) at (4,1) {};
      \node[vertex,fill=black,scale=.8] (10) at (3,0) {};
      \node[vertex,fill=black,scale=.8] (11) at (2,-1) {};
      \node[vertex,fill=black,scale=.8] (12) at (3,-1) {};
      \draw[thick](12)--(11)--(6)--(7)--(8)--(9)
                  (7)--(10);
      \node () at (1,-3.5) {$T_J(n)$};
    \end{scope}
    \begin{scope}[shift={( 0,-13)}]
      \draw[white] (-1.1,-4) rectangle (5.1,2);
      \foreach \nn in {1,...,5}{\node[nullvertex] (\nn) at (45+\nn*45:1) {}; \draw[thick] (0,0) -- (\nn);}
      \node[cir] at (0,0)[label=below:$S_{n-4}$,scale=0.8] {};
      \node[vertex,fill=black,scale=.8] (6) at (0,0) {};
      \node[vertex,fill=black,scale=.8] (7) at (1,0) {};
      \node[vertex,fill=black,scale=.8] (8) at (2,0) {};
      \node[vertex,fill=black,scale=.8] (9) at (3,0) {};
      \node[vertex,fill=black,scale=.8] (10) at (4,1) {};
      \node[vertex,fill=black,scale=.8] (11) at (4,-1) {};
      \draw[thick](6)--(7)--(8)--(9)--(10)
                  (9)--(11);
      \node () at (1,-3.5) {$T_K(n)$};
    \end{scope}
    \begin{scope}[shift={(9,-13)}]
      \draw[white] (-1.1,-4) rectangle (5.1,2);
      \foreach \nn in {1,...,5}{\node[nullvertex] (\nn) at (45+\nn*45:1) {}; \draw[thick] (0,0) -- (\nn);}
      \node[cir] at (0,0)[label=below:$S_{n-5}$,scale=0.8] {};
      \node[vertex,fill=black,scale=.8] (6) at (0,0) {};
      \node[vertex,fill=black,scale=.8] (7) at (2,0) {};
      \node[vertex,fill=black,scale=.8] (8) at (3,-1) {};
      \node[vertex,fill=black,scale=.8] (9) at (3,1) {};
      \node[vertex,fill=black,scale=.8] (10) at (4,1) {};
      \node[vertex,fill=black,scale=.8] (11) at (5,1) {};
      \draw[thick](8)--(7)--(6)--(7)--(9)--(10)--(11);
      \node () at (1,-3.5) {$T_L(n)$};
    \end{scope}
    \begin{scope}[shift={( 18,-13)}]
      \draw[white] (-1.1,-4) rectangle (5.1,2);
      \foreach \nn in {1,...,5}{\node[nullvertex] (\nn) at (45+\nn*45:1) {}; \draw[thick] (0,0) -- (\nn);}
      \node[cir] at (0,0)[label=below:$S_{n-4}$,scale=0.8] {};
      \node[vertex,fill=black,scale=.8] (6) at (0,0) {};
      \node[vertex,fill=black,scale=.8] (7) at (1,0) {};
      \node[vertex,fill=black,scale=.8] (8) at (2,0) {};
      \node[vertex,fill=black,scale=.8] (9) at (2,1) {};
      \node[vertex,fill=black,scale=.8] (10) at (2,-1) {};
      \node[vertex,fill=black,scale=.8] (11) at (3,1) {};
      \draw[thick](7)--(8)
                  (7)--(10)
                  (6)--(7)--(9)--(11);
      \node () at (1,-3.5) {$T_M(n)$};
    \end{scope}
    \begin{scope}[shift={( 0,-19.5)}]
      \draw[white] (-1.1,-4) rectangle (5.1,2);
      \foreach \nn in {1,...,5}{\node[nullvertex] (\nn) at (45+\nn*45:1) {}; \draw[thick] (0,0) -- (\nn);}
      \node[cir] at (0,0)[label=below:$S_{n-4}$,scale=0.8] {};
      \node[vertex,fill=black,scale=.8] (6) at (0,0) {};
      \node[vertex,fill=black,scale=.8] (7) at (1,0) {};
      \node[vertex,fill=black,scale=.8] (8) at (2,1) {};
      \node[vertex,fill=black,scale=.8] (9) at (3,1) {};
      \node[vertex,fill=black,scale=.8] (10) at (2,-1) {};
      \node[vertex,fill=black,scale=.8] (11) at (3,-1) {};
      \draw[thick](11)--(10)--(7)--(8)--(9)
      (6)--(7);
      \node () at (1,-3.5) {$T_N(n)$};
    \end{scope}
    \begin{scope}[shift={(9,-19.5)}]
      \draw[white] (-1.1,-4) rectangle (5.1,2);
      \foreach \nn in {1,...,5}{\node[nullvertex] (\nn) at (45+\nn*45:1) {}; \draw[thick] (0,0) -- (\nn);}
      \node[cir] at (0,0)[label=below:$S_{n-4}$,scale=0.8] {};
      \node[vertex,fill=black,scale=.8] (6) at (0,0) {};
      \node[vertex,fill=black,scale=.8] (7) at (1,0) {};
      \node[vertex,fill=black,scale=.8] (8) at (2,1) {};
      \node[vertex,fill=black,scale=.8] (9) at (2,-1) {};
      \node[vertex,fill=black,scale=.8] (10) at (3,1.5) {};
      \node[vertex,fill=black,scale=.8] (11) at (3,0.5) {};
      \draw[thick](10)--(8)--(7)--(9)
                  (6)--(7)
                  (8)--(11);
      \node () at (1,-3.5) {$T_P(n)$};
    \end{scope}
    \begin{scope}[shift={( 18,-19.5)}]
      \draw[white] (-1.1,-4) rectangle (5.1,2);
      \foreach \nn in {1,...,5}{\node[nullvertex] (\nn) at (45+\nn*45:1) {}; \draw[thick] (0,0) -- (\nn);}
      \node[cir] at (0,0)[label=below:$S_{n-6}$,scale=0.8] {};
      \node[vertex,fill=black,scale=.8] (6) at (0,0) {};
      \node[vertex,fill=black,scale=.8] (7) at (2,1) {};
      \node[vertex,fill=black,scale=.8] (8) at (3,1) {};
      \node[vertex,fill=black,scale=.8] (9) at (2,-1) {};
      \node[vertex,fill=black,scale=.8] (10) at (3,-1.75) {};
      \node[vertex,fill=black,scale=.8] (11) at (3,-0.25) {};
      \node[vertex,fill=black,scale=.8] (12) at (3,-1) {};
      \draw[thick](8)--(7)--(6)--(9)--(10)
                  (9)--(11)
                  (9)--(12);
      \node () at (1,-3.5) {$T_Q(n)$};
    \end{scope}
    \begin{scope}[shift={( 4.5,-26)}]
      \draw[white] (-1.1,-4) rectangle (5.1,2);
      \foreach \nn in {1,...,5}{\node[nullvertex] (\nn) at (45+\nn*45:1) {}; \draw[thick] (0,0) -- (\nn);}
      \node[cir] at (0,0)[label=below:$S_{n-6}$,scale=0.8] {};
      \node[vertex,fill=black,scale=.8] (6) at (0,0) {};
      \node[vertex,fill=black,scale=.8] (7) at (2,1) {};
      \node[vertex,fill=black,scale=.8] (8) at (3,1) {};
      \node[vertex,fill=black,scale=.8] (9) at (4,1) {};
      \node[vertex,fill=black,scale=.8] (10) at (2,-1) {};
      \node[vertex,fill=black,scale=.8] (11) at (3,-1.5) {};
      \node[vertex,fill=black,scale=.8] (12) at (3,-0.5) {};
      \draw[thick](11)--(10)--(12)
                  (6)--(10)
                  (6)--(7)--(8)--(9);
      \node () at (1,-3.5) {$T_R(n)$};
    \end{scope}
    \begin{scope}[shift={(13.5,-26)}]
      \draw[white] (-1.1,-4) rectangle (5.1,2);
      \foreach \nn in {1,...,5}{\node[nullvertex] (\nn) at (45+\nn*45:1) {}; \draw[thick] (0,0) -- (\nn);}
      \node[cir] at (0,0)[label=below:$S_{n-6}$,scale=0.8] {};
      \node[vertex,fill=black,scale=.8] (6) at (0,0) {};
      \node[vertex,fill=black,scale=.8] (7) at (2,-1) {};
      \node[vertex,fill=black,scale=.8] (8) at (3,-1) {};
      \node[vertex,fill=black,scale=.8] (9) at (2,1) {};
      \node[vertex,fill=black,scale=.8] (10) at (3,1) {};
      \node[vertex,fill=black,scale=.8] (11) at (4,1.5) {};
      \node[vertex,fill=black,scale=.8] (12) at (4,0.5) {};
      \draw[thick](11)--(10)--(9)--(6)--(7)--(8)
                  (10)--(12);
      \node () at (1,-3.5) {$T_S(n)$};
    \end{scope}
    \end{tikzpicture}
  \caption{Tree graphs $T_n$ with $\Delta(T_n)=n-5$.}\label{fig:Tree graph Delta=n-5}
\end{figure}

The Ramsey numbers $R(T_n,W_8)$ for these $19$ tree graphs are determined as follows.
\begin{theorem}
\label{thm:R(Sn(1,4)-TDEFGHJKLM)}
If $n\geq 8$, 
then $R(T_n,W_8) = 2n-1$ for all 
\[
  T_n\in \{S_n(1,4), S_n(5), S_n[5],S_n(2,2), S_n(4,1), T_D(n), \ldots, T_S(n)\}
\]
except when $T_n\in\{T_E(8), T_F(8), S_n(1,4), S_n(2,2), T_D(n), T_N(n)\}$ and $n\equiv 0\pmod{4}$,
in which case $R(T_n,W_8) = 2n$.
\end{theorem}

A proof of this theorem is given in Section~\ref{sec:n-5proof}.

\section{Auxiliary results}
\label{sec:auxiliary_results}

\noindent
To prove the main theorems, 
the following auxiliary results will be used.
For any simple graph $G = (V,E)$, 
let $\delta(G)$ be the minimum degree of any vertex in~$G$, 
and 
let $\overline{G} = \big(V, \binom{V}{2}\backslash E\big)$ be the complement of~$G$.

\begin{lemma}{\rm\cite{Bo71}}
\label{lem:pancyclic}
Let $G$ be a graph of order $n$. 
If $\delta(G)\geq \frac{n}{2}$, 
then either $G$ contains $C_\ell$ for all $3\leq \ell\leq n$, 
or $n$ is even and $G=K_{\frac{n}{2},\frac{n}{2}}$.
\end{lemma}

\begin{lemma}{\rm\cite{ChLeZh15}}
\label{lem:tree}
Let $G$ be a graph with $\delta(G)\geq n-1$. 
Then $G$ contains all tree graphs of order $n$.
\end{lemma}

\begin{observation}\label{obs2}
If $G=H_1\cup H_2$ is the disjoint union of graphs $H_1$ and $H_2$, 
where $\overline{H_1}$ contains $S_5$ and $H_2$ is a graph of order at least $4$,
then $\overline{G}$ contains $W_8$.
\end{observation}
 
\begin{lemma}\label{lm4:R(Sn(1,2),W8)} 
Let $H_1$ be a graph whose complement $\overline{H_1}$ contains $S_4$, 
and let $H_2$ be a graph of order $m\ge 5$. 
If $G=H_1\cup H_2$, 
then either $\overline{G}$ contains $W_8$, or $H_2$ is $K_m$ or $K_m-e$, 
where $e$ is an edge in $K_m$.
\end{lemma}

\begin{proof} 
If $\overline{H_2}$ has at most one edge, 
then $H_2$ is the complete graph $K_m$ or the graph $K_m-e$ obtained from removing an edge $e$ from $K_m$. 
Suppose now that $\overline{H_2}$ has at least two edges.
Consider a star $S_4$ in $\overline{H_1}$ and let $v_0$ be its center and $v_1,v_2,v_3$ its leaves.
Note that each $v_i$ is adjacent to each $a\in V(H_2)$ in~$\overline{G}$.
Choose $5$ vertices $a,b,c,d,e \in V(H_2)$ such that either $ab$ and $cd$ are independent edges,
or $abc$ is a path, in $\overline{H_2}$. 
In both cases, $\overline{G}$ contains $W_8$ with hub $v_0$. 
In the former case, $v_1abv_2cdv_3ev_1$ forms the $C_8$ rim; 
in the latter, $v_1abcv_2dv_3ev_1$ forms the $C_8$ rim.
\end{proof}

The neighbourhood $N_G(v)$ of a vertex $v$ in $G$ is the set of vertices that are adjacent to $v$ in $G$
and $d_G(v)=|N_G(v)|$ is the degree of the vertex $v$. 
For $X, Y\subseteq V$, 
$G[X]$ is the subgraph induced by $X$ in $G$ and 
$E_G(X,Y)$ is the set of edges in $G$ with one endpoint in $X$ and the other in $Y$.
The following lemma provides sufficient conditions for a graph or its complement to contain $C_8$.

\begin{lemma}
\label{lem:1-1relation}
Suppose that $U=\{u_1,\ldots,u_4\}$ and $V=\{v_1,\ldots,v_4\}$ are two disjoint subsets of vertices 
of a graph $G$ for which $|N_{G[V\cup \{u\}]}(u)|\leq 1$ for each $u\in U$
and $|N_{G[U\cup \{v\}]}(v)|\leq 2$ for each $v\in V$.
Then $\overline{G}[U\cup V]$ contains $C_8$.  
\end{lemma}

\begin{proof}
Suppose that $N_{G[U\cup \{v\}]}(v)\le 1$ for each $v\in V$.
Then $\overline{G}[U\cup V]$ contains a subgraph obtained 
by removing a matching from $K_{4,4}$ and therefore contains $C_8$.
Suppose now that $N_{G[U\cup \{v_1\}]}(v_1) = \{u_1,u_2\}$, 
and assume without loss of generality that 
$v_3\notin N_{G[V\cup \{u_3\}]}(u_3)$ and $v_4\notin N_{G[V\cup \{u_4\}]}(u_4)$.
Neither $u_1$ nor $u_2$ is adjacent to $v_2$, $v_3$ or $v_4$, 
so $v_1u_3v_3u_1v_2u_2v_4u_4v_1$ forms~$C_8$ in $\overline{G}[U\cup V]$.
\end{proof}

\begin{lemma}{\rm\cite{Ja81}}
\label{lem:cycleinbipartite}
Let $G(u,v,k)$ be a simple bipartite graph with bipartition $U$ and $V$, 
where $|U|=u\geq 2$ and $|V|=v\geq k$, 
and where each vertex of $U$ has degree of at least $k$. 
If $u\leq k$ and $v\leq 2k-2$, 
then $G(u,v,k)$ contains a cycle of length $2u$.
\end{lemma}

\begin{corollary}
\label{cor:bipartite4-6}
Suppose that $U$ and $V$ are two disjoint subsets of vertices of a graph $G$ 
for which $|N_{G[V\cup \{u\}]}(u)|\leq 2$ for each $u\in U$. 
If $|U|\geq 4$ and $|V|\geq 6$,
then $\overline{G}[U\cup V]$ contains $C_8$.  
\end{corollary}

\begin{proof}
Since $|U|\geq 4$ and $|V|\geq 6$, 
we can choose any $4$ vertices from $U$ to form $U'$ and any $6$ vertices from $V$ to form~$V'$. 
We have that $N_{G[V'\cup \{u\}]}(u)\leq 2$ for each $u\in U'$. 
Then each vertex of $U'$ is adjacent to at least $4$ vertices of $V'$ in $\overline{G}$ 
and $\overline{G}[U'\cup V']$ must contain 
a graph with the properties of $G(4,6,4)$ in Lemma~\ref{lem:cycleinbipartite}.
Hence by that lemma, $\overline{G}[U\cup V]$ must contain $C_8$.    
\end{proof}

\begin{corollary}
\label{cor:bipartite4-8}
Suppose that $U$ and $V$ are two disjoint subsets of vertices of a graph $G$ 
for which $|N_{G[V\cup \{u\}]}(u)|\leq 3$ for each $u\in U$. 
If $|U|\geq 4$ and $|V|\geq 8$, 
then $\overline{G}[U\cup V]$ contains $C_8$.  
\end{corollary}

\begin{proof}
Since $|U|\geq 4$ and $|V|\geq 8$, 
we can choose any $4$ vertices from $U$ to form $U'$ and any $8$ vertices from $V$ to form $V'$. 
We have that $N_{G[V'\cup \{u\}]}(u)\leq 3$ for each $u\in U'$. 
Then each vertex of $U'$ is adjacent to at least $5$ vertices of $V'$ in $\overline{G}$ 
and $\overline{G}[U'\cup V']$ must contain a graph with 
the properties of $G(4,8,5)$ in Lemma~\ref{lem:cycleinbipartite}.
Hence by that lemma, $\overline{G}[U\cup V]$ must contain $C_8$. 
\end{proof}

\section{Proof of Theorem~\ref{thm:ABCDE}}
\label{sec:proofofn-4ABCDE}

The proof of Theorem~\ref{thm:ABCDE} is here proved as three theorems, 
the first of which is as follows.

\begin{theorem}
\label{thm:ABC(7)}
$R(T,W_8) = 13$ for each $T\in\{A,B,C\}$.
\end{theorem}

\begin{proof}
Note that $G=2K_6$ does not contain $A$, $B$ or $C$ and that $\overline{G}$ does not contain $W_8$. 
Therefore, $R(T,W_8)\geq 13$ for $T = A,B,C$.

Let $G$ be a graph of order $13$ whose complement $\overline{G}$ does not contain $W_8$. 
By Theorem~\ref{thm:R(Pn-Sn(1,2)-Sn(3)-Sn(2,1),W8)}, $G$ has a subgraph $T=S_7(2,1)$. 
Label $V(T)$ as in Figure~\ref{fig:S_7(2,1) subgraph G}.
Set $U=V(G)-V(T)$; then $|U|=6$. 

First suppose that $A\nsubseteq G$.
Then $v_1$ is not adjacent to $v_2$ or $v_6$,
and $v_2$ and $v_5$ are not adjacent.

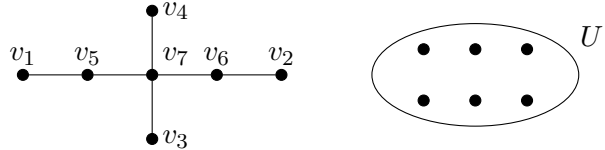
\begin{figure}[ht]
  \centering
  \begin{tikzpicture}[scale=.85]
    \begin{scope}[shift={(0,0)}]
      \foreach \nn/\x/\y in {1/0/1, 2/4/1, 3/2/0/, 4/2/2, 5/1/1, 6/3/1, 7/2/1}{
        \node [vertex,fill=black,scale=0.8] (\nn) at (\x,\y) {};}
      \draw (1) -- (5) -- (7) -- (6) -- (2)  (3) -- (7) -- (4);
      \foreach \nn in {1,5,6,2}{\draw (\nn) node[above](){$v_\nn$};}
      \foreach \nn in {3,4}{    \draw (\nn) node[right](){$v_\nn$};}
      \draw (7) node[above right](){$v_7$};      
    \end{scope}
    \begin{scope}[shift={(7,1)}]
      \foreach \x in {-.8,0,.8}{\foreach \y in {-.4,.4}{\draw (\x,\y) node[vertex,fill=black,scale=0.8] {};}}
      \draw (0,0) ellipse (1.6 and .8);
      \draw (1.8,.6) node {$U$};    
    \end{scope}  
  \end{tikzpicture}
  \caption{$S_7(2,1)$ and $U$ in~$G$.}
  \label{fig:S_7(2,1) subgraph G}
\end{figure}

\noindent {\bf Case 1a}: There is a vertex in $U$, say $u$, that is adjacent to $v_1$.

Since $A$ is not contained in $G$, 
$v_1$ is not adjacent to $v_3$, $v_4$ or any vertex of $U$ other than~$u$. 
Let $W=\{v_2,v_3,v_4,v_6,u_1,\ldots,u_4\}$ for any $4$ vertices $u_1,\ldots,u_4$ in $U$ other than~$u$.
If $\delta(\overline{G}[W])\geq 4$, 
then $\overline{G}[W]$ contains $C_8$ by Lemma~\ref{lem:pancyclic} and, 
together with $v_1$ as hub, forms $W_8$, a contradiction. 
Thus, $\delta(\overline{G}[W])\leq 3$ and $\Delta(G[W])\geq 4$. 
Note that $|N_{G[\{u_1,\ldots,u_4,v_i\}]}(v_i)|\leq 1$ for $i=2,3,4,6$ since $G$ does not contain~$A$. 
It is now straightforward to check that $v_2$, $v_3$, $v_4$ and $v_6$ cannot be the vertex with degree at least $4$. 
Without loss of generality, assume that $u_1$ has degree at least $4$ in $G[W]$.
Then $u_1$ is adjacent to at least one of $v_2,v_3,v_4,v_6$, 
so $G$ contains $A$, a contradiction. 

\smallskip
\noindent {\bf Case 1b}: $v_1$ is not adjacent to any vertices in $U$.

By arguments similar to those in Case~1a, $v_2$ is not adjacent to any vertex in $U$.
Let $W = \{v_2,v_6\}\cup U$.
If $\delta(\overline{G}[W])\geq 4$, 
then $\overline{G}[W]$ contains $C_8$ by Lemma~\ref{lem:pancyclic} which, 
with $v_1$ as hub, forms $W_8$ in $\overline{G}[W]$, 
a contradiction. 
Thus, $\delta(\overline{G}[W])\leq 3$ and $\Delta(G[W])\geq 4$. 
Since $v_2$ is not adjacent to any vertex in~$U$, 
there are only three subcases to be considered.

\smallskip
\noindent {\bf Subcase 1b.1}: $d_{G[W]}(v_6)\geq 4$. 

Label $U = \{u_1,\ldots,u_6\}$ so that $v_6$ is adjacent to $u_1$, $u_2$ and $u_3$ in $G[W]$. 
Since $G$ does not contain $A$, vertices $u_1,u_2,u_3,v_2$ are not adjacent to $v_3$ or $v_4$ in~$G$.
Note that by arguments as in Case 1a, $u_1$, $u_2$ and $u_3$ are isolated vertices in $G[U]$. 
Then $v_1u_4u_2v_3v_2u_5u_3u_6v_1$ and $u_1$ form $W_8$ in $\overline{G}$, a contradiction. 

\smallskip
\noindent {\bf Subcase 1b.2}: $d_{G[W]}(v_6)\leq 3$ and $v_6$ is adjacent to a vertex $u\in U$ with $d_{G[W]}(u)\geq 4$.

The graph $G$ contains $A$, with $u$ as the vertex of degree $3$ in $A$, a contradiction.

\smallskip
\noindent {\bf Subcase 1b.3}: $d_{G[W]}(v_6)\leq 3$ and $v_6$ is not adjacent to any vertex $u\in U$ with $d_{G[W]}(u)\geq 4$.

Label $V(U) = \{u_1,\ldots,u_6\}$ so that $u_6$ is adjacent to $u_2$, $u_3$, $u_4$ and $u_5$ in~$G$. 
Since $A\nsubseteq G$, none of $v_1,\ldots,v_7$ is adjacent in $G$ to any of $u_2,\ldots,u_5$. 
If $v_1$ is not adjacent in $G$ to any two of the vertices $v_3,v_4,v_7$, 
then $\overline{G}$ contains $W_8$ by Observation~\ref{obs2}, a contradiction. 
Therefore, $N_{G[v_3,v_4,v_7]}(v_1)\geq 2$ and, 
similarly, $N_{G[v_3,v_4,v_7]}(v_2)\geq 2$. 
Hence, one of $v_3,v_4,v_7$ is adjacent in $G$ to both $v_1$ and $v_2$. 
If $v_3$ or $v_4$ is adjacent to both $v_1$ and $v_2$, 
then $G$ contains $A$, with $v_7$ as vertex of degree $3$, a contradiction.
Finally, if both $v_1$ and $v_2$ are adjacent in $G$ to $v_7$ 
and each of them is adjacent to a different vertex in $v_3$ and $v_4$, 
then $G$ also contains $A$, where either $v_1$ or $v_2$ is the vertex of degree $3$, a contradiction.

Therefore, $R(A,W_8)\leq 13$, so $R(A,W_8) = 13$. 

Now, suppose that $B\nsubseteq G$.
Then $v_1,v_2,v_5,v_6$ are not adjacent to $v_3$ or $v_4$ in $G$, 
and $v_1$ and $v_2$ are not adjacent to $U$ in~$G$. 
Label the vertices $U=\{u_1,\ldots,u_6\}$ and let $W=\{v_3,v_4\}\cup U$. 
If $\delta(\overline{G}[W])\geq 4$, 
then $\overline{G}[W]$ contains $C_8$ by Lemma~\ref{lem:pancyclic} 
which, with $v_1$ as hub, forms $W_8$, a contradiction. 
Therefore, $\delta(\overline{G}[W])\leq 3$ and $\Delta(G[W])\geq 4$. 
If $v_3$ or $v_4$ is adjacent to the vertex of degree at least $4$ in $G[W]$, 
then $B$ is contained in $G$, with $v_7$ as the vertex of degree $3$. 
Hence, only two cases need to be considered. 

\smallskip
\noindent {\bf Case 2a}: $v_3$ or $v_4$ is the vertex of degree at least $4$ in $G[W]$.

Without loss of generality, assume that $v_3$ is the vertex of degree at least $4$ in $G[W]$. 
As previously shown, $v_3$ is not adjacent to $v_4$.
Therefore, it may be assumed that $v_3$ is adjacent to $u_1$, $u_2$, $u_3$ and $u_4$ in~$G$. 
Since $B\nsubseteq G$, $u_1,\ldots,u_4$ are independent in $G$ and are not adjacent to $\{v_1,v_2,v_4,v_5,v_6\}$.
Also, $v_1$ is not adjacent to $v_6$ and $v_2$ is not adjacent to $v_5$. 
Then $v_1v_6u_2v_2v_5u_3v_4u_4v_1$ and $u_1$ form $W_8$ in $\overline{G}$, a contradiction.

\smallskip

\noindent {\bf Case 2b}: One of the vertices in $U$, say $u_1$, is the vertex of degree at least $4$ in $G[W]$.

As above, $u_1$ is not adjacent to $v_3$ or $v_4$ in~$G$. 
It may then be assumed that $u_1$ is adjacent to $u_2$, $u_3$, $u_4$ and $u_5$. 
Since $B\nsubseteq G$, $v_1,\ldots,v_7$ are not adjacent to $\{u_2,\ldots,u_5\}$. 
Note that $v_3$ is not adjacent to $\{v_1,v_2,v_5,v_6\}$.
By Observation~\ref{obs2}, $\overline{G}$ contains $W_8$, a contradiction.

\smallskip

Therefore, $R(B,W_8)\leq 13$. 

Lastly, suppose that $C\nsubseteq G$.
Then $v_5$ and $v_6$ are not adjacent in $G$ to each other or to $v_3$, $v_4$ or $U$. 
Furthermore, $v_5$ is not adjacent to $v_2$ and $v_6$ is not adjacent to $v_1$. 
Label the vertices $U=\{u_1,\ldots,u_6\}$ and let $W=\{v_3,v_4,v_6,u_1,\ldots,u_5\}$. 
If $\delta(\overline{G}[W])\geq 4$, 
then $\overline{G}[W]$ contains $C_8$ by Lemma~\ref{lem:pancyclic} which, 
with $v_5$ as hub, forms $W_8$, a contradiction. 
Then $\delta(\overline{G}[W])\leq 3$ and $\Delta(G[W])\geq 4$. 
Since $v_6$ is not adjacent to $v_3$, $v_4$ or $U$, 
$v_6$ is not the vertex of degree at least $4$ in $G[W]$ and is not adjacent to that vertex. 
Note that if $v_3$ or $v_4$ is the vertex of degree $4$, 
then $G$ contains $C$, with $v_3$ or $v_4$ and $v_7$ as the vertices of degree $3$. 
Thus, one of the vertices in $U$, say $u_1$, is the vertex of degree at least $4$ in $G[W]$. 
Now, consider the following three cases. 
\smallskip

\noindent {\bf Case 3a}: Both $v_3$ and $v_4$ are adjacent to $u_1$ in $G[W]$. 

Suppose that $u_1$ is also adjacent to $u_2$ and $u_3$ in $G[W]$. 
Since $C\nsubseteq G$, $v_3$ is not adjacent in $G$ to $v_4$ 
and neither $v_3$ nor $v_4$ is adjacent to $\{v_1,v_2,v_5,v_6,u_2,\ldots,u_6\}$. 
Note that $|N_{G[\{v_1,v_2,u_i\}]}(u_i)|\leq 1$ for $i=2,3$ since $C\nsubseteq G$. 
If $v_1$ is adjacent to $u_2$ and $u_3$ in $\overline{G}$, 
then $v_1u_2v_5u_4v_3u_5v_6u_3v_1$ and $v_4$ form $W_8$ in $\overline{G}$, a contradiction. 
Therefore, $v_1$ is adjacent in $G$ to at least one of $u_2$ and $u_3$. 
Similarly, $v_2$ is adjacent to at least one of $u_2$ and $u_3$. 
Since $|N_{G[\{v_1,v_2,u_i\}]}(u_i)|\leq 1$ for $i=2,3$, 
$v_1$ is adjacent to $u_2$ and $v_2$ is adjacent to $u_3$, or vice versa. 
Then neither $u_2$ nor $u_3$ is adjacent in $G$ to $u_4, u_5, u_6$, since $C\nsubseteq G$. 
Therefore, $v_1v_3v_2v_5u_2u_4u_3v_6v_1$ and $v_4$ form $W_8$ in $\overline{G}$, a contradiction.

\smallskip

\noindent {\bf Case 3b}: One of $v_3$ and $v_4$, say $v_3$, is adjacent to $u_1$ in $G[W]$.

Suppose that $u_1$ is adjacent to $u_2$, $u_3$ and $u_4$ in $G[W]$. 
Then $v_1, v_2, v_4, v_5, v_6, u_2, u_3, u_4\notin N_G(v_3)$ 
and $|N_{G[\{v_4,u_2,u_3,u_4\}]}(v_4)|\leq 1$. 
Without loss of generality, assume that $v_4$ is not adjacent to $u_2$ or $u_3$ in~$G$. 
Now, suppose that $v_4$ is adjacent to $u_4$ in~$G$. 
Since $C\nsubseteq G$, $u_4$ is not adjacent to $v_1$ or $v_2$ in~$G$. 
Then $v_1u_4v_2v_5u_2v_4u_3v_6v_1$ and $v_3$ form $W_8$ in $\overline{G}$, a contradiction. 
Otherwise, suppose that $v_4$ is not adjacent to $u_4$ in~$G$. 
Then, $|N_{G[\{u_i,v_1,v_2\}]}(u_i)|\leq 1$ for $i=2,3,4$ and at least two of $u_2$, $u_3$ and $u_4$ is not adjacent to $v_1$ or $v_2$ in~$G$. 
Without loss of generality, assume that $u_2$ and $u_3$ are not adjacent to $v_1$ in~$G$. 
In this case, $v_1u_2v_4u_4v_5u_5v_6u_3v_1$ and $v_3$ form $W_8$ in $\overline{G}$, again a contradiction. 

\smallskip

\noindent {\bf Case 3c}: $v_3$ and $v_4$ are both not adjacent in $G[W]$ to $u_1$. 

Assume that $u_1$ is adjacent to each of $u_2,\ldots,u_5$ in $G[W]$. 
Since $C\nsubseteq G$, $|N_{G[\{v_1,\ldots,v_7,u_i\}]}(u_i)|\leq 1$ for $i=2,\ldots,5$,
and $|N_{G[\{u_2,\ldots,u_5,v_j\}]}(v_j)|\leq 1$ for $j=3,4$. 
Since $|N_{G[\{v_1,v_2,u_i\}]}(u_i)|\leq 1$ for $i=2,\ldots,5$, 
one of $v_1$ and $v_2$, say $v_1$, satisfies $|N_{G[\{u_2,\ldots,u_5,v_1\}]}(v_1)|\leq 2$. 
By Lemma~\ref{lem:1-1relation}, 
$\overline{G}[v_1,v_3,v_4,v_5,u_2,\ldots,u_5]$ contains $C_8$ which, 
with hub $v_6$, forms $W_8$ in $\overline{G}$. 

\smallskip

Therefore, $R(C,W_8)\leq 13$. 
This completes the proof of the theorem. 
\end{proof}

\begin{theorem}\label{thm:R(D,W_8)}
$R(D,W_8)=14$.
\end{theorem}

\begin{proof}
Let $G=K_6\cup H$ where $H$ is the graph shown in Figure~\ref{fig:S_7(3)}.

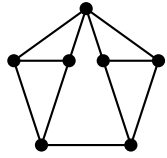
\begin{figure}[ht!]
  \centering
    \begin{tikzpicture}[scale=1,thick]
      \foreach \nn in {1,...,5}{\node[vertex, fill=black,scale=0.8] (\nn) at (72*\nn+90:1) {};}
        \node[vertex,fill=black,scale=.8] (6) at (126:.382) {};
        \node[vertex,fill=black,scale=.8] (7) at ( 54:.382) {};
        \draw (1)--(2)--(3)--(4)--(5)--(1)--(6)--(5) (2)--(6) (3)--(7)--(5) (4)--(7);
    \end{tikzpicture}
  \caption{The graph $H$.}
  \label{fig:S_7(3)}
\end{figure}

Since $G$ does not contain $D$ and $\overline{G}$ does not contain $W_8$, 
$R(D,W_8)\geq 14$.

Now, let $G$ be any graph of order $14$. 
Suppose neither $G$ contains $D$ as a subgraph, nor $\overline{G}$ contains $W_8$ as a subgraph. 
By Theorem~\ref{thm:ABC(7)}, $B\subseteq G$. 
Label the vertices of $B$ as shown in Figure~\ref{fig:B subgraph G} and set $U=\{u_1,\ldots,u_7\}=V(G)-V(B)$. 
Since $D\nsubseteq G$, $v_7$ is non-adjacent to $v_6$ and $U$, 
and $v_4$ is non-adjacent to $v_1$ and $v_2$. 

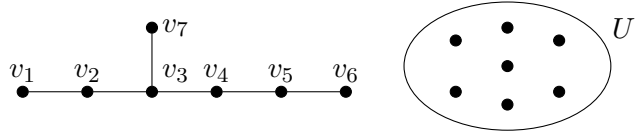
\begin{figure}[ht]
  \centering
  \begin{tikzpicture}[scale=.85]
    \begin{scope}[shift={(0,0)}]
      \foreach \nn/\x/\y in {1/-3/1, 2/-2/1, 3/-1/1/, 4/0/1, 5/1/1, 6/2/1, 7/-1/2}{
        \node [vertex,fill=black,scale=0.8] (\nn) at (\x,\y){};}
      \draw (1) -- (2) -- (3) -- (4) -- (5) -- (6)  (7) -- (3);
      \foreach \nn in {1,2,4,5,6}{\draw (\nn) node[above](){$v_\nn$};}
      \draw (3) node[above right](){$v_3$};
      \draw (7) node[      right](){$v_7$};
    \end{scope}
    \begin{scope}[shift={(4.5,1.4)}]
      \foreach \x in {-.8,.8}{\foreach \y in {-.4,.4}{\draw (\x,\y) node[vertex,fill=black,scale=0.8] {};}}
      \foreach \x in {0}{\foreach \y in {-.6,0,.6}{\draw (\x,\y) node[vertex,fill=black,scale=0.8] {};}}
      \draw (0,0) ellipse (1.6 and 1);
      \draw (1.8,.6) node {$U$};    
    \end{scope}  
  \end{tikzpicture}
  \caption{$B\subseteq G$.}
  \label{fig:B subgraph G}
\end{figure}

Let $W=\{v_6\}\cup U$.
If $\delta(\overline{G}[W])\geq 4$, 
then $\overline{G}[W]$ contains $C_8$ by Lemma~\ref{lem:pancyclic} which, with $v_7$ as hub, forms $W_8$, a contradiction. 
Thus, $\delta(\overline{G}[W])\leq 3$ and $\Delta(G[W])\geq 4$. 
Three cases will now be considered. 

\smallskip
\noindent {\bf Case 1}: $v_6$ is the vertex of degree at least $4$ in $G[W]$. 

Assume that $v_6$ is adjacent to $u_1$, $u_2$, $u_3$ and $u_4$ in $G[W]$. 
Then $v_5$ is adjacent to $v_1$ and $v_2$ in $\overline{G}$ 
 and $v_3$ is adjacent in $\overline{G}$ to $v_6$, $u_1$, $u_2$, $u_3$ and $u_4$. 

\smallskip
\noindent {\bf Subcase 1.1}: $E_G(\{u_1,\ldots,u_4\},\{u_5,u_6,u_7\})\neq \emptyset$. 

Without loss of generality, assume that $u_1$ is adjacent to $u_5$ in~$G$. 
Since $D\nsubseteq G$, 
$\{u_2,u_3,u_4\}$ is independent in $G$ and is adjacent to $v_1$, $v_2$, $u_6$ and $u_7$ in $\overline{G}$;
$v_6$ is adjacent in $\overline{G}$ to $v_1$ and $v_2$; 
$v_4$ and $v_5$ are adjacent in $\overline{G}$ to $u_1$ and $u_5$;
and 
$v_3$ is adjacent in $\overline{G}$ to $u_5$.
If $v_4$ is adjacent to $u_2$ in $G$, 
then $v_5$ is adjacent in $\overline{G}$ to $u_3$ and $u_4$, 
so $v_1v_5v_2u_2u_6v_7u_7u_3v_1$ and $u_4$ form $W_8$ in $\overline{G}$, a contradiction. 
Thus, $v_4$ is adjacent to $u_2$ in $\overline{G}$, 
and $v_1v_4v_2u_4u_6v_7u_7u_3v_1$ and $u_2$ form $W_8$ in $\overline{G}$, again a contradiction.

\smallskip
\noindent {\bf Subcase 1.2}: $\{u_1,\ldots,u_4\}$ is not adjacent to $\{u_5,u_6,u_7\}$ in $G[W]$.

Suppose that $v_5$ is adjacent in $G$ to $v_7$; 
then $v_7$ is not adjacent to $v_1$ or $v_2$. 
If $|N_{G[\{u_1,\ldots,u_4,v_2\}]}(v_2)|\leq 2$, 
then $\overline{G}[u_1,\ldots,u_7,v_2]$ contains $C_8$ by Lemma~\ref{lem:1-1relation} which with $v_7$ forms $W_8$ in $\overline{G}$, a contradiction. 
Thus, $|N_{G[\{u_1,\ldots,u_4,v_2\}]}(v_2)|\geq 3$, 
so $v_1$ is not adjacent to $u_1,\ldots,u_4$ in~$G$. 
By Lemma~\ref{lem:1-1relation}, 
$\overline{G}[u_1,\ldots,u_7,v_1,v_7]$ contains $W_8$, a contradiction.

Hence, $v_5$ is not adjacent to $v_7$ in~$G$. 
If $|N_{G[\{u_1,\ldots,u_4,v_5\}]}(v_5)|\leq 2$, 
then $\overline{G}[u_1,\ldots,u_7,v_5]$ contains $C_8$ by Lemma~\ref{lem:1-1relation} 
which with $v_7$ forms $W_8$ in $\overline{G}$, a contradiction. 
Thus $|N_{G[\{u_1,\ldots,u_4,v_5\}]}(v_5)|\geq 3$, 
so $v_4$ is not adjacent to $\{u_1,\ldots,u_4\}$ in $G$, 
or else $G$ will contain $D$ with $v_4$ be the vertex of degree 3. 
By Lemma~\ref{lem:1-1relation}, $\overline{G}[u_1,\ldots,u_7,v_1]$ contains $C_8$. 
If $v_4$ is not adjacent to $v_7$ in $G$, 
then $\overline{G}$ contains $W_8$, a contradiction. 
Thus, $v_4$ is adjacent to $v_7$, 
and since $D\nsubseteq G$, $v_1$ is not adjacent to $v_7$. 
If $|N_{G[\{u_1,\ldots,u_4,v_1\}]}(v_1)|\leq 2$, 
then $\overline{G}[u_1,\ldots,u_7,v_1]$ contains $C_8$ by Lemma~\ref{lem:1-1relation} 
which with $v_7$ forms $W_8$, a contradiction. 
Thus, $|N_{G[\{u_1,\ldots,u_4,v_1\}]}(v_1)|\geq 3$, 
so $|N_{G[\{u_1,\ldots,u_4,v_1\}]}(v_1)\cap  N_{G[\{u_1,\ldots,u_4,v_5\}]}(v_5)|\geq 2$, 
and $G$ contains $D$ with $v_5$ as the vertex of degree $3$, a contradiction. 

\smallskip
\noindent {\bf Case 2}: $u_1$ is the vertex of degree at least $4$ in $G[W]$ and $v_6$ is adjacent to $u_1$.

Without loss of generality, suppose that $u_1$ is adjacent to $u_2$, $u_3$ and $u_4$ in $G[W]$. 
If $v_5$ is adjacent to $u_1$, then Case 1 applies with $v_6$ replaced by $u_1$.
Suppose then that $v_5$ is not adjacent to $u_1$.
Since $D\nsubseteq G$, $v_1$ and $v_2$ are not adjacent in $G$ to $v_4$, $v_5$ or $v_6$;
$v_3$ is not adjacent to $v_6, u_1, \ldots, u_4$;
and $v_4$ is not adjacent to $u_1,\ldots,u_4$. 

\smallskip
\noindent {\bf Subcase 2.1}: $E_G(\{u_2,u_3,u_4\},\{u_5,u_6,u_7\})\neq \emptyset$. 

Without loss of generality, assume that $u_2$ is adjacent to $u_5$ in~$G$. 
Then $u_3$ and $u_4$ are not adjacent to each other or to $v_1, v_2, u_6, u_7$. 
Also, $u_1$ is not adjacent to $v_1$ or $v_2$, 
and neither $u_2$ nor $u_5$ is adjacent to $v_3, v_4, v_5, v_6$.

Suppose that $v_7$ is adjacent to $v_4$ in~$G$. 
If $u_1$ is adjacent to $v_1$, $u_5$, $u_6$ or $u_7$, 
then Case 1 can be applied through a slight adjustment of the vertex labelings. 
Suppose that $u_1$ is not adjacent to any of these vertices.
Since $D\nsubseteq G$, $v_7$ is not adjacent to $v_1$. 
If $v_6$ is not adjacent to $u_6$,
then $v_1u_1u_5v_6u_6u_3u_7u_4v_1$ and $v_7$ form $W_8$ in $\overline{G}$, a contradiction. 
Similarly, $\overline{G}$ contains $W_8$ if $v_6$ is not adjacent to $u_7$, a contradiction.
Therefore, $v_6$ is adjacent to both $u_6$ and $u_7$ in~$G$. 
Since $D\nsubseteq G$, $u_6$ is not adjacent to $u_7$, 
and neither $u_6$ nor $u_7$ is adjacent to $u_2$. 
Then $v_1u_1u_5v_6u_2u_6u_7u_3v_1$ and $v_7$ form $W_8$ in $\overline{G}$, a contradiction. 

Suppose now that $v_7$ is not adjacent to $v_4$ in~$G$. 
If $v_7$ is adjacent to $v_5$, 
then $v_7$ is not adjacent to $v_1$ or $v_2$, 
and $v_4$ is not adjacent to $v_6$, $u_6$ or $u_7$. 
Then $v_1u_1v_2u_3u_6v_4u_7u_4v_1$ and $v_7$ form $W_8$ in $\overline{G}$, a contradiction. 
Therefore, $v_7$ is not adjacent to $v_5$ in~$G$. 
If $v_6$ is not adjacent to $u_3$, 
then $u_3v_6u_2v_5u_5v_4u_4u_6u_3$ and $v_7$ form $W_8$ in $\overline{G}$, a contradiction.  
Similarly, $\overline{G}$ contains $W_8$ if $v_6$ is not adjacent to $u_4$, a contradiction.
Then $v_6$ is adjacent to both $u_3$ and $u_4$ in $G$, 
so $v_6$ is not adjacent to $u_6$ and $u_7$, or else Case 1 applies. 
Hence, $v_4u_2v_5u_5v_6u_6u_3u_4v_4$ and $v_7$ form $W_8$ in $\overline{G}$, a contradiction.  

\smallskip

\noindent {\bf Subcase 2.2}: $\{u_2,u_3,u_4\}$ is not adjacent to $\{u_5,u_6,u_7\}$ in $G[W]$.

If $|N_{G[\{u_2,u_3,u_4,v_6\}]}(v_6)|\geq 3$ 
or $|N_{G[\{u_5,u_6,u_7,v_6\}]}(v_6)|\geq 3$, then Case 1 applies, 
so $|N_{G[\{u_2,u_3,u_4,v_6\}]}(v_6)|$ $\leq 2$ 
and $|N_{G[\{u_5,u_6,u_7,v_6\}]}(v_6)|\leq 2$. 
Without loss of generality, 
assume that $v_6$ is not adjacent in $G$ to $u_2$ or $u_5$. 

Suppose that $v_4$ is not adjacent to $v_7$ in~$G$. 
If $u_5$ is adjacent to $u_6$ or $u_7$, say $u_6$, 
then $v_4$ is not adjacent to $u_5$ or $u_6$,
so $v_4u_2v_6u_5u_3u_7u_4u_6v_4$ and $v_7$ form $W_8$ in $\overline{G}$, a contradiction. 
If $u_5$ is not adjacent to $u_6$ or $u_7$, 
then $v_4u_2v_6u_5u_6u_3u_7u_4v_4$ and $v_7$ form $W_8$ in $\overline{G}$, a contradiction. 
Suppose that $v_4$ is adjacent to $v_7$ in~$G$. 
By similar arguments to those in Subcase 2.1, 
$u_1$ is not adjacent to $v_1$, $u_5$, $u_6$ or $u_7$, 
and $v_7$ is not adjacent to $v_1$. 
Then $v_1v_6u_5u_2u_6u_3u_7u_1v_1$ and $v_7$ form $W_8$ in $\overline{G}$, a contradiction. 

\smallskip

\noindent {\bf Case 3}: $u_1$ is the vertex of degree at least $4$ in $G[W]$ and $v_6$ is not adjacent to $u_1$.

Assume that $u_1$ is adjacent to $u_2$, $u_3$, $u_4$ and $u_5$ in $G[W]$. 
Since $D\nsubseteq G$, $v_3$ and $v_4$ are not adjacent to $u_1$, $u_2$, $u_3$, $u_4$ or $u_5$ in~$G$. 
If either $v_1$ or $v_5$ are adjacent to $u_1$ in $G$, then Case 1 applies, 
so suppose that $v_1$ and $v_5$ are not adjacent to $u_1$.
In addition, $v_1$ and $v_5$ is not adjacent to $u_2$, $u_3$, $u_4$ or $u_5$ in $G$, 
or else Case 2 applies. 

\noindent {\bf Subcase 3.1}: $N_{G[u_2,\ldots,u_5]}(v_6)\neq \emptyset$.

Assume that $v_6$ is adjacent to $u_2$ in~$G$. 
Note that $v_4$ is not adjacent to $v_6$, $v_7$, $u_6$ or $u_7$ in $G$, 
and $v_3$ is not adjacent to $v_5$ in $G$, 
or else Case 2 applies by slight adjustment of vertex labels. 
Since $D\nsubseteq G$, $v_1$ and $v_2$ are not adjacent in $G$ to $v_5$, $v_6$ or $u_2$, 
and $v_3$ is not adjacent to $v_6$ in~$G$. 

If $u_2$ and $u_6$ are not adjacent in $G$, 
then $v_1u_1v_6v_2u_2u_6v_7u_3v_1$ and $v_4$ form $W_8$ in $\overline{G}$, a contradiction. 
A similar contradiction arises if $u_2$ and $u_7$ not adjacent.
Therefore, $u_2$ is adjacent to both $u_6$ and $u_7$ in $G$, 
and $u_3$, $u_4$ and $u_5$ are not adjacent to $u_6$ or $u_7$ in $G$ since $D\nsubseteq G$. 
Then $v_1u_1v_6v_2u_2v_7u_6u_3v_1$ and $v_4$ form $W_8$ in $\overline{G}$, a contradiction. 

\smallskip

\noindent {\bf Subcase 3.2}: $N_{G[u_2,\ldots,u_5]}(v_6)=\emptyset$.

Suppose that $v_1$ is adjacent to $v_7$ in~$G$. 
Then $v_2$ is not adjacent to $v_5$, $v_6$ or $U$ since $D\nsubseteq G$. 
If $|N_{G[\{u_2,\ldots,u_6\}]}(u_6)|\leq 2$, 
then Lemma~\ref{lem:1-1relation} implies that $\overline{G}[u_2,\ldots,u_5,v_4,v_5,v_6,u_6]$
contains $C_8$ in $\overline{G}$ which with $v_2$ forms $W_8$, a contradiction. 
Thus, $|N_{G[\{u_2,\ldots,u_6\}]}(u_6)|\geq 3$. 
Similarly, $|N_{G[\{u_2,\ldots,u_5,u_7\}]}(u_7)|\geq 3$. 
By the Inclusion-exclusion Principle, 
$|N_{G[\{u_2,\ldots,u_6\}]}(u_6)\cap N_{G[\{u_2,\ldots,u_5,u_7\}]}(u_7)|\geq 2$. 
Without loss of generality, 
    $u_6$ is adjacent to $u_2$, $u_3$ and $u_4$ in $G$, 
and $u_7$ is adjacent to $u_3$ and $u_4$, 
and $G[u_1,\ldots,u_7]$ contains $D$ with $u_3$ or $u_4$ being the vertex of degree 3, a contradiction. 

Now suppose that $v_1$ is not adjacent to $v_7$ in~$G$. 
If $v_7$ is adjacent to $v_4$ in $G$, 
then $v_2$ is not adjacent to any of $u_1, \ldots, u_5$ in $G$, 
or else either Case 1 or~2 applies. 
Also, $|N_{G[\{v_2,v_5,v_7\}]}(v_7)|\leq 1$ since $D\subseteq G$. 
Assume that $v_7$ is not adjacent to $v_2$ in~$G$. 
If $|N_{G[\{u_2,\ldots,u_6\}]}(u_6)|\leq 2$, 
then Lemma~\ref{lem:1-1relation} implies that $\overline{G}[u_2,\ldots,u_5,v_1,v_2,v_6,u_6]$
contains $C_8$ which with $v_7$ forms $W_8$, a contradiction. 
Thus, $|N_{G[\{u_2,\ldots,u_6\}]}(u_6)|\geq 3$. 
Similarly, $|N_{G[\{u_2,\ldots,u_5,u_7\}]}(u_7)|\geq 3$, 
so $|N_{G[\{u_2,\ldots,u_6\}]}(u_6)\cap N_{G[\{u_2,\ldots,u_5,u_7\}]}(u_7)|\geq 2$. 
By arguments similar to those in the previous paragraph, 
$G$ will contain a subgraph $D$, a contradiction.

\smallskip

Thus, $R(D,W_8)\leq 14$ which completes the proof of the theorem. 
\end{proof}

\begin{theorem}
$R(E,W_8)=15$.
\end{theorem}

\begin{proof}
The graph $G = K_6\cup K_{4,4}$ does not contain $E$
and $\overline{G}$ does not contain $W_8$. 
Thus, $R(E,W_8)\geq 15$.
For the upper bound, let $G$ be any graph of order $15$. 
Suppose that $G$ does not contain $E$ and that $\overline{G}$ does not contain $W_8$. 
By Theorem~\ref{thm:R(Pn-Sn(1,2)-Sn(3)-Sn(2,1),W8)}, $G$ contains a $T=S_7(3)$ subgraph. 
Label the vertices of this subgraph as in Figure~\ref{fig:S_7(3) subgraph G} 
and set $U=V(G)-V(T)$. 
Note that $|U|= 8$.

\begin{figure}[ht]
  \centering
  \begin{tikzpicture}[scale=.85]
    \begin{scope}[shift={(0,0)}]
      \foreach \nn/\x/\y in {1/1/1, 2/0/0, 3/1/-1/, 4/1/0, 5/2/0, 6/2.5/.86, 7/2.5/-.86}{%
        \node [vertex,fill=black,scale=0.8] (\nn) at (\x,\y) {};}
      \draw (2) -- (4) -- (5) -- (6) (1) -- (4) -- (3) (5) -- (7);
      \foreach \nn in {1,3,6,7}{\draw (\nn) node[right](){$v_\nn$};}
      \draw (2) node[above]()      {$v_2$};
      \draw (4) node[above right](){$v_4$};
      \draw (5) node[right]()      {$v_5$};
    \end{scope}
    \begin{scope}[shift={(6,0)}]
      \foreach \x in {-1.5,-.5,.5,1.5}{\foreach \y in {-.4,.4}{\draw (\x,\y) node[vertex,fill=black,scale=0.8] {};}}
      \draw (0,0) ellipse (2.4 and .8);
      \draw (2.65,-.6) node {$U$};
    \end{scope}
  \end{tikzpicture}
  \caption{$S_7(3)$ and $U$ in~$G$.}
  \label{fig:S_7(3) subgraph G}
\end{figure}
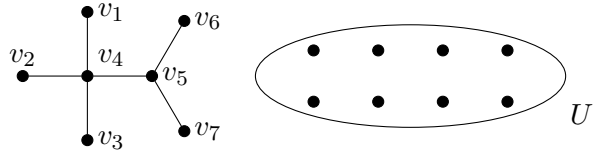

\noindent {\bf Case 1}: Some vertex $u$ in $U$ is adjacent to $v_6$.

Since $E\nsubseteq G$, $v_6$ is not adjacent to $v_1$, $v_2$, $v_3$, $v_7$ or any vertex of $U$ other than $u$. 
Let $W=\{v_1,v_2,v_3,v_7,u_1,\ldots,u_4\}$, 
for any vertices $u_1,\ldots,u_4$ in $U$ other than $u$.
If $\delta(\overline{G}[W])\geq 4$, 
then $\overline{G}[W]$ contains $C_8$ by Lemma~\ref{lem:pancyclic} which with $v_6$ forms $W_8$, 
a contradiction. 
Therefore, $\delta(\overline{G}[W])\leq 3$ and $\Delta(G[W])\geq 4$. 
Since $E\nsubseteq G$, 
$N_{G[\{u_1,\ldots,u_4,v_1,v_7\}]}(v_7)\leq 1$ 
and $N_{G[\{u_1,\ldots,u_4,v_7,v_i\}]}(v_i)\leq 1$ for $i=1,2,3$, 
so none of $v_1,v_2,v_3,v_7$ has degree at least $4$. 
Without loss of generality, assume that $u_1$ has degree at least $4$.
If $u_1$ is adjacent to $v_7$, 
then $G$ contains $E$ with $u_1$ and $v_5$ as the vertices of degree $3$, a contradiction. 
Similarly, if $u_1$ is adjacent to $v_1$, $v_2$ or $v_3$, 
then $G$ contains $E$ with $u_1$ and $v_4$ as the vertices of degree $3$, a contradiction. 
Therefore, $u_1$ is not adjacent to $v_1$, $v_2$, $v_3$ or $v_7$. 
However, then $u_1$ has degree at most $3$ in $G[W]$, a contradiction.

\smallskip

\noindent {\bf Case 2}: $v_6$ is not adjacent to any vertices in $U$.

If $v_7$ is adjacent to some vertex in $U$, 
then Case 1 applies with $v_7$ replacing $v_6$, 
so suppose that $v_7$ is not adjacent to any vertex in $U$.
Now, if $\delta(\overline{G}[U])\geq 4$, 
then $\overline{G}[U]$ contains $C_8$ by Lemma~\ref{lem:pancyclic} which with $v_6$ or $v_7$ forms $W_8$, 
a contradiction. 
Thus, $\delta(\overline{G}[U])\leq 3$ and $\Delta(G[U])\geq 4$. 
Let $V(U)=\{u_1,\ldots,u_8\}$. 
Without loss of generality, 
assume that $u_1$ is adjacent to $u_2$, $u_3$, $u_4$ and $u_5$. 
Since $E\nsubseteq G$, 
$v_4$ is not adjacent in $G$ to any of $u_1,\ldots,u_5$;
$v_5$ is not adjacent to any of $v_1, v_2, v_3, u_1, \ldots, u_5$; 
and $u_1$ is not adjacent to $v_1$, $v_2$ or $v_3$. 
Furthermore, 
$|N_{G[\{u_2,\ldots,u_5,v_i\}]}(v_i)|\leq 1$ for $i=1,2,3$ and 
$|N_{G[\{v_1,v_2,v_3,u_j\}]}(u_j)|\leq 1$  for $j=2,\ldots,5$. 

Suppose that $N_{G[\{v_5,u_6,u_7,u_8\}]}(v_5)=\emptyset$. 
If $|N_{G[\{u_2,\ldots,u_6\}]}(u_6)|\leq 1$, 
then $\overline{G}[u_2,\ldots,u_5,v_1,v_2,v_3,u_6]$ contains $C_8$ by Lemma~\ref{lem:1-1relation} which with $v_5$ forms $W_8$, a contradiction. 
Therefore, $|N_{G[\{u_2,\ldots,u_6\}]}(u_6)|\geq 2$. 
Similarly, 
$|N_{G[\{u_2,\ldots,u_5,u_7\}]}(u_7)|\geq 2$ and 
$|N_{G[\{u_2,\ldots,u_5,u_8\}]}(u_8)|\geq 2$. 
By the Inclusion-Exclusion Principle, 
$u_2$, $u_3$, $u_4$ or $u_5$ is adjacent in $G$ to at least two of $u_6,u_7,u_8$. 
Without loss of generality, 
assume that $u_2$ is adjacent to $u_6$ and $u_7$.
Then $u_2$ is not adjacent to $u_3$, $u_4$ or $u_5$,
Therefore, Lemma~\ref{lem:1-1relation} implies that 
$\overline{G}[u_1,u_3,u_4,u_5,v_1,v_2,v_3,u_2]$ contains $C_8$ which with $v_5$ forms $W_8$, a contradiction. 

On the other hand, if $N_{G[u_6,u_7,u_8]}(v_5)\neq \emptyset$, 
then without loss of generality assume that $u_6$ is adjacent to $v_5$ in~$G$. 
Since $E\nsubseteq G$, 
$v_4$ is not adjacent to $v_6$, $v_7$ or $u_6$ in~$G$. 
Also, $\{v_1,v_2,v_3\}$ and $\{v_6,v_7,u_6\}$ are independent in $G$, 
and $v_1, v_2, v_3, v_6, v_7, u_6\notin N_G(u_i)$ for $i=1,\ldots,5,7,8$, 
or else Case 1 applies with vertex label adjustments. 
Now, if $u_1$ is not adjacent to both $u_7$ and $u_8$ in $G$, 
then $v_1v_2v_3u_7v_6v_7u_6u_8v_1$ and $u_1$ form $W_8$ in $\overline{G}$, a contradiction. 
Therefore, $N_{G[\{u_1,u_7,u_8\}]}(u_1)\neq \emptyset$. 
Without loss of generality, assume that $u_1$ is adjacent to $u_7$ in~$G$. 
Note that for $E\nsubseteq G$, $|N_{G[\{v_4,v_5,u_8\}]}(u_8)|\leq 1$. 
Assume that $u_8$ is not adjacent to $v_4$ in~$G$. 
If $|N_{G[\{u_2,\ldots,u_5,u_8\}]}(u_8)|\leq 3$, 
then assume without loss of generality that $u_8$ is not adjacent to $u_2$ or $u_3$ in~$G$.
Then $v_6u_4v_7u_5u_6u_2u_8u_3v_6$ and $v_4$ form $W_8$ in $\overline{G}$, a contradiction. 
Similar arguments work if $u_8$ is not adjacent to $v_5$ in $G$, 
by replacing $v_4$ with $v_5$ and $v_6,v_7,u_6$ with $v_1,v_2,v_3$, respectively. 
Hence, $|N_{G[\{u_2,\ldots,u_5,u_7,u_8\}]}(u_8)|\geq 4$. 
However, $G$ then contains $E$ with $u_1$ and $u_8$ of degree $3$, a contradiction. 

Thus, $R(E,W_8)\leq 15$. 
This completes the proof of the theorem. 
\end{proof}

\section{Proof of Theorem~\ref{thm:R(Sn(4)-Sn[4]-Sn(1,3)-TA-TB-TC-Sn(3,1),W8)}}
\label{sec:proofofn-4}

Consider the tree graphs $T_n$ of order $n\geq 8$ with $\Delta(T_n) = n-4$, 
namely $S_n(4)$, $S_n[4]$, $S_n(1,3)$, $S_n(3,1)$, $T_A(n)$, $T_B(n)$ and $T_C(n)$; 
see Figures~\ref{fig:Tree example} and~\ref{fig:Tree graph Delta=n-4}. 

\begin{lemma}
\label{lem:lower-bound-on-S_n(4)-S_n[4]-S_n(1,3)-S_n(3,1)-T_A(n)-T_B(n)-T_C(n)}
Let $n\geq 8$. 
Then $R(T_n,W_8)\geq 2n-1$ for each $T_n \in\{S_n(4), S_n(3,1), T_C(n)\}$. 
Also for each $T_n \in \{S_n[4], S_n(1,3), T_A(n), T_B(n)\}$,
$R(T_n,W_8)\geq 2n-1$ if $n\not\equiv 0 \pmod{4}$
and 
$R(T_n,W_8)\geq 2n$ otherwise.
\end{lemma}

\begin{proof}
The graph $G=2K_{n-1}$ clearly does not contain any tree graph of order $n$,
and $\overline{G}$ does not contain $W_8$. 
Finally, if $n\equiv 0 \pmod{4}$, 
then the graph $G=K_{n-1}\cup K_{4,\ldots,4}$ of order $2n-1$ does not contain 
$S_n[4]$, $S_n(1,3)$, $T_A(n)$ or $T_B(n)$; 
nor does the complement $\overline{G}$ contain $W_8$. 
\end{proof}

\begin{theorem}
\label{thm:R(Sn(4),W8)}
If $n\geq 8$, then 
\[
  R(S_n(4),W_8)=\begin{cases}
    2n-1 & \text{if $n\geq 9$}\,;\\
    16   & \text{if $n = 8$}\,.
  \end{cases}
\]
\end{theorem}

\begin{proof}
By Lemma~\ref{lem:lower-bound-on-S_n(4)-S_n[4]-S_n(1,3)-S_n(3,1)-T_A(n)-T_B(n)-T_C(n)},
$R(S_n(4),W_8)\geq 2n-1$ for $n\ge 8$.
For $n = 8$, observe that the graph $G=K_7\cup H_8$, where $H_8$ is the graph of order $8$ as shown in Figure~\ref{fig:S_8(4)}
does not contain $S_8(4)$ and its complement $\overline{G}$ does not contain $W_8$.
Therefore, for $n=8$, we have a better bound of $R(S_8(4),W_8)\ge 16$.

\begin{figure}[ht!]
  \centering
    \begin{tikzpicture}[scale=1]
      \foreach \nn in {1,...,5}{\node[vertex, fill=black,scale=0.8] (\nn) at (72*\nn+90:1) {};}
        \draw[thick] (1)--(2)--(3)--(4)--(5)--(1)--(4)--(2)--(5)--(3)--(1);
      \foreach \a in {-1,0,1}{\node[vertex, fill=black,scale=0.8] () at (72*\a+272:.382) {};}
    \end{tikzpicture}
  \caption{The graphs $H_8$.}
  \label{fig:S_8(4)}
\end{figure}
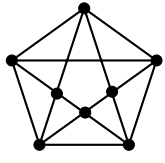

For the upper bound, let $G$ be any graph of order $2n-1$ if $n\ge 9$, and of order $16$ if $n=8$.
Assume that $G$ does not contain $S_n(4)$ and that $\overline{G}$ does not contain~$W_8$.

If $n\ge 9$ is odd or $n=8$, 
then $G$ has a subgraph $T=S_n(3)$ by Theorem~\ref{thm:R(Pn-Sn(1,2)-Sn(3)-Sn(2,1),W8)}.
Let $V(T)=\{v_0, \ldots, v_{n-3}, w_1, w_2\}$
and $E(T)=\{v_0v_1,\ldots,v_0v_{n-3},v_1w_1,v_1w_2\}$.
Also, let $V=\{v_2,\ldots,v_{n-3}\}$ and $U=V(G)-V(T)$;
then $|V|=n-4\ge 5$ and $|U|=n-1\ge 8$ if $n$ is odd, while $|U|=8$ if $n=8$.
Since $S_n(4)\nsubseteq G$,
$v_1$ is not adjacent in $G$ to any vertex of $U\cup V$ in~$G$.
Furthermore, for each $2\leq i\leq n-3$, 
$v_i$ is adjacent to at most two vertices of $U$ in $G$. 
By Corollary~\ref{cor:bipartite4-6}, $\overline{G}[U\cup V]$ contains $C_8$, and together with $v_1$, gives us $W_8$ in $\overline{G}$, a contradiction.

For the remaining case when $n\ge 10$ is even,
$S_{n-1}\subseteq G$ by Theorem~\ref{thm:R(Sn,W8)}.
Let $v_0$ be the center of $S_{n-1}$ 
and set $L=N_{S_{n-1}}(v_0)=\{v_1,\ldots,v_{n-2}\}$ and $U=V(G)-V(S_{n-1})$.
Then $|U|=n$.
Since $G$ does not contain $S_n(4)$, each vertex of $L$ is adjacent to at most two vertices of $U$.
We consider two cases.

\smallskip

\noindent {\bf Case 1}: $E(L,U)=\emptyset$.

If $\Delta(\overline{G}[U])\geq 4$,
then some vertex $u$ in $U$ is adjacent to at least four vertices in $\overline{G}[U]$.
These four vertices and any four vertices from $L$ form $C_8$ in $\overline{G}$ which with $u$ forms $W_8$, 
a contradiction.
Therefore, $\Delta(\overline{G}[U])\leq 3$ and $\delta(G[U])\geq n-4$.
Suppose $\delta(G[U])=n-4+l$ for some $l\ge 0$, 
and let $u_0$ be a vertex in $U$ with minimum degree in $G[U]$.
Label the remaining vertices in $U$ as $u_1,\ldots,u_{n-1}$ 
such that $U_A=\{u_1,\ldots,u_{n-4}\}\subseteq N_G(u_0)$, and let $U_B=\{u_{n-3},u_{n-2},u_{n-1}\}$.
Since $S_n(4)\nsubseteq G$, each vertex in $U_A$ is adjacent to at most two vertices in $U_B$, 
and so $|E_G(U_A,U_B)|\leq 2(n-4)$.
On the other hand, 
noting that $u_0$ is adjacent to exactly $l$ vertices in $U_B$ 
and letting $e_B\le 3$ be the number of edges in $G[U_B]$, 
we see that $|E_G(U_A,U_B)|\ge 3\delta(G[U]) - l - 2e_B = 3(n-4+l) - l - 2e_B$.
Therefore, $2(n-4)\ge |E_G(U_A,U_B)|\ge 3n - 12 + 2l - 2e_B$, 
implying that $n+2l\le 4+2e_B\le 10$, 
which is only possible when $n=10$, $l=0$, $e_B=3$, and $|E_G(U_A,U_B)| = 2(n-4) = 12$.
For such scenario where $n=10$, noting that $u_0$ was an arbitrary vertex with minimum degree in $G[U]$, 
it is straightforward to deduce that 
the only possible edge set of $G[U]$ (up to isomorphism) with $S_{10}(4)\nsubseteq G[U]$ is 
$\{u_0u_1,\ldots,u_0u_6\}\cup
 \{u_1u_7,\ldots,u_4u_7\}\cup
 \{u_1u_8,u_2u_8,u_5u_8,u_6u_8\}\cup
 \{u_3u_9,\ldots,u_6u_9\}\cup\{u_1u_2,u_3u_4,u_5u_6\} \cup \{u_1u_3,u_1u_5,u_3u_5\} 
\cup\{u_2u_4,u_2u_6, u_4u_6\}\cup \{u_7u_8,u_7u_9,u_8u_9\}$.
Observe now that $\overline{G}[U]$ contains $C_8$ which forms $W_8$ in $\overline{G}$ 
with any vertex in $L$ as hub, a contradiction.
\smallskip

\noindent {\bf Case 2}: $E(L,U)\neq \emptyset$.

Without loss of generality, assume that $v_1$ is adjacent to $u_1$ in~$G$.
Since $S_n(4)\nsubseteq G$, $v_1$ is adjacent to at most one vertex of $U\cup L\setminus\{u_1\}$ in~$G$.
Therefore, we can find a $4$-vertex set $V'\subseteq V\setminus \{v_1\}$
and an $8$-vertex set $U'\subseteq U\setminus \{u_1\}$ 
such that $v_1$ is not adjacent in $G$ to any vertex of $U'\cup V'$.
Note that each vertex of $V'$ is adjacent to at most two vertices of $U'$ in $G$, so $|E(V',U')|\leq 8$.
This implies that there are four vertices in $U'$ 
that are each adjacent in $G$ to at most one vertex of $V'$, 
and so $\overline{G}$ contains $C_8$ by Lemma~\ref{lem:1-1relation} which with $v_1$ forms $W_8$, 
a contradiction.

\smallskip

Thus, $R(S_n(4),W_8)\leq 2n-1$ when $n\ge 9$ and $R(S_n(4),W_8)\le 16$ when $n=8$.
This completes the proof of the theorem.
\end{proof}

\begin{lemma}\label{lm:R(Sn[4],W8)} 
Let $H$ be a graph of order $n\geq 8$ with minimum degree $\delta(H)\geq n-4$. 
Then either $H$ contains $S_n[4]$ and $T_A(n)$, 
or $n\equiv 0 \pmod{4}$ and $\overline H$ is the disjoint union of $\frac{n}{4}$ copies of $K_4$, 
i.e., $\overline H=\frac{n}{4} K_4$.
\end{lemma}

\begin{proof}
Let $V(H)=\{u_0,\ldots,u_{n-1}\}$.
First consider the case where $H$ has a vertex of degree at least $n-3$, 
say $u_0$,
and that $\{u_1,\ldots,u_{n-3}\}\subseteq N_H(u_0)$. 

Suppose $u_{n-2}$ is adjacent to $u_{n-1}$ in $H$. 
Since $\delta(H)\ge n-4$, 
$u_{n-2}$ is adjacent to at least $n-6\ge 2$ vertices of $\{u_1,\ldots, u_{n-3}\}$, 
say $u_1$ and $u_2$, and so $H$ contains $S_n[4]$.
Furthermore by the minimum degree condition, 
$u_1$ is adjacent to at least $n-7\ge 1$ vertices of $\{u_1,\ldots, u_{n-3}\}$, 
and so $H$ contains $T_A(n)$.

Suppose now that $u_{n-2}$ is not adjacent to $u_{n-1}$ in $H$. 
Then by the minimum degree condition, there is a vertex in $\{u_1,\ldots, u_{n-3}\}$, say $u_1$, 
that is adjacent to both $u_{n-2}$ and $u_{n-1}$.
The vertices $u_1$ and $u_{n-2}$ must also each be adjacent to a vertex of $\{u_2,\ldots, u_{n-3}\}$, 
and so $H$ contains both $S_n[4]$ and $T_A(n)$.

For the remaining case, suppose that $H$ is $(n-4)$-regular and that $N_H(u_0)=\{u_1,\ldots,u_{n-4}\}$.
Let $U=\{u_{n-3},u_{n-2},u_{n-1}\}$ and suppose that $H[U]$ has an edge, 
say $u_{n-3}u_{n-2}$. 
Since $u_{n-3}$ must be adjacent in $H$ to some vertex of $N_H(u_0)$, 
it follows that $H$ contains $S_n[4]$ if $u_{n-3}$ or $u_{n-2}$ is adjacent to $u_{n-1}$.
Suppose then that neither $u_{n-3}$ nor $u_{n-2}$ is adjacent to $u_{n-1}$. 
Then $u_{n-1}$ is adjacent to every vertex of $N_H(u_0)$. 
Note that $d_{H[N_H(u_0)\cup \{u_{n-3}\}]}(u_{n-3})=n-5$
and let $u$ be the vertex of $N_H(u_0)$ that is not adjacent in $H$ to $u_{n-3}$. 
Since $d_H(u)=n-4$, 
$u$ is adjacent in $H$ to some vertex in $N_H(u_{n-3})$, so $H$ contains $S_n[4]$. 
Also, note that $u_{n-3}$ is adjacent in $H$ to at least $n-6$ vertices of $N_H(u_0)$.
If $u_{n-1}$ is adjacent to some vertex of $N_{H[N_H(u_0)\cup \{u_{n-3}\}]}(u_{n-3})$, 
then $H$ contains $T_A(n)$.
Note that this will always happen for $n\geq 9$. 
For $n=8$, there is a case 
where $|N_{H[N_H(u_0)\cup \{u_{n-3}\}]}(u_{n-3})|=|N_{H[N_H(u_0)\cup \{u_{n-1}\}]}(u_{n-1})|=2$ 
and $N_{H[N_H(u_0)\cup \{u_{n-3}\}]}(u_{n-3})\cap N_{H[N_H(u_0)\cup \{u_{n-1}\}]}(u_{n-1})=\emptyset$,  
so $u_{n-1}$ is adjacent to $u_{n-3}$ and $u_{n-2}$, 
giving $T_A(n)$ in~$H$. 

Now, suppose that $H[U]$ contains no edge.
Then $U_1=U\cup \{u_0\}$ is an independent set in $H$.
Furthermore, $N_H(u)=\{u_1,\ldots, u_{n-4}\}$ for every $u\in U$, 
as every vertex has degree $n-4$.
Therefore, $\overline H[U_1]$ is a $K_4$ component in $\overline H$. 
Repeating the above proof for each vertex $u$ of $H$ shows that 
either $u$ is contained in a $K_4$ component of $\overline H$,
or $H$ contains both $S_n[4]$ or $T_A(n)$.
In other words, either $H$ contains both $S_n[4]$ and $T_A(n)$,
or $\overline H$ is the disjoint union of $\frac{n}{4}$ copies of $K_4$, 
and so $n\equiv 0 \pmod{4}$.
\end{proof}

\begin{theorem}\label{thm:R(Sn[4],W8)}
If $n\geq 8$, then 
\[
  R(S_n[4],W_8)=\begin{cases}
    2n-1 & \text{if $n\not\equiv 0 \pmod{4}$}\,;\\
    2n   & \text{otherwise}\,.
  \end{cases}
\]
\end{theorem}

\begin{proof}
Lemma~\ref{lem:lower-bound-on-S_n(4)-S_n[4]-S_n(1,3)-S_n(3,1)-T_A(n)-T_B(n)-T_C(n)}
provides the lower bounds, so it remains to prove the upper bounds.
Now let $G$ be a graph that does not contain $S_n[4]$ and assume that $\overline{G}$ does not contain $W_8$.

First suppose that $G$ has order $2n$ if $n\equiv 0 \pmod{4}$ and $G$ has order $2n-1$ if $n$ is odd.
By Theorem~\ref{thm:R(Pn-Sn(1,2)-Sn(3)-Sn(2,1),W8)}, 
$G$ has a subgraph $T=S_n(3)$.
Let $V(T)=\{v_0,\ldots,v_{n-3},w_1,w_2\}$
and $E(T)=\{v_0v_1,\ldots,v_0v_{n-3}\}\cup \{v_1w_1,v_1w_2\}$.
Set $U=V(G)-V(T)$ and $V=\{v_2,\ldots,v_{n-3}\}$.
Then $|U|=n-j$, for $j=0$ if $n\equiv 0 \pmod{4}$ and $j=1$ if $n$ is odd, 
and $|V|=n-4$.
Since $G$ does not contain $S_n[4]$,
$v_1$ is not adjacent to any vertex of $V$ in $G$, 
and each vertex of $V$ is adjacent to at most $n-6$ vertices of $U\cup V$ in~$G$.
Noting also that $w_1$ and $w_2$ each is adjacent to at most one vertex of $\{w_1,w_2\}\cup U$ in~$G$, 
we consider two cases.

\smallskip
\noindent {\bf Case 1}: 
At least one of $w_1$ and $w_2$ is not an isolated vertex in $G[\{w_1,w_2\}\cup U]$.

Without loss of generality, 
assume that $w_1$ is adjacent to some vertex $u\in \{w_2\}\cup U$ in~$G$.
Let $Z = \big(V\cup U\cup \{w_2\}\big)\setminus \{u\}$ 
and note that $|Z|=2n-4-j$.
Since $S_n[4]\nsubseteq G$, 
$w_1$ is not adjacent to any vertex of $Z$ in~$G$.
If $\delta(\overline{G}[Z])\geq \lceil\frac{2n-4-j}{2}\rceil$,
then $\overline{G}[Z]$ contains $C_8$ by Lemma~\ref{lem:pancyclic} 
which with $w_1$, forms $W_8$ in $\overline{G}$, a contradiction.
Therefore, $\delta(\overline{G}[Z])\leq \lceil\frac{2n-4-j}{2}\rceil-1$ 
and $\Delta(G[Z])\geq \lfloor\frac{2n-4-j}{2}\rfloor = n-2-j$.
Since each $v$ of $V$ is adjacent to at most $n-6$ vertices of $U\cup V$ in $G$,
and $w_2$ is adjacent to at most one vertex of $U$ in $G$, 
a vertex with maximum degree in $G[Z]$ must be a vertex of $U\setminus\{u\}$.
So let $u_2$ be a vertex of $U$ with $d_{G[Z]}(u_2)\geq n-2$.
As $S_n[4]\nsubseteq G$, observe that $N_{G[Z]}(u_2)\subseteq U$; 
each vertex of $V$ is adjacent to at most one vertex of $N_{G[Z]}(u_2)$ in $G$; 
and each vertex of $N_{G[Z]}(u_2)$ is adjacent to at most one vertex of $V$ in~$G$.
Then by Lemma $\ref{lem:1-1relation}$, 
any four vertices from $V$ and any four vertices from $N_{G[Z]}(u_2)$
form $C_8$ in $\overline{G}$ which with $w_1$ forms $W_8$ in $\overline{G}$, 
a contradiction.

\smallskip

\noindent {\bf Case 2}: 
$w_1$ and $w_2$ are isolated vertices in $G[\{w_1,w_2\}\cup U]$.

If $\delta(\overline{G}[U])\geq \frac{n-j}{2}$,
then $\overline{G}[U]$ contains $C_8$ by Lemma~\ref{lem:pancyclic} 
which with $w_1$ forms $W_8$, a contradiction.
Thus, $\delta(\overline{G}[U])\leq \frac{n-j}{2}-1$,
and $\Delta(G[U])\geq \frac{n-j}{2}$.
Let $u_1$ be a vertex of $U$ with $d_{G[U]}\geq \frac{n-j}{2}$.
Since $S_n[4]\nsubseteq G$, $v_0$ is not adjacent to any vertex of $N_{G[U]}(u_1)$ in~$G$.
Now, if $v_1$ is adjacent to some vertex $u$ of $N_{G[U]}(u_1)$ in $G$, 
then apply Case 1 with $w_1$ and $u$ interchanged.
So assume that $v_1$ is not adjacent to any vertex of $N_{G[U]}(u_1)$ in~$G$.

If $E(V,N_{G[U]}(u_1))=\emptyset$ in $G$,
then any four vertices of $V$ and any four vertices of $N_{G[U]}(u_1)$ form $C_8$ in $\overline{G}$, 
and with $v_1$, form $W_8$ in $\overline{G}$, a contradiction.
So without loss of generality, 
assume that $v_2$ is adjacent to some vertex $u_2$ of $N_{G[U]}(u_1)$ in~$G$.
Since $S_n[4]\nsubseteq G$, 
$u_2$ is not adjacent to any vertex of $U\setminus\{u_1\}$.
Then $v_0,v_1,w_1,w_2$ and any four vertices from $U\setminus\{u_1,u_2\}$, 
at least three of which are from $N_{G[U]}(u_1)\setminus \{u_2\}$,
form $C_8$ in $\overline{G}$ and, with $u_2$, form $W_8$ in $\overline{G}$, 
a contradiction.

In either case, $R(S_n[4],W_8)\leq 2n$ for $n\equiv 0 \pmod{4}$ 
and $R(S_n[4],W_8)\leq 2n-1$ for odd $n$.

Next, suppose that $n\equiv 2 \pmod{4}$ and $G$ has order $2n-1$. 
If $G$ contains a subgraph $S_n(3)$, 
then the previous arguments show that $R(S_n[4],W_8)\leq 2n-1$.
Hence, we only need to consider the case where $G$ does not contain $S_n(3)$. 
Now, by Theorem~\ref{thm:R(Sn(4),W8)}, 
$G$ has a subgraph $T=S_n(4)$. 
Let $V(T)=\{v_0,\ldots,v_{n-4},w_1,w_2,w_3\}$ and $E(T)=\{v_0v_1,\ldots,v_0v_{n-4},v_1w_1,v_1w_2,v_1w_3\}$. 
Let $U=V(G)-V(T)$; then $|U|=n-1$. 
Since $G$ does not contain $S_n(3)$ and $S_n[4]$, $v_0$ is not adjacent in $G$ to $w_1$, $w_2$, $w_3$ or $U$. 
Now, set $U'=N_{G[U\cup \{w_1\}]}(w_1)\cup N_{G[U\cup \{w_2\}]}(w_2)\cup N_{G[U\cup \{w_3\}]}(w_3)$.
Then $|U'|\leq 3$ and $w_1$, $w_2$ and $w_3$ are not adjacent in $G$ to any vertex of $U\setminus U'$. 
By Lemma~\ref{lm4:R(Sn(1,2),W8)}, $G[U\setminus U']$ is either $K_{n-1-|U'|}$ or $K_{n-1-|U'|}-e$. 
If $d_{\overline{G}[U\setminus U']}(u')\geq 2$ for some vertex $u'$ in $U'$, 
then at least two vertices of $U\setminus U'$ are not adjacent to $u'$ in~$G$. 
Let $X$ be a set containing these two vertices and any other two vertices in $U\setminus U'$,
and set $Y=\{w_1,w_2,w_3,u'\}$. 
Note that $\overline{G}[X\cup Y]$ contains $C_8$ by Lemma~\ref{lem:1-1relation} which with $v_0$ forms $W_8$, 
a contradiction. 
Therefore, every vertex of $U'$ is adjacent in $G$ to at least $n-2-|U'|$ vertices of $U\setminus U'$. 
Hence, $\delta(G[U])\geq n-5$, 
and since $S_n[4]\nsubseteq G$, $E_G(T,U)=\emptyset$. 
Now, if $\overline{G}[V(T)]$ contains $S_5$, 
then $\overline{G}$ contains $W_8$ by Observation~\ref{obs2}, a contradiction. 
Therefore, $\delta(G[V(T)])\geq n-4$. 
By Lemma~\ref{lm:R(Sn[4],W8)}, $G$ contains $S_n[4]$, a contradiction. 
Thus, $R(S_n[4],W_8)\leq 2n-1$ for $n\equiv 2 \pmod{4}$.
%
\end{proof}

\begin{theorem}\label{thm:R(Sn(1,3),W8)}
If $n\geq 8$, then 
\[
  R(S_n(1,3),W_8)=\begin{cases}
    2n-1 & \text{if $n\not\equiv 0 \pmod{4}$}\,;\\
    2n   & \text{otherwise}\,.
  \end{cases}
\]
\end{theorem}

\begin{proof}
Lemma~\ref{lem:lower-bound-on-S_n(4)-S_n[4]-S_n(1,3)-S_n(3,1)-T_A(n)-T_B(n)-T_C(n)}
provides the lower bounds, so it remains to prove the upper bounds.
Let $G$ be any graph of order $2n$ if $n\equiv 0\pmod{4}$ and of order $2n-1$ if $n\not\equiv 0\pmod{4}$.
Assume that $G$ does not contain $S_n(1,3)$ and that $\overline{G}$ does not contain $W_8$.
By Theorem~\ref{thm:R(Sn[4],W8)}, $G$ has a subgraph $T = S_n[4]$.
Let 
$V(T)=\{v_0,\ldots,v_{n-4},w_1,w_2,w_3\}$ and 
$E(T)=\{v_0v_1,\ldots,v_0v_{n-4},w_1v_1,w_1w_2,w_1w_3\}$.
Set $V=\{v_2,\ldots,v_{n-4}\}$ and $U=V(G)-V(T)$.
Since $S_n(1,3)\nsubseteq G$, 
$w_2$ and $w_3$ are not adjacent to each other, or to any vertex in $U\cup V$.
Since $C_8\nsubseteq\overline{G}[U\cup V]$ as $W_8\nsubseteq\overline{G}$,
Lemma~\ref{lem:pancyclic} implies that $G[U\cup V]$ has a vertex $u$ of degree at least $n-3$ in $G[U\cup V]$.
Since $S_n(1,3)\nsubseteq G$, $u\in U$ and $u$ is not adjacent to any vertices in $V$.
Furthermore, $E(V,N_{G[U]}(u))=\emptyset$.
Finally, note that $w_3$, any $3$ vertices in $V$ and any $4$ vertices in $N_{G[U]}(u)$ form $C_8$ in $\overline{G}$ which, with $w_2$ as hub, form $W_8$, a contradiction.
\end{proof}

\begin{theorem}\label{thm:R(TA,W8)}
If $n\geq 8$, then 
\[
  R(T_A(n),W_8)=\begin{cases}
    2n-1 & \text{if $n\not\equiv 0 \pmod{4}$}\,;\\
    2n   & \text{otherwise}\,.
  \end{cases}
\]
\end{theorem}

\begin{proof}
Lemma~\ref{lem:lower-bound-on-S_n(4)-S_n[4]-S_n(1,3)-S_n(3,1)-T_A(n)-T_B(n)-T_C(n)}
provides the lower bounds, so it remains to prove the upper bounds.
Let $G$ be any graph of order $2n$ if $n\equiv 0\pmod 4$ and of order $2n-1$ if $n\not\equiv 0\pmod 4$.
Assume that $G$ does not contain $T_A(n)$ and that $\overline{G}$ does not contain $W_8$.

\smallskip

Suppose first that $G$ has a subgraph $T=S_n(3)$.
Let $V(T)=\{v_0,\ldots,v_{n-3},w_1,w_2\}$
and $E(T)=\{v_0v_1,\ldots,v_0v_{n-3},v_1w_1,v_1w_2\}$.
Set $V=\{v_2,\ldots,v_{n-3}\}$ and $U=V(G)-V(T)$. 
Since $G$ does not contain $T_A(n)$,
$w_1$ and $w_2$ are not adjacent to any vertex of $U\cup V$ in~$G$.
Let $V'$ be the set of any $n-5$ vertices in $V$, and $U'$ be the set of any $n-1$ vertices in $U$.
If $\delta(\overline{G}[U'\cup V'])\geq n-3$,
then $\overline{G}[U'\cup V']$ contains $C_8$ by Lemma~\ref{lem:pancyclic} which,
with $w_1$ as hub, form $W_8$, a contradiction.
Therefore, $\delta(\overline{G}[U'\cup V'])\leq n-4$ and $\Delta(G[U'\cup V'])\geq n-3$.
Since $T_A(n)\nsubseteq G$, $d_{G[U'\cup V']}(v)\leq n-6$ for each $v\in V'$.
Hence, some vertex $u\in U'$ satisfies $d_{G[U'\cup V']}(u)\geq n-3$, which also implies that $u$ is adjacent to at least two vertices of $U$.

Since $T_A(n)\nsubseteq G$, each vertex of $V$ is adjacent to at most one vertex of $N_{G[U]}(u)$.
If $|N_{G[U]}(u)|\ge n-4$, 
then each vertex of $N_{G[U]}(u)$ is adjacent to at most one vertex of $V$, 
and so $\overline{G}[V\cup N_{G[U]}(u)]$ contains $C_8$ by Lemma~\ref{lem:pancyclic} 
which with $w_1$ forms $W_8$, a contradiction.
Thus, at least three vertices of $V'$ (and so of $V$),
say $v_2, v_3, v_4$, are adjacent to $u$ in $G$.
Let $a$ and $b$ be any two vertices in $N_{G[U]}(u)$.
As $T_A(n)\nsubseteq G$, each of $v_2,v_3,v_4$ is not adjacent to any vertex of $V(G)\setminus\{u,v_0\}$.
Then $w_1v_5w_2v_3av_1bv_4w_1$ and $v_2$ form $W_8$ in $\overline{G}$, a contradiction.

\smallskip
By Theorem~\ref{thm:R(Pn-Sn(1,2)-Sn(3)-Sn(2,1),W8)}, 
$R(S_n(3),W_8)\leq 2n$ for $n\equiv 0 \pmod{4}$.
So now assume that $G$ has order $2n-1$ with $n\not\equiv 0 \pmod{4}$
and that $G$ does not contain $S_n(3)$.
By Theorem~\ref{thm:R(Sn(4),W8)}, $G$ has a subgraph $T=S_n(4)$.
Let $V(T)=\{v_0,\ldots,v_{n-4},w_1,w_2,w_3\}$ and $E(T)=\{v_0v_1,\ldots,v_0v_{n-4},v_1w_1,v_1w_2,v_1w_3\}$.
Then $U=V(G)-V(T)$ and $|U|=n-1$.
Since $T_A(n)\nsubseteq G$, $w_1,w_2,w_3$ are not adjacent to each other in $G$ or to any vertex of $U$.
Since $S_3(n)\nsubseteq G$, $v_0$ is not adjacent any vertex of $U\cup\{w_1,w_2,w_3\}$.
By Lemma~\ref{lm4:R(Sn(1,2),W8)}, $G[U]$ is $K_{n-1}$ or $K_{n-1}-e$.
Since $T_A(n)\nsubseteq G$, each vertex of $T$ is not adjacent to any vertex of $U$ in $G$,
and so $\delta(G[V(T)])\geq n-4$ by Observation~\ref{obs2}, which in turn implies that $G[V(T)]$ contains $T_A(n)$ by Lemma~\ref{lm:R(Sn[4],W8)}, a contradiction.

This completes the proof of the theorem.
\end{proof}

\begin{theorem}\label{thm:R(TB,W8)}
If $n\geq 8$, then 
\[
  R(T_B(n),W_8)=\begin{cases}
    2n-1 & \text{if $n\not\equiv 0 \pmod{4}$}\,;\\
    2n   & \text{otherwise}\,.
  \end{cases}
\]
\end{theorem}

\begin{proof}
Lemma~\ref{lem:lower-bound-on-S_n(4)-S_n[4]-S_n(1,3)-S_n(3,1)-T_A(n)-T_B(n)-T_C(n)}
provides the lower bounds, so it remains to prove the upper bounds.
Let $G$ be a graph with no $T_B(n)$ subgraph 
whose complement $\overline{G}$ does not contain $W_8$. 

Suppose that $n\equiv 0 \pmod{4}$ and that $G$ has order $2n$. 
By Theorem~\ref{thm:R(Sn[4],W8)}, $G$ has a subgraph $T=S_n[4]$. 
Let $V(T)=\{v_0,\ldots,v_{n-4},w_1,w_2,w_3\}$ 
and $E(T)=\{v_0v_1,\ldots,v_0v_{n-4},v_1w_1,w_1w_2,w_1w_3\}$. 
Set $V=\{v_2,\ldots,v_{n-4}\}$ and $U=V(G)-V(T)$; 
then $|V|=n-5$ and $|U|=n$. 
Since $T_B(n)\nsubseteq G$, 
$E_G(U,V)=\emptyset$ and neither $w_2$ nor $w_3$ is adjacent in $G$ to $V$. 
Suppose that $n\geq 12$.
If $w_2$ is non-adjacent to some $4$ vertices from $U$, 
then these $4$ vertices and any $4$ vertices from $V$ form $C_8$ in $\overline{G}$ 
that with $w_2$ forms $W_8$, 
a contradiction.
Otherwise, $w_2$ must be adjacent to at least $n-3$ vertices of $U$ in $G$. 
Since $T_B(n)\nsubseteq G$, $w_3$ must not be adjacent to these $n-3$ vertices;
then any $4$ vertices from these $n-3$ vertices and $4$ vertices from $V$ form $C_8$ in $\overline{G}$ 
and with $w_3$ forms $W_8$, again a contradiction. 
For $n=8$, $|V|=3$ and $|U|=8$. 
If $w_2$ is not adjacent to any vertex of $U$ in $G$, 
then by Lemma~\ref{lm4:R(Sn(1,2),W8)}, 
$G[U]$ is $K_8$ or $K_8 - e$ which contains $T_B(8)$, a contradiction. 
Otherwise, suppose that $w_2$ is adjacent to $u\in U$. 
Since $T_B(8)\nsubseteq G$, 
$w_1$ must not be adjacent to $(U\cup V)\setminus \{u\}$ in $G$. 
Now, if $w_3$ is not adjacent to $v_0$ in $G$, 
then by Observation~\ref{obs2}, $\overline{G}$ contains $W_8$, a contradiction. 
Otherwise, $u$ is not adjacent to $V\cup \{w_3\}$, 
and again by Observation \ref{obs2}, 
$\overline{G}$ contains $W_8$, another contradiction. 
Thus, $R(T_B(n),W_8)\leq 2n$ for $n\equiv 0 \pmod{4}$. 

Next, suppose that $n\not\equiv 0 \pmod{4}$ and that $G$ has order $2n-1$.
By Theorem~\ref{thm:R(Sn[4],W8)}, $G$ has a subgraph $T=S_n[4]$. 
Let $V(T)=\{v_0,\ldots,v_{n-4},w_1,w_2,w_3\}$ 
and $E(T)=\{v_0v_1,\ldots,v_0v_{n-4},v_1w_1,w_1w_2,w_1w_3\}$. 
Set $V=\{v_2,\ldots,v_{n-4}\}$ and $U=V(G)-V(T)$; 
then $|V|=n-5$ and $|U|=n-1$. 
Since $T_B(n)\nsubseteq G$, 
$E_G(U,V)=\emptyset$ and neither $w_2$ nor $w_3$ is adjacent in $G$ to $V$. 
For $n\geq 9$, if $w_2$ is non-adjacent to some $4$ vertices from $U$, 
then these $4$ vertices and any $4$ vertices from $V$ form $C_8$ in $\overline{G}$ and with $w_2$ form $W_8$, 
a contradiction. 
Otherwise, $w_2$ is adjacent to at least $n-4$ vertices of $U$ in $G$. 
Since $T_B(n)\nsubseteq G$, 
$w_3$ is not adjacent to these $n-4$ vertices, 
so any $4$ vertices from these $n-4$ vertices and $4$ vertices from $V$ 
form $C_8$ in $\overline{G}$ which with $w_3$ form $W_8$, again a contradiction. 
Therefore, $R(T_B(n),W_8)\leq 2n-1$ for $n\not\equiv 0 \pmod{4}$.

This completes the proof.
\end{proof}

\begin{theorem}
\label{thm:R(TC,W8)}
For $n\geq 8$, $R(T_C(n),W_8)=2n-1$.
\end{theorem}

\begin{proof}
Lemma~\ref{lem:lower-bound-on-S_n(4)-S_n[4]-S_n(1,3)-S_n(3,1)-T_A(n)-T_B(n)-T_C(n)}
provides the lower bound, so it remains to prove the upper bound.
Let $G$ be any graph of order $2n-1$
and assume that $G$ does not contain $T_C(n)$ and that $\overline{G}$ does not contain $W_8$.

Suppose first that there is a subset $X\subseteq V(G)$ of size $n$ with $\delta(G[X])\ge n-4$.
If $\delta(G[X]) = n-4$, then let $x\in X$ be such that $d_{G[X]}(x) = n-4$, 
and set $Y = X\setminus (\{x\}\cup N_{G[X]}(x))$ where $|Y|=3$.
Noting that $3(n-6)>n-4$ for $n\ge 8$, 
there must be two vertices of $Y$ that are adjacent to a common vertex of $N_{G[X]}(x)$ in $G$, 
say to $x'\in N_{G[X]}(x)$.
Then the remaining vertex of $Y$ is not adjacent to any vertex of $N_{G[X]}(x)\setminus\{x'\}$,
as $T_C(n)\nsubseteq G$, contradicting $\delta(G[X])\geq n-4$.
So $\delta(G[X]) \ge n-3$.
Pick any vertex $x\in X$ and any subset $X'\subseteq N_{G[X]}(x)$ of size $n-3$.
Set $Y = X\setminus (\{x\}\cup X')$ where $|Y|=2$.
As $2(n-5)>n-3$ for $n\ge 8$, 
the two vertices of $Y$ must be adjacent to a common vertex of $X'$ in $G$, say $x'$.
Then $G[X'\setminus\{x'\}]$ is an empty graph as $T_C(n)\nsubseteq G$, 
contradicting $\delta(G[X])\geq n-3$.

Now assume that $\delta(G[X])\le n-5$ whenever $X\subseteq V(G)$ is of size $n$.
By Theorem~\ref{thm:R(Pn-Sn(1,2)-Sn(3)-Sn(2,1),W8)},
$G$ has a subgraph $T=S_{n-1}(3)$.
Let $V(T)=\{v_0,\ldots,v_{n-4},w_1,w_2\}$ and $E(T)=\{v_0v_1,\ldots,v_0v_{n-4},v_1w_1,v_1w_2\}$.
Set $V=\{v_2,v_3,\ldots,v_{n-4}\}$ and $U=V(G)-V(T)$; then $|V|=n-5$ and $|U|=n$.
Since $T_C(n)\nsubseteq G$, $E_G(U,V)=\emptyset$.

For the case $n=8$ such that $v_1$ is not adjacent to any vertex of $U$ in $G$, or the case $n\ge 9$, 
there are four vertices of $V(T)$ that are not adjacent to any vertex of $U$ in~$G$.
Since $\delta(G[U])\le n-5$, $\overline{G}[U]$ contains $S_5$, and so $\overline{G}$ contains $W_8$ by Observation~\ref{obs2}, a contradiction.

For the final case $n=8$ with $v_1$ adjacent to some vertex $u$ of $U$ in $G$, 
observe that since $T_C(8)\nsubseteq G$, 
the vertex $u$ is not adjacent to any vertex of  $\{v_2, v_3, v_4\}\cup U$.
By Lemma~\ref{lm4:R(Sn(1,2),W8)}, $G[U\setminus\{u\}]$ is $K_7$ or $K_7 - e$, 
which implies that no vertex of $V(T)\cup \{u\}$ is adjacent to any vertex of $U\setminus\{u\}$ in $G$, 
as $T_C(8)\nsubseteq G$.
Since $\delta(G[V(T)\cup\{u\}])\le n-5$, $\overline{G}[V(T)\cup\{u\}]$ contains $S_5$, 
and so $\overline{G}$ contains $W_8$ by Observation~\ref{obs2}, a contradiction.

This completes the proof of the theorem.
\end{proof}

\begin{theorem}
\label{thm:R(Sn(3,1),W8)}
For $n\geq 8$, $R(S_n(3,1),W_8)=2n-1$. 
\end{theorem}

\begin{proof}
Lemma~\ref{lem:lower-bound-on-S_n(4)-S_n[4]-S_n(1,3)-S_n(3,1)-T_A(n)-T_B(n)-T_C(n)}
provides the lower bound, so it remains to prove the upper bound.
Let $G$ be any graph of order $2n-1$.
Assume that $G$ does not contain $S_n(3,1)$ and that $\overline{G}$ does not contain $W_8$.

Suppose first that there is a subset $X\subseteq V(G)$ of size $n$ with $\delta(G[X])\ge n-4$.
Let $x_0$ be any vertex of $X$, and pick a subset $X'\subseteq N_{G[X]}(x_0)$ of size $n-4$.
Set $Y = X\setminus (\{x_0\}\cup X')$, and so $|Y|=3$.
Since $\delta(G[X])\ge n-4$, each vertex of $Y$ is adjacent to at least $n-7$ vertices of $X'$ in~$G$.
For $n\ge 10$, it is straightforward to see that there is a matching from $Y$ to $X'$ in $G$;
hence, $G$ contains $S_n(3,1)$, a contradiction.
For $n=9$, if $d_{G[X]}(x_0) = n-4 = 5$, 
then we can similarly deduce the contradiction that $G$ contains $S_9(3,1)$, 
since in this case, each vertex of $Y$ is adjacent to at least $n-6=3$ vertices of $X'$ in~$G$.
As $x_0$ was arbitrary, we may assume for the case when $n=9$ that $\delta(G[X])\ge n-3 = 6$, 
which again leads to the contradiction that $G$ contains $S_9(3,1)$.

Now for $n=8$, suppose $d_{G[X]}(x_0) = 4$.
Let $X' = \{x_1,x_2,x_3,x_4\}$ and $Y = \{x_5,x_6,x_7\}$.
Since $\delta(G[X])\ge n-4$ and $S_8(3,1)\nsubseteq G$, $G[Y]$ is $K_3$; 
all three vertices of $Y$ are adjacent to exactly two common vertices of $X'$ in $G$, 
say to $x_1$ and $x_2$; 
and neither $x_3$ nor $x_4$ are adjacent to any vertex of $Y$ in~$G$.
By the minimum degree condition, 
$x_3$ and $x_4$ are then adjacent in $G$, and each is also adjacent to both $x_1$ and $x_2$.
This implies that $G$ contains $S_8(3,1)$, with $x_1$ being the vertex with degree four, a contradiction.
As $x_0$ was arbitrary, assume for the case when $n=8$ that $\delta(G[X])\ge 5$, 
which again leads to the contradiction that $G$ contains $S_8(3,1)$.

Now assume that $\delta(G[X])\le n-5$ whenever $X\subseteq V(G)$ is of size $n$.
Recall that $G$ has order $2n-1$, and so by Theorem~\ref{thm:R(Pn-Sn(1,2)-Sn(3)-Sn(2,1),W8)}, 
$G$ has a subgraph $T=S_{n-1}(2,1)$.
Let $V(T)=\{v_0,\ldots,v_{n-4},w_1,w_2\}$ and $E(T)=\{v_0v_1,\ldots,v_0v_{n-4},v_1w_1,v_2w_2\}$.
Set $V=\{v_3,v_4,\ldots,v_{n-4}\}$ and $U=V(G)-V(T)$; then $|V|=n-6$ and $|U|=n$.
Since $S_n(3,1)\nsubseteq G$, $E_G(U,V)=\emptyset$.
Now as $\delta(G[U])\le n-5$, $\overline{G}[U]$ contains $S_5$, 
and so for $n\ge 10$, $\overline{G}$ contains $W_8$ by Observation~\ref{obs2}, a contradiction.

For $n=9$, Theorem~\ref{thm:R(Pn-Sn(1,2)-Sn(3)-Sn(2,1),W8)} shows that $G$ has a subgraph $T=S_9(2,1)$, 
so without loss of generality assume that $v_0$ is adjacent to some vertex $u$ in $U$.
Since $S_9(3,1)\nsubseteq G$, 
$G[V\cup\{u\}]$ is an empty graph and $u$ is not adjacent to any vertex of $U$ in~$G$.
By Lemma~\ref{lm4:R(Sn(1,2),W8)}, $G[U\setminus\{u\}]$ is $K_8$ or $K_8 - e$, 
which implies that no vertex of $V(T)\cup \{u\}$ is adjacent to any vertex of $U\setminus\{u\}$ in $G$, 
as $S_9(3,1)\nsubseteq G$.
Since $\delta(G[V(T)\cup\{u\}])\le n-5$, $\overline{G}[V(T)\cup\{u\}]$ contains $S_5$, 
and so $\overline{G}$ contains $W_8$ by Observation~\ref{obs2}, a contradiction.

Finally for $n=8$, recall that $G$ has order $15$, and so $G$ has a subgraph $T'=S_7$ by Theorem~\ref{thm:R(Sn,W8)}.
Let $V(T')=\{v_0',\ldots,v_6'\}$ and $E(T')=\{v_0'v_1',\ldots,v_0'v_6'\}$.
Set $V'=\{v_1',\ldots,v_6'\}$ and $U'=V(G)-V(T')$, then $|U'| = 8$.
Suppose that $v_2'$ and $v_3'$ are adjacent to a common vertex $u$ of $U'$ in $G$, 
while $v_1'$ is adjacent to another vertex $u'\neq u$ of $U'$ in~$G$.
Then as $S_8(3,1)\nsubseteq G$, 
no vertex of $\{v_4',v_5',v_6'\}\cup (U'\setminus\{u,u'\})$ is adjacent to 
any vertex of $V'\setminus\{v_1'\}$ in~$G$.
Now $G[V'\setminus\{v_1'\}]$ contains $S_5$ and $|U'\setminus\{u,u'\}|=6$, 
and so $\overline{G}$ contains $W_8$ by Observation~\ref{obs2}, a contradiction.
Similar arguments lead to the same contradiction when the roles of $v_1',v_2'$, and $v_3'$ 
are replaced by any three vertices of $V'$.
So assume that it is not the case that 
two vertices of $V'$ are adjacent to a common vertex of $U'$ in $G$ 
while a third vertex of $V'$ is adjacent to another vertex of $U'$ in~$G$.

For $1\le i\le 6$, 
let $d_i = |E_G(\{v_i'\}, U)|$ be the number of vertices of $U'$ that are adjacent to $v_i'$.
Without loss of generality, assume that $d_1\ge d_2\ge \cdots \ge d_6$.
Recalling that $\delta(G[U'])\le 3$ and so $S_5\subseteq\overline{G}[U']$, 
Observation~\ref{obs2} implies that $d_3\ge 1$.
If $d_1\ge 3$ and $d_2\ge 2$, 
then it is trivial that $G$ contains $S_8(3,1)$, a contradiction.
By our assumption on the adjacencies of vertices in $V'$ to vertices of $U'$ in $G$, 
it is clear that when $(d_1,d_2,d_3)$ is of the form  $(\ge 3, 1, 1)$, $(2,2,2)$, or $(2,2,1)$, 
there is a matching from $\{v_1',v_2',v_3'\}$ to $U'$ in $G$, 
as $v_2'$ and $v_3'$ are adjacent to different vertices of $U'$ in~$G$.
This implies that $G$ contains $S_8(3,1)$, a contradiction.
If $(d_1,d_2,d_3) = (2,1,1)$, 
then, similarly, $v_2'$ and $v_3'$ are adjacent to different vertices of $U'$ in $G$, 
say to $u$ and $u'$, respectively, 
which in turn implies that $v_1'$ is adjacent to two vertices in $U'\setminus\{u,u'\}$.
So $G$ contains $S_8(3,1)$, again a contradiction.

For the final case when $d_1=d_2=d_3=1$, 
our assumption implies that $v_1'$, $v_2'$ and $v_3'$ must be adjacent to a common vertex $u$ of $U'$ in $G$
to avoid a matching from $\{v_1',v_2',v_3'\}$ to $U'$ in~$G$.
Furthermore, no vertex of $\{v_4',v_5',v_6'\}$ is adjacent to any vertex of $U'\setminus\{u\}$ in~$G$.
Now if $S_5\subseteq\overline{G}[V']$, 
then $\overline{G}$ contains $W_8$ by Observation~\ref{obs2}, a contradiction.
So $\delta(G[V'])\ge 2$, and in particular, $v_4'$ is adjacent to some vertex of $V'$ in~$G$.
Without loss of generality, $v_4$ is adjacent to either $v_1$ or $v_5$ in~$G$.
Since $S_8(3,1)\nsubseteq G$, $\overline{G}[\{v_5',v_2',v_3',v_6'\}]$ contains $S_4$ 
if $v_4'$ is adjacent to $v_1'$ in $G$, 
while $\overline{G}[\{v_6',v_1',v_2',v_3'\}]$ contains $S_4$ if $v_4'$ is adjacent to $v_5'$ in~$G$.
By Lemma~\ref{lm4:R(Sn(1,2),W8)}, $G[U'\setminus\{u\}]$ is $K_7$ or $K_7 - e$, 
which implies that no vertex of $V(T')\cup \{u\}$ is adjacent to any vertex of $U'\setminus\{u\}$ in $G$, 
as $S_8(3,1)\nsubseteq G$.
Since $\delta(G[V(T')\cup\{u\}])\le 3$, $\overline{G}[V(T)\cup\{u\}]$ contains $S_5$, 
and so $\overline{G}$ contains $W_8$ by Observation~\ref{obs2}, a contradiction.

Thus, $R(S_n(3,1),W_8)\leq 2n-1$ for $n\geq 8$ which completes the proof.
\end{proof}

\section{Proof of Theorem~\ref{thm:R(Sn(1,4)-TDEFGHJKLM)}}
\label{sec:n-5proof}

\begin{lemma}
\label{lem:lower-bound-on-S_n(1,4)-T_D(n)-...-T_L(n)}
Let $n\geq 8$. 
Then $R(T_n,W_8)\geq 2n-1$ for each 
\[
  T_n \in\{S_n(1,4),S_n(5), S_n[5], S_n(2,2), S_n(4,1), T_D(n), \ldots, T_S(n)\}\,.
\]
Also, $R(T_n,W_8)\geq 2n$ if $n\equiv 0 \pmod{4}$ 
and $T_n \in \{S_n(1,4), S_n(2,2), T_D(n), T_N(n)\}$
or if $T_n \in \{T_E(8), T_F(8)\}$.
\end{lemma}

\begin{proof}
The graph $G=2K_{n-1}$ clearly does not contain any tree graph of order $n$,
and $\overline{G}$ does not contain $W_8$. 
Furthermore, if $n\equiv 0 \pmod{4}$, 
then the graph $G=K_{n-1}\cup K_{4,\ldots,4}$ of order $2n-1$ does not contain 
$S_n(1,4)$, $T_D(n)$ or $S_n(2,2)$; 
nor does the complement $\overline{G}$ contain $W_8$. 
Finally, the graph $G = K_7\cup K_{4,4}$ does not contain $T_E(8)$ or $T_F(8)$ 
and $\overline{G}$ does not contain $W_8$.
\end{proof}

\begin{theorem}\label{thm:R(Sn(1,4),W8)}
If $n\geq 8$, then 
\[
  R(S_n(1,4),W_8)=\begin{cases}
    2n-1 & \text{if $n\not\equiv 0 \pmod{4}$}\,;\\
    2n   & \text{otherwise}\,.
  \end{cases}
\]
\end{theorem}

\begin{proof}
Lemma~\ref{lem:lower-bound-on-S_n(1,4)-T_D(n)-...-T_L(n)}
provides the lower bound, so it remains to prove the upper bound.
Let $G$ be a graph with no $S_n(1,4)$ subgraph whose complement $\overline{G}$ does not contain $W_8$. 
Suppose that $G$ has order $2n$ if $n\equiv 0 \pmod{4}$ 
and that $G$ has order $2n-1$ if $n\not\equiv 0 \pmod{4}$. 
By Theorem~\ref{thm:R(Sn(1,3),W8)}, 
$G$ has a subgraph $T=S_n(1,3)$. 
Let $V(T)=\{v_0,\ldots,v_{n-4},w_1,w_2,w_3\}$ 
and $E(T)=\{v_0v_1,\ldots,v_0v_{n-4},v_1w_1,w_1w_2,w_2w_3\}$. 
Set $V=\{v_2,\ldots,v_{n-4}\}$ and $U=V(G)-V(T)$;
then $|V|=n-5$ and $|U|=j$ 
where $j=n$ if $n\equiv 0 \pmod{4}$ and $j=n-1$ if $n\not\equiv 0 \pmod{4}$. 
Since $S_n(1,4)\nsubseteq G$, 
$w_3$ is not adjacent in $G$ to any vertex of $U\cup V$ 
and $d_{G[U\cup V]}(v_i)\leq n-7$ for each $v_i\in V$. 
If $\delta(\overline{G}[U\cup V])\geq \lceil\frac{n-5+j}{2}\rceil\geq \frac{n-5+j}{2}$, 
then $\overline{G}[U\cup V]$ contains $C_8$ by Lemma~\ref{lem:pancyclic} and thus $W_8$ with $w_3$ as hub,
a contradiction. 
Therefore, $\delta(\overline{G}[U\cup V])\leq \lceil\frac{n-5+j}{2}\rceil-1$
and 
$\Delta(G[U\cup V])
 \geq n-5+j-\lceil\frac{n-5+j}{2}\rceil 
   =  \lfloor\frac{n-5+j}{2}\rfloor
 \geq n-3$.
Since $d_{G[U\cup V]}(v_i)\leq n-7$ for each $v_i\in V$, 
$d_{G[U\cup V]}(u)\geq n-3$ for some vertex $u\in U$. 
Since $S_n(1,4)\nsubseteq G$, 
no vertex of $V$ is adjacent to $\{u\}\cup N_{G[U\cup V]}(u)$ in~$G$. 

For $n\geq 9$, any $4$ vertices from $V$ and any $4$ vertices from $\{u\}\cup N_{G[U\cup V]}(u)$ form $C_8$ in $\overline{G}$ and, with $w_3$ as hub, form $W_8$, a contradiction. 
Suppose that $n=8$; then $V=\{v_2,v_3,v_4\}$. 
Let $\{u_1,\ldots,u_4\}$ be $4$ vertices in $N_{G[U\cup V]}(u)$. 
Since $S_8(1,4)\nsubseteq G$, $w_1$ is not adjacent to $N_{G[U\cup V]}(u)$.
If $w_1$ is not adjacent to $w_3$, 
then $w_1u_1v_2u_2v_3u_3v_4u_4w_1$ and $w_3$ form $W_8$ in $\overline{G}$, a contradiction. 
Therefore, $w_1$ is adjacent to $w_3$ in~$G$. 
Then $w_2$ is not adjacent to any vertex of $U\cup V$ in~$G$. 
Since $d_{G[V]}(v_i)\leq 1$ for $i=2,3,4$, 
one of the vertices of $V$, say $v_2$, is not adjacent to the other two vertices of $V$. 
Then $u_1w_2u_2w_3u_3v_3u_4v_4u_1$ and $v_2$ form $W_8$ in $\overline{G}$, a contradiction.
Thus, $R(S_n(1,4),W_8)\leq 2n$ for $n\equiv 0 \pmod{4}$ and $R(S_n(1,4),W_8)\leq 2n-1$ for $n\not\equiv 0 \pmod{4}$. 

This completes the proof. 
\end{proof}

\begin{theorem}
\label{thm:R(Sn(5),W8)}
If $n\geq 9$, then $R(S_n(5),W_8)=2n-1$.
\end{theorem}

\begin{proof}
Lemma~\ref{lem:lower-bound-on-S_n(1,4)-T_D(n)-...-T_L(n)} provides the lower bound, 
so it remains to prove the upper bound. 
Let $G$ be any graph of order $2n-1$. 
Assume that $G$ does not contain $S_n(5)$ and that $\overline{G}$ does not contain $W_8$. 
By Theorem~\ref{thm:R(Sn(4),W8)}, $G$ has a subgraph $T=S_n(4)$. 
Let $V(T)=\{v_0,\ldots,v_{n-4},w_1,w_2,w_3\}$ 
and $E(T)=\{v_0v_1,\ldots,v_0v_{n-4},v_1w_1,v_1w_2,v_1w_3\}$. 
Set $V=\{v_2,v_3,\ldots,v_{n-4}\}$ and $U=V(G)-V(T)$; 
then $|V|=n-5$ and $|U|=n-1$. 
Since $S_n(5)\nsubseteq G$, 
$v_1$ is not adjacent to any vertex of $U\cup V$ in~$G$. 
Furthermore, for each $v_i$ in $V$, $v_i$ is adjacent to at most three vertices of $U$ in~$G$.

For $n\geq 9$, $|V|\geq 4$ and $|U|\geq 8$. 
By Corollary~\ref{cor:bipartite4-8}, 
$\overline{G}[U\cup V]$ contains $C_8$
which together with $v_1$ gives $W_8$ in $\overline{G}$, a contradiction. 
Thus, $R(S_n(5),W_8)\leq 2n-1$ which completes the proof.
\end{proof}

\begin{theorem}
If $n\geq 9$, then $R(S_n[5],W_8)=2n-1$.
\end{theorem}

\begin{proof}
Lemma~\ref{lem:lower-bound-on-S_n(1,4)-T_D(n)-...-T_L(n)} provides the lower bound, 
so it remains to prove the upper bound. 
Let $G$ be any graph of order $2n-1$. 
Assume that $G$ does not contain $S_n[5]$ and that $\overline{G}$ does not contain~$W_8$. 
By Theorem~\ref{thm:R(Sn(5),W8)}, $G$ has a subgraph $T=S_{n}(5)$. 
Let $V(T)=\{v_0,\ldots,v_{n-5},w_1,\ldots,w_4\}$ 
and $E(T)=\{v_0v_1,\ldots,v_0v_{n-5},v_1w_1,\ldots,v_1w_4\}$. 
Set $V=\{v_2,\ldots,v_{n-5}\}$ and $U=V(G)-V(T)$; 
then $|V|=n-6$ and $|U|=n-1$. 
Since $S_n[5]\nsubseteq G$, 
$v_0$ is not adjacent to $w_1,\ldots,w_4$ in $G$ 
and $w_1,\ldots,w_4$ are each adjacent to at most two vertices of $U$ in~$G$. 
Now, suppose that $v_0$ is non-adjacent to at least six vertices of $U$ in~$G$. 
By Corollary~\ref{cor:bipartite4-6}, six of these vertices together with $w_1,\ldots,w_4$ contain $C_8$ in $\overline{G}$ which with $v_0$ gives $W_8$ in $\overline{G}$, a contradiction. 
Then, suppose that $v_0$ is adjacent to at least $n-6$ vertices of $U$ in~$G$. 
Choose a set $U'$ of $n-6$ of these vertices.
Since $S_n[5]\nsubseteq G$, $v_1$ is not adjacent to any vertex of $V\cup U'$ in~$G$. 
If $\delta(\overline{G}[V\cup U'])\geq n-6$, 
then by Lemma~\ref{lem:pancyclic}, 
$\overline{G}[V\cup U']$ contains $C_8$ which with $v_1$ gives $W_8$ in $\overline{G}$, a contradiction. 
Therefore, $\delta(\overline{G}[V\cup U'])\leq n-7$ and $\Delta(G[V\cup U'])\geq n-6$. 
However, this gives $S_n[5]$ in $G$ with $u$ and $v_1$ as the center of $S_{n-5}$ and $S_5$, respectively, 
where $u$ is a vertex in $V\cup U'$ with $d_{G[V\cup U']}(u)\geq n-6$,
a contradiction. 
Thus, $R(S_n[5],W_8)\leq 2n-1$ which completes the proof.
\end{proof}

\begin{theorem}
If $n\geq 8$, then 
\[
  R(S_n(2,2),W_8)=\begin{cases}
    2n-1 & \text{if $n\not\equiv 0 \pmod{4}$}\,;\\
    2n   & \text{otherwise}\,.
  \end{cases}
\]  
\end{theorem}

\begin{proof}
Lemma~\ref{lem:lower-bound-on-S_n(1,4)-T_D(n)-...-T_L(n)}
provides the lower bound, so it remains to prove the upper bound.
Assume that $G$ is a graph with no $S_n(2,2)$ subgraph whose complement $\overline{G}$ does not contain $W_8$.
Suppose that $n\equiv 0 \pmod{4}$ and that $G$ has order $2n$. 
By Theorem~\ref{thm:R(TB,W8)}, $G$ has a subgraph $T=T_B(n)$. 
Let $V(T)=\{v_0,\ldots,v_{n-4},w_1,w_2,w_3\}$ 
and $E(T)=\{v_0v_1,\ldots,v_0v_{n-4},v_1w_1,w_1w_2,v_2w_3\}$. 
Set $V=\{v_3,\ldots,v_{n-4}\}$ and $U=V(G)-V(T)$; then $|V|=n-6$ and $|U|=n$. 
Since $S_n(2,2)\nsubseteq G$, $w_3$ is not adjacent in $G$ to $U\cup V$ and $v_2$ is not adjacent to $V$. 
If $\delta(\overline{G}[U\cup V])\geq \frac{2n-6}{2}=n-3$, 
then $\overline{G}[U\cup V]$ contains $C_8$ by Lemma~\ref{lem:pancyclic} which with $w_2$ forms $W_8$, 
a contradiction. 
Therefore, $\delta(\overline{G}[U\cup V])\leq n-4$, and $\Delta(G[U\cup V])\geq n-3$. 
Now, there are two cases to be considered. 

\smallskip
\noindent {\bf Case 1a}: One of the vertices of $V$, say $v_3$, is a vertex of degree at least $n-3$ in $G[U\cup V]$. 

Note that in this case, there are at least $4$ vertices from $U$, say $u_1,\ldots,u_4$, that are adjacent to $v_3$ in~$G$. 
Since $S_n(2,2)\nsubseteq G$, these $4$ vertices are independent and are not adjacent to any other vertices of $U$. 
Since $n\geq 8$, $U$ contains at least $4$ other vertices, say $u_5,\ldots,u_8$, 
so $u_1u_5u_2u_6u_3u_7u_4u_8u_1$ and $w_3$ forms $W_8$ in $\overline{G}$, a contradiction.

\smallskip
\noindent {\bf Case 1b}: Some vertex $u\in U$ has degree at least $n-3$ in $G[U\cup V]$.

Since $S_n(2,2)\nsubseteq G$, $u$ is not adjacent to any vertex of $V$ in~$G$. 
Therefore, $u$ must be adjacent to at least $n-3$ vertices of $U$ in~$G$. 
Without loss of generality, suppose that $u_1,\ldots,u_{n-3}\in N_{G[U]}(u)$. 
Note that $V$ is not adjacent to $N_{G[U]}(u)$, or else there will be $S_n(2,2)$ in $G$, a contradiction. 
If $n\geq 12$, 
then any $4$ vertices from $N_{G[U]}(u)$ and any $4$ vertices from $V$ form $C_8$ in $\overline{G}$ which, with $w_3$ as hub, forms $W_8$, a contradiction. 
Suppose that $n=8$ and let the remaining two vertices be $u_6$ and $u_7$. 
If $|N_{G[\{u_1,\ldots,u_5,u_i\}}(u_i)|\leq 1$ for $i=6,7$, 
then let $X=\{u_1,\ldots,u_4\}$ and $Y=\{v_3,v_4,u_6,u_7\}$. 
By Lemma~\ref{lem:1-1relation}, $\overline{G}[X\cup Y]$ contains $C_8$ and, with $w_3$ as hub, forms $W_8$ in $\overline{G}$, a contradiction. 
Therefore, one of $u_6$ and $u_7$, say $u_6$, is adjacent to at least two of $u_1,\ldots,u_5$, 
say $u_1$ and $u_2$. 
Since $S_8(2,2)\nsubseteq G$, $u_7$ is adjacent in $\overline{G}$ to at least two of $u_3,u_4,u_5$, 
say $u_3$ and $u_4$, 
and $v_0,\ldots,v_4,w_1$ are not adjacent in $G$ to $u,u_1,\ldots,u_6$. 
Now, if $w_3$ is not adjacent to some vertex $a\in\{v_0,v_1,w_1\}$, 
then $u_1v_3u_2v_4u_3u_7u_4au_1$ and $w_3$ form $W_8$ in $\overline{G}$, a contradiction. 
Hence, $w_3$ is adjacent to $v_0$, $v_1$ and $w_1$ in~$G$. 
Similarly, $v_2$ is not adjacent to $u_7$ and $v_2$ is adjacent to $v_1$ and $w_1$. 
Since $S_8(2,2)\nsubseteq G$, $w_2$ is not adjacent to $U\cup V$, and $w_1$ is not adjacent to $V$. 
Then $u_1v_2u_2w_1u_3w_2u_4w_3u_1$ and $v_3$ forms $W_8$ in $\overline{G}$, a contradiction. 

In either case, $R(S_n(2,2),W_8)\leq 2n$. 

Suppose that $n\not\equiv 0 \pmod{4}$ and that $G$ has order $2n-1$. 
By Theorem~\ref{thm:R(TB,W8)}, $G$ has a subgraph $T=T_B(n)$. 
Let $V(T)=\{v_0,\ldots,v_{n-4},w_1,w_2,w_3\}$ 
and $E(T)=\{v_0v_1,\ldots,v_0v_{n-4},v_1w_1,w_1w_2,v_2w_3\}$. 
Set $V=\{v_3,\ldots,v_{n-4}\}$ and $U=V(G)-V(T)$; 
then $|V|=n-6$ and $|U|=n-1$. 
Since $S_n(2,2)\nsubseteq G$, 
$w_3$ is not adjacent in $G$ to $U\cup V$. 
If $\delta(\overline{G}[U\cup V])\geq \lceil\frac{2n-5}{2}\rceil$, 
then $\overline{G}[U\cup V]$ contains $C_8$ by Lemma~\ref{lem:pancyclic} which with $w_3$ forms $W_8$, 
a contradiction. 
Therefore, $\delta(\overline{G}[U\cup V])\leq \lceil\frac{2n-5}{2}\rceil-1=n-3$, 
and $\Delta(G[U\cup V])\geq n-3$. 
Again, there are two cases to be considered. 

\smallskip
\noindent {\bf Case 2a}: A vertex of $V$, say $v_3$, has degree at least $n-3$ in $G[U\cup V]$. 

There must be at least $4$ vertices from $U$, say $u_1,\ldots,u_4$ that are adjacent to $v_3$ in~$G$. 
Since $S_n(2,2)\nsubseteq G$, $u_1,\ldots,u_4$ are independent and are not adjacent to any other vertex of $U$. 
Since $n\geq 9$, there are at least $4$ other vertices of $U$, say $u_5,\ldots,u_8$, 
and $u_1u_5u_2u_6u_3u_7u_4u_8u_1$ and $w_3$ form $W_8$ in $\overline{G}$, a contradiction.

\smallskip
\noindent {\bf Case 2b}: A vertex $u\in U$ has degree at least $n-3$ in $G[U\cup V]$.

Since $S_n(2,2)\nsubseteq G$, no vertex of $V$ is adjacent to $u$ or to $N_{G[U]}(u)$.
Then $u$ is adjacent to at least $n-3$ vertices of $U$ in $G$; 
suppose without loss of generality that $u_1,\ldots,u_{n-3}\subseteq N_{G[U]}(u)$. 
If $n\geq 10$, 
then any $4$ vertices from $N_{G[U]}(u)$, any $4$ vertices from $V$ and $w_3$ form $W_8$ in $\overline{G}$, 
a contradiction. 
Suppose that $n=9$ and let $u_7$ be the vertex in $U\setminus\{u,u_1,\ldots,u_{n-3}\}$. 
If $u_7$ is adjacent in $\overline{G}$ to at least two of $u_1,\ldots,u_6$, say $u_1$ and $u_2$, 
then $u_1u_7u_2v_3u_3v_4u_4v_5u_1$ and $w_3$ form $W_8$ in $\overline{G}$, a contradiction. 
Therefore, $u_7$ is adjacent in $G$ to at least $5$ of the vertices $u_1,\ldots,u_6$, 
say $u_1,\ldots,u_5$. 
Since $S_9(2,2)\nsubseteq G$, 
$U$ is not adjacent in $G$ to $\{v_0,v_1,v_2,w_1\}\cup V$
and $w_2$ is not adjacent to $u$ or $u_7$.
If $w_3$ is not adjacent to some vertex $a\in\{v_0,v_1,w_1,w_2\}$, 
then $uv_3u_1v_4u_2v_5u_7au$ and $w_3$ form $W_8$ in $\overline{G}$, a contradiction. 
Hence, $w_3$ is adjacent to $v_0$, $v_1$, $w_1$ and $w_2$ in~$G$. 
Similarly, $v_2$ is adjacent to $v_1$, $w_1$ and $w_2$. 
Since $S_9(2,2)\nsubseteq G$, $w_2$ is non-adjacent to at least one of $v_3,v_4,v_5$, 
say $v_3$ without loss of generality. 
If $v_1$ is also not adjacent to $v_3$, 
then $uw_2u_7v_1u_1v_2u_2w_3u$ and $w_3$ form $W_8$ in $\overline{G}$, a contradiction. 
Thus, $v_1$ is adjacent to $v_3$, 
then $v_3$ is not adjacent to both $v_4$ and $v_5$, or else $G$ contains $S_9(2,2)$.
Without loss of generality, assume that $v_3$ is not adjacent to $v_4$ in~$G$. 
Then $uw_2u_7v_4u_1v_2u_2w_3u$ and $w_3$ form $W_8$ in $\overline{G}$, a contradiction. 

In either case, $R(S_n(2,2),W_8)\leq 2n-1$ for $n\not\equiv 0 \pmod{4}$, which completes the proof. 
\end{proof}

\begin{theorem}
If $n\geq 9$, then $R(S_n(4,1),W_8)=2n-1$.
\end{theorem}

\begin{proof}
Lemma~\ref{lem:lower-bound-on-S_n(1,4)-T_D(n)-...-T_L(n)}
provides the lower bound, so it remains to prove the upper bound.
Let $G$ be any graph of order $2n-1$.
Assume that $G$ does not contain $S_n(4,1)$ and that $\overline{G}$ does not contain $W_8$.

Suppose first that there is a subset $X\subseteq V(G)$ of size $n$ with $\delta(G[X])\ge n-4$.
Let $x_0$ be any vertex of $X$, and pick a subset $X'\subseteq N_{G[X]}(x_0)$ of size $n-5$.
Set $Y = X\setminus (\{x_0\}\cup X')$, and so $|Y|=4$.
Since $\delta(G[X])\ge n-4$, each vertex of $Y$ is adjacent to at least $n-8$ vertices of $X'$ in $G$ 
and each vertex of $X'$ is adjacent to at least one vertex of $Y$ in~$G$.
Hence, for $n\ge 11$, it is straightforward to see that there is a matching from $Y$ to $X'$ in $G$;
hence, $G$ contains $S_n(4,1)$, a contradiction.

For $n=10$ and $\delta(G[X])\ge n-4=6$, let $X=\{x_0,\ldots,x_9\}$ and $\{x_1,\ldots,x_6\}\subseteq N_{G[X]}(x_0)$. 
Since $\delta(G[X])\ge 6$, vertices $x_7$, $x_8$ and $x_9$ must each be adjacent to at least $3$ vertices of $x_1,\ldots,x_6$. 
It is straightforward to see that there is a matching from $\{x_7,x_8,x_9\}$ to $\{x_1,\ldots,x_6\}$ in $G$;
without loss of generality, assume that $x_i$ is adjacent to $x_{i+6}$ in $G$ for $i=1,2,3$. 
Now, if there is any edge in $G[\{x_4,x_5,x_6\}]$, 
then $S_{10}(4,1)\subseteq G$, a contradiction. 
Otherwise, $G[\{x_4,x_5,x_6\}]$ must be independent 
and each of $x_4,x_5,x_6$ must be adjacent to at least two vertices of $x_7,x_8,x_9$ in $G$. 
Without loss of generality, 
assume that $x_4$ is adjacent to $x_7$ and $x_8$ in $G$. 
Since $S_{10}(4,1)\nsubseteq G$, 
$x_5$ cannot be adjacent to $x_1$ and $x_2$ in $G$, 
but this is impossible since $\delta(G[X])\ge 6$. 

Now for $n=9$, suppose that $d_{G[X]}(x_0) = n-4 = 5$.
Let $N_{G[X]}(x_0) = \{x_1,\ldots,x_5\}$ and $Y = \{x_6,x_7,x_8\}$.
Then, three vertices of $Y$ are each adjacent to at least $n-6=3$ vertices of $N_{G[X]}(x_0)$ in~$G$. 
Without loss of generality, assume that 
$x_1$ is adjacent to $x_6$, $x_2$ is adjacent to $x_7$ and $x_3$ is adjacent to $x_8$, respectively. 
Now, if $x_4$ is adjacent to $x_5$, then $G$ contains $S_9(4,1)$, a contradiction. 
Otherwise, $x_4$ and $x_5$ must each be adjacent to at least one of $x_6$, $x_7$ and $x_8$.
Assume that $x_4$ is adjacent to $x_6$. 
Then $x_5$ is not adjacent to $x_1$ and $x_4$ in $G$, or else $G$ contains $S_9(4,1)$.
If $x_5$ is adjacent to $x_6$, 
then $x_1,x_4,x_5$ must be independent in $G$, 
and they are each adjacent to $x_7$ or $x_8$ in~$G$; 
assume that $x_1$ is adjacent to $x_7$. 
Then, $x_4$ and $x_5$ are not adjacent to $x_2$ in $G$, and since $\delta(G[X])\ge 5$, 
they are adjacent to $x_7$ and $x_8$ in $G$, and $G$ contains $S_9(4,1)$, a contradiction. 
If $x_5$ is not adjacent to $x_6$, then since $d_{G[X]}(v_0) \ge 5$, 
$x_5$ is adjacent to $x_2$, $x_3$, $x_7$ and $x_8$ in~$G$. 
Then, $x_4$ is not adjacent to $x_2$ and $x_3$ in $G$, 
and $x_4$ is adjacent to $x_1$, $x_6$, $x_7$ and $x_8$ in $G$, 
and this gives us $S_9(4,1)$ in $G$, a contradiction. 
As $x_0$ was arbitrary, assume for the case when $n=9$ that $\delta(G[X])\ge n-3 = 6$, 
which again leads to the contradiction that $G$ contains $S_9(4,1)$.

Now assume that $\delta(G[X])\le n-5$ whenever $X\subseteq V(G)$ is of size $n$.
Recall that $G$ has order $2n-1$, and so by Theorem~\ref{thm:R(Sn(3,1),W8)}, 
$G$ has a subgraph $S_n(3,1)$ and thus a subgraph $T=S_{n-1}(3,1)$. 
Let $V(T)=\{v_0,\ldots,v_{n-5},w_1,w_2,w_3\}$ 
and $E(T)=\{v_0v_1,\ldots,v_0v_{n-5},v_1w_1,v_2w_2,v_3w_3\}$. 
Set $V=\{v_4,\ldots,v_{n-5}\}$ and $U=V(G)-V(T)=\{u_1,\ldots,u_n\}$; 
then $|V|=n-8$ and $|U|=n$. 
Since $S_n(4,1)\nsubseteq G$, 
$V$ is not adjacent to any vertex of $U$ in~$G$. 
Now as $\delta(G[U])\le n-5$, $\overline{G}[U]$ contains $S_5$,
and so for $n\ge 12$, $\overline{G}$ contains $W_8$ by Observation~\ref{obs2}, a contradiction.

Suppose that $n=11$. 
If $v_0$ is not adjacent to any vertex of $U$ in $G$,
then $\overline{G}$ contains $W_8$ by Observation~\ref{obs2}, a contradiction.
Assume that $v_0$ is adjacent to some vertex $u\in U$. 
Since $S_{11}(4,1)\nsubseteq G$, 
$G[V\cup\{u\}]$ is an empty graph and $u$ is not adjacent to any vertex of $U$ in~$G$.
By Lemma~\ref{lm4:R(Sn(1,2),W8)}, 
$G[U\setminus\{u\}]$ is $K_{10}$ or $K_{10} - e$, 
so no vertex of $V(T)\cup \{u\}$ is adjacent to any vertex of $U\setminus\{u\}$ in $G$, 
as $S_{11}(4,1)\nsubseteq G$.
Since $\delta(G[V(T)\cup \{u\}])\le n-5$, 
$\overline{G}[V(T)\cup \{u\}]$ contains $S_5$, 
so $\overline{G}$ contains $W_8$ by Observation~\ref{obs2}, a contradiction.

Now, suppose that $n=10$. 
Then $G$ has order $19$, 
and by Theorem~\ref{thm:R(Sn(4)-Sn[4]-Sn(1,3)-TA-TB-TC-Sn(3,1),W8)}, 
$G$ has a subgraph $T'=S_{10}(3,1)$. 
Let $V(T')=\{v_0',\ldots,v_6',w_1',w_2',w_3'\}$ 
and $E(T')=\{v_0'v_1',\ldots,v_0'v_6',v_1'w_1',v_2'w_2',v_3'w_3'\}$. 
Set $V'=\{v_4',v_5',v_6'\}$ and $U'=V(G)-V(T')=\{u_1',\ldots,u_9'\}$. 
Since $S_{10}(4,1)\nsubseteq G$, $V'$ must be independent in $G$ and is not adjacent to any vertex of $U'$ in $G$. 
If $v_0'$ is adjacent to some vertices in $U'$ in $G$, say $u_1'$.
Since $S_{10}(4,1)\nsubseteq G$, $u_1'$ is not adjacent to any vertex of $V'$ or $U'\setminus \{u_1'\}$ in $G$. 
Then, by Lemma~\ref{lm4:R(Sn(1,2),W8)}, 
$G[U'\setminus\{u_1'\}]$ is $K_8$ or $K_8 - e$, 
so no vertex of $V(T')$ is adjacent to any vertex of $U'\setminus\{u_1'\}$ in $G$, 
as $S_{10}(4,1)\nsubseteq G$.
Since $\delta(G[V(T')])\le 5$, 
$\overline{G}[V(T')]$ contains $S_5$, 
so $\overline{G}$ contains $W_8$ by Observation~\ref{obs2}, a contradiction. 
Now, suppose that $v_0'$ is not adjacent to any vertex of $U'$ in $G$. 
Note that $|U'\cup \{w_1'\}|=n$; 
therefore, $\delta(G[U'\cup \{w_1'\}])\le 5$, 
and so $\overline{G}[U'\cup \{w_1'\}]$ contains $S_5$. 
If $w_1'$ is not adjacent to any vertex from $V'\cup \{v_0'\}$, 
then by Observation~\ref{obs2}, $\overline{G}$ contains $W_8$, a contradiction. 
Otherwise, there are two cases to be considered. 

\smallskip
\noindent {\bf Case 1a}: $w_1'$ is adjacent to some vertices of $V'$ in~$G$.

Without loss of generality, 
assume that $w_1'$ is adjacent to $v_4'$ in~$G$. 
In this case, $v_1'$ is not adjacent to $U'\cup \{v_5',v_6'\}$. 
Then by Lemma~\ref{lm4:R(Sn(1,2),W8)}, 
$G[U']$ is $K_9$ or $K_9 - e$, 
so no vertex of $V(T')$ is adjacent to 
any vertex of $U'$ in $G$, as $S_{10}(4,1)\nsubseteq G$.
Since $\delta(G[V(T')])\le 5$, $\overline{G}[V(T')]$ contains $S_5$, 
and so $\overline{G}$ contains $W_8$ by Observation~\ref{obs2}, a contradiction.

\smallskip
\noindent {\bf Case 1b}: $w_1'$ is non-adjacent to each vertex of $V'$ in~$G$.

In this case, $w_1'$ is adjacent to $v_0'$ in $G$. 
Note that $w_1'$ is not adjacent to $U'$, 
since this would revert to the case where $v_0'$ is adjacent to some vertex of $U'$. 
Then again by Lemma~\ref{lm4:R(Sn(1,2),W8)}, 
$G[U']$ is $K_9$ or $K_9 - e$, 
so no vertex of $V(T')$ is adjacent to 
any vertex of $U'$ in $G$, as $S_{10}(4,1)\nsubseteq G$.
Since $\delta(G[V(T')])\le 5$, 
$\overline{G}[V(T')]$ contains $S_5$, 
and so $\overline{G}$ contains $W_8$ by Observation~\ref{obs2}, a contradiction. 

\smallskip

Finally, suppose that $n=9$. 
Then $G$ has order $17$, 
and so $G$ has a subgraph $T'=S_9(2,1)$ by Theorem~\ref{thm:R(Pn-Sn(1,2)-Sn(3)-Sn(2,1),W8)}.
Let $V(T')=\{v_0',\ldots,v_6',w_1',w_2'\}$ 
and $E(T')=\{v_0'v_1',\ldots,v_0'v_6',v_1'w_1',v_2'w_2'\}$.
Set $V'=\{v_3',\ldots,v_6'\}$ and $U'=V(G)-V(T')=\{u_1',\ldots,u_8'\}$.

Now, suppose that $E_G(V',U')\neq \emptyset$.
Without loss of generality, assume that $v_3'$ is adjacent to $u_1'$ in~$G$. 
Since $S_9(4,1)\nsubseteq G$, 
$v_4',v_5',v_6'$ are independent 
and not adjacent to any vertex of $U'\setminus \{u_1'\}$ in~$G$. 

Suppose that $v_0'$ is adjacent to some vertex of $U'\setminus \{u_1'\}$, say $u_2'$. 
Then $u_2'$ is non-adjacent to $\{v_4',v_5',v_6'\}\cup U'\setminus \{u_1',u_2'\}$ in~$G$. 
Since $\delta(G[\{w_1',w_2'\}\cup U'\setminus \{u_2'\}])\le n-5$, 
  $\overline{G}[\{w_1',w_2'\}\cup U'\setminus \{u_2'\}]$ contains $S_5$.
If $v_4'$, $v_5'$, $v_6'$ and $u_2'$ are not adjacent to $w_1'$, $w_2'$ or $u_1'$ in $G$, 
then $\overline{G}$ contains $W_8$
by Observation~\ref{obs2}, a contradiction. 
Assume that $v_4'$ is adjacent to $w_1'$ in~$G$. 
In this case, $v_1'$ is not adjacent to $\{v_5',v_6'\}\cup U'\setminus \{u_1'\}$ in $G$, 
and $v_1'u_3'v_4'u_4'v_6'u_7'u_2'u_8'v_1'$ and $v_5'$ form $W_8$ in $\overline{G}$, 
a contradiction.
Similar contradictions occur if we assume that $v_5'$, $v_6'$ or $u_2'$ are adjacent to $w_1'$, $w_2'$ or $u_1'$ in $G$.

Thus, $v_0'$ is not adjacent to any vertex of $U'\setminus \{u_1'\}$ in~$G$. 
Since $\delta(G[\{w_1',w_2'\}\cup U'\setminus \{u_1'\}])\le n-5$, 
  $\overline{G}[\{w_1',w_2'\}\cup U'\setminus \{u_1'\}]$ contains $S_5$.
If $v_0'$, $v_4'$, $v_5'$ and $v_6'$ are not adjacent to $w_1'$ or $w_2'$ in $G$, 
then $\overline{G}$ contains $W_8$ by Observation~\ref{obs2}, a contradiction. 
There are two cases to be considered.

\smallskip
\noindent {\bf Case 2a}: $v_0'$ is adjacent to $w_1'$ or $w_2'$ in~$G$.

Without loss of generality, 
assume that $v_0'$ is adjacent to $w_1'$ in~$G$. 
Note that $v_1'$ and $w_1'$ are not adjacent to $U'\setminus \{u_1'\}$, 
since this would revert to the case where $v_0'$ is adjacent to some vertex of $U'\setminus \{u_1'\}$. 
Again, since $\delta(G[\{w_2'\}\cup U'])\le n-5$, 
$\overline{G}[\{w_2'\}\cup U'\}]$ contains $S_5$. 
If $v_1'$, $v_4'$, $v_5'$ and $v_6'$ are not adjacent to $w_2'$ and $u_1'$ in $G$, 
then $\overline{G}$ contains $W_8$ by Observation~\ref{obs2}, 
a contradiction. 

Suppose that $v_1'$ is adjacent to $w_2'$ or $u_1'$, say $w_2'$, in~$G$. 
If $w_1'$ is not adjacent to $v_4'$, $v_5'$ or $v_6'$, 
then by Lemma~\ref{lm4:R(Sn(1,2),W8)}, 
$G[U'\setminus\{u_1'\}]$ is $K_7$ or $K_7 - e$, 
so no vertex of $V(T')\cup \{u_1'\}$ is adjacent to 
any vertex of $U'\setminus\{u_1'\}$ in $G$, as $S_9(4,1)\nsubseteq G$.
Since $\delta(G[V(T')])\le n-5$, $\overline{G}[V(T')]$ contains $S_5$, 
and so $\overline{G}$ contains $W_8$ by Observation~\ref{obs2}, a contradiction. 
Otherwise, $w_1'$ is adjacent to at least one of $v_4',v_5',v_6'$ in $G$, say $v_4'$. 
Then, $v_2'$ is not adjacent to $\{v_5',v_6'\}\cup U'\setminus \{u_1'\}$, 
since $G$ does not contain $S_9(4,1)$. 
Similarly, by Lemma~\ref{lm4:R(Sn(1,2),W8)}, 
$G[U'\setminus\{u_1'\}]$ is $K_7$ or $K_7 - e$, 
so no vertex of $V(T')\cup \{u_1'\}$ is adjacent to 
any vertex of $U'\setminus\{u_1'\}$ in $G$, as $S_9(4,1)\nsubseteq G$.
Again, since $\delta(G[V(T')])\le n-5$, $\overline{G}[V(T')]$ contains $S_5$, 
and so $\overline{G}$ contains $W_8$ by Observation~\ref{obs2}, a contradiction. 

Now suppose that $v_1'$ is non-adjacent to both $w_2'$ and $u_1'$ in~$G$. 
Then, one of $v_4',v_5',v_6'$ is adjacent to $w_2'$ or $u_1'$ in~$G$.
Without loss of generality, assume that $v_4'$ is adjacent to $w_2'$ in~$G$. 
In this case, $v_2'$ is not adjacent to $\{v_5',v_6'\}\cup U'\setminus \{u_1'\}$.
Then, again, by Lemma~\ref{lm4:R(Sn(1,2),W8)}, 
$G[U'\setminus\{u_1'\}]$ is $K_7$ or $K_7 - e$, 
so no vertex of $V(T')\cup \{u_1'\}$ is adjacent to 
any vertex of $U'\setminus\{u_1'\}$ in $G$, as $S_9(4,1)\nsubseteq G$.
Since $\delta(G[V(T')])\le n-5$, $\overline{G}[V(T')]$ contains $S_5$, 
and so $\overline{G}$ contains $W_8$ by Observation~\ref{obs2}, a contradiction.

\smallskip
\noindent {\bf Case 2b}: $v_0'$ is non-adjacent to both $w_1'$ and $w_2'$ in~$G$.

In this case, one of $v_4',v_5',v_6'$ is adjacent to $w_1'$ or $w_2'$ in $G$, 
say $v_4'$ to $w_1'$ in~$G$. 
Since $S_9(4,1)\nsubseteq G$, 
$v_1'$ is not adjacent to $\{v_5',v_6'\}\cup U'\setminus \{u_1'\}$ in~$G$. 
By Lemma~\ref{lm4:R(Sn(1,2),W8)}, 
$G[U'\setminus\{u_1'\}]$ is $K_7$ or $K_7 - e$, 
so no vertex of $V(T')\cup \{u_1'\}$ is adjacent to 
any vertex of $U\setminus\{u_1'\}$ in $G$, 
as $S_9(4,1)\nsubseteq G$.
Since $\delta(G[V(T')])\le n-5$, $\overline{G}[V(T')]$ contains $S_5$, 
and so $\overline{G}$ contains $W_8$ by Observation~\ref{obs2}, a contradiction. 

\smallskip

Now suppose that $E_G(V',U')= \emptyset$. 
If $\delta(G[V'])=0$, then by Lemma~\ref{lm4:R(Sn(1,2),W8)}, 
$G[U']$ is $K_8$ or $K_8 - e$, 
and no vertex of $V(T')$ is adjacent to any vertex of $U'$ in $G$, 
as $S_9(4,1)\nsubseteq G$. 
Since $\delta(G[V(T')])\le n-5$, $\overline{G}[V(T')]$ contains $S_5$, 
and so $\overline{G}$ contains $W_8$ by Observation~\ref{obs2}, a contradiction. 
Hence, $\delta(G[V'])\ge 1$, and since $S_9(4,1)\nsubseteq G$, 
one of the vertices in $V'$ is adjacent to other three in~$G$. 
Without loss of generality, 
assume that $v_3'$ is adjacent to $v_4'$, $v_5'$ and $v_6'$ in~$G$. 
Since $G$ does not contain $S_9(4,1)$, $v_4',v_5',v_6'$ are independent in~$G$. 
Furthermore, $v_0'$ is not adjacent to $U'$ in $G$ or else this reverts to the case 
where $v_3'$ is adjacent to $u_1'$ and $v_0'$ is adjacent to any vertex of $U'\setminus \{u_1'\}$. 
Since $\delta(G[\{w_1'\}\cup U'])\le n-5$, 
$\overline{G}[\{w_1'\}\cup U']$ contains $S_5$.
If $v_0'$, $v_4'$, $v_5'$ and $v_6'$ are non-adjacent to $w_1'$ in $G$,
then $\overline{G}$ contains $W_8$ by Observation~\ref{obs2}, a contradiction. 
Again, there are two cases to be considered.

\smallskip
\noindent {\bf Case 3a}: $v_0'$ is adjacent to $w_1'$ in~$G$.

Note that $v_1'$ and $w_1'$ are not adjacent to $U'$, 
or else this reverts to the case where $v_3'$ is adjacent to $u_1'$ 
and $v_0'$ is adjacent to any vertex of $U'\setminus \{u_1'\}$. 
Now, since $\delta(G[\{w_2'\}\cup U'])\le n-5$, 
$\overline{G}[\{w_2'\}\cup U'\}]$ contains $S_5$.
If $v_0'$, $v_4'$, $v_5'$ and $v_6'$ are non-adjacent to $w_2'$ in $G$, 
then $\overline{G}$ contains $W_8$ by Observation~\ref{obs2}, a contradiction. 

Suppose that $v_0'$ is adjacent to $w_2'$ in~$G$. 
Again, $v_2'$ and $w_2'$ are non-adjacent to $U'$, 
or else else this reverts to the case where $v_3'$ is adjacent to $u_1'$ 
and $v_0'$ is adjacent to any vertex of $U'\setminus \{u_1'\}$. 
Now, $E_G(V(T'),U')= \emptyset$, and since $\delta(G[V(T')])\le n-5$, 
$\overline{G}[V(T')]$ contains $S_5$, 
and so $\overline{G}$ contains $W_8$ by Observation~\ref{obs2}, a contradiction. 

Therefore, $w_2'$ is adjacent to at least one of $v_4'$, $v_5'$ and $v_6'$ in $G$,
say $v_4'$. 
Then, $v_2'$ is not adjacent to $v_5'$, $v_6'$ or $U'$, as $S_9(4,1)\nsubseteq G$, 
a contradiction. 
By Lemma~\ref{lm4:R(Sn(1,2),W8)}, $G[U']$ is $K_8$ or $K_8 - e$, 
so no vertex of $V(T')$ is adjacent to any vertex of $U'$ in $G$, as $S_9(4,1)\nsubseteq G$.
Again, since $\delta(G[V(T')])\le n-5$, $\overline{G}[V(T')]$ contains $S_5$, 
and so $\overline{G}$ contains $W_8$ by Observation~\ref{obs2}, a contradiction.

\smallskip
\noindent {\bf Case 3b}: $v_0'$ is not adjacent to $w_1'$ in~$G$.

In this case, one of $v_4',v_5',v_6'$ is adjacent to $w_1'$ in $G$, say $v_4'$.
Since $S_9(4,1)\nsubseteq G$, $v_1'$ is not adjacent to $v_5'$, $v_6'$ or $U'$ in~$G$. 
By Lemma~\ref{lm4:R(Sn(1,2),W8)}, $G[U']$ is $K_8$ or $K_8 - e$, 
so no vertex of $V(T')\cup \{u_1'\}$ is adjacent to any vertex of $U'$ in $G$, 
as $S_9(4,1)\nsubseteq G$.
Since $\delta(G[V(T')])\le n-5$, $\overline{G}[V(T')]$ contains $S_5$, 
and so $\overline{G}$ contains $W_8$ by Observation~\ref{obs2}, a contradiction. 

\smallskip

Thus, $R(S_n(4,1),W_8)\leq 2n-1$ for $n\geq 9$ which completes the proof.
\end{proof}

\begin{theorem}
If $n\geq 8$, then 
\[
  R(T_D(n),W_8)=\begin{cases}
    2n-1 & \text{if $n\not\equiv 0 \pmod{4}$}\,;\\
    2n   & \text{otherwise}\,.
  \end{cases}
\]  
\end{theorem}

\begin{proof}
Lemma~\ref{lem:lower-bound-on-S_n(1,4)-T_D(n)-...-T_L(n)}
provides the lower bound, so it remains to prove the upper bound.
Let $G$ be a graph with no $T_D(n)$ subgraph whose complement $\overline{G}$ does not contain $W_8$. 
Suppose that $n\equiv 0 \pmod{4}$ and that $G$ has order~$2n$. 
By Theorem~\ref{thm:R(Sn[4],W8)}, $G$ has a subgraph $T=S_n[4]$. 
Let $V(T)=\{v_0,\ldots,v_{n-4},w_1,w_2,w_3\}$ 
and $E(T)=\{v_0v_1,\ldots,v_0v_{n-4},v_1w_1,w_1w_2,w_1w_3\}$. 
Set $V=\{v_2,\ldots,v_{n-4}\}$ and $U=V(G)-V(T)$; 
then $|V|=n-5$ and $|U|=n$. 
Since $T_D(n)\nsubseteq G$, neither $w_2$ nor $w_3$ is adjacent in $G$ to $U\cup V$. 

Suppose that $n=8$.
Since $G$ does not contain $T_D(n)$, $V$ must be independent and non-adjacent to $U$ in~$G$. 
Then for any vertices $u_1,\ldots,u_4$ in $U$, 
$v_3u_1v_4u_2w_2u_3w_3u_4v_3$ and $v_2$ form $W_8$ in $\overline{G}$, a contradiction. 
Suppose that  $n\geq 12$. 
Then $|U\cup V|=2n-5$. 
If $\delta(\overline{G}[U\cup V])\geq \lceil\frac{2n-5}{2}\rceil$, 
then $\overline{G}[U\cup V]$ contains $C_8$ by Lemma~\ref{lem:pancyclic} which with $w_2$ forms $W_8$,
a contradiction. 
Thus, $\delta(\overline{G}[U\cup V])\leq \lceil\frac{2n-5}{2}\rceil-1=n-3$, 
and $\Delta(G[U\cup V])\geq n-3$. 
Now, there are two cases to consider. 

\smallskip
\noindent {\bf Case 1}: One of the vertices of $V$, say $v_2$, is a vertex of degree at least $n-3$ in $G[U\cup V]$. 

Since $T_D(n)\nsubseteq G$, 
$v_1$ is not adjacent in $G$ to $w_2$, $w_3$ or $U\cup V\setminus \{v_2\}$. 
Let $U'=\{w_2,w_3\}\cup U\cup V\setminus \{v_2\}$; then $|U'|=2n-4$. 
Now, if $\delta(\overline{G}[U'])\geq \frac{2n-4}{2}=n-2$, 
then $\overline{G}[U']$ contains $C_8$ by Lemma~\ref{lem:pancyclic} which with $v_1$ forms $W_8$, 
a contradiction. 
Hence, $\delta(\overline{G}[U'])\leq n-3$, and $\Delta(G[U'])\geq n-2$. 
Note that neither $w_2$ nor $w_3$ have degree $\Delta(G[U'])$. 
Therefore, $d_{G[U']}(u')\geq n-2$ for some vertex $u'\in U\cup V\setminus \{v_2\}$. 
By the Inclusion-Exclusion Principle, 
some vertex $a\in U\cup V\setminus \{v_2\}$ is adjacent in $G$ to both $u'$ and $v_2$.
Then $G$ has a subgraph $T_D(n)$ in which $u'$ is the vertex of degree $n-5$ and $v_2$ is the vertex of degree $3$, a contradiction. 

\smallskip
\noindent {\bf Case 2}: Some vertex $u\in U$ has degree at least $n-3$ in $G[U\cup V]$.

Suppose that there is at least one vertex in $V$ that is adjacent to $u$ in $G$, say $v_2$. 
Then $G$ has a subgraph $T_D(n)$ in which $u$ is the vertex of degree $n-5$ 
and $v_0$ is the vertex of degree~$3$, a contradiction. 
Similarly, no other vertex of $V$ is adjacent to $u$. 
Now, since $T_D(n)\nsubseteq G$, 
$d_{G[N_{G[U]}(u)\cup \{v\}]}(v)\leq 1$ 
and $d_{G[V\cup \{x\}]}(x)\leq 1$, 
for any $v\in V$ and $x\in N_{G[U]}(u)$. 
Then, by Lemma \ref{lem:1-1relation}, 
$\overline{G}[V\cup N_{G[U]}(u)]$ must contain $C_8$, 
which with $w_2$ as hub, forms $W_8$ in $\overline{G}$, a contradiction. 

Now, suppose that $n\not\equiv 0 \pmod{4}$ and that $G$ has order $2n-1$. 
By Theorem~\ref{thm:R(Sn[4],W8)}, $G$ has a subgraph $T=S_n[4]$. 
Let $V(T)=\{v_0,\ldots,v_{n-4},w_1,w_2,w_3\}$ 
and $E(T)=\{v_0v_1,\ldots,v_0v_{n-4},v_1w_1,w_1w_2,w_1w_3\}$. 
Set $V=\{v_2,\ldots,v_{n-4}\}$ and $U=V(G)-V(T)$; 
then $|V|=n-5$ and $|U|=n-1$. 
Since $T_D(n)\nsubseteq G$, 
neither $w_2$ nor $w_3$ is adjacent to $U\cup V$ in~$G$. 
If $\delta(\overline{G}[U\cup V])\geq \frac{2n-6}{2}=n-3$,
then $\overline{G}[U\cup V]$ contains $C_8$ by Lemma~\ref{lem:pancyclic} which with $w_2$ forms $W_8$ in $\overline{G}$, a contradiction. 
Thus, $\delta(\overline{G}[U\cup V])\leq n-4$, 
and $\Delta(G[U\cup V])\geq n-3$. 
The arguments of the preceding cases then lead to contradictions. 

Thus, $R(T_D(n),W_8)\leq 2n$, which completes the proof.
\end{proof}

\begin{lemma}\label{lm:R(TF(n),W8)} 
Each graph $H$ of order $n\geq 8$ with minimal degree at least $n-4$
contains $T_E(n)$ unless $n=8$ and $H=K_{4,4}$.
\end{lemma}

\begin{proof}
Let $V(H)=\{u_0,\ldots,u_{n-1}\}$. 
First, suppose that $\Delta(H)\geq n-3$ 
and assume without loss of generality that $u_1,\ldots,u_{n-3}\in N_H(u_0)$. 
Suppose that $u_{n-2}$ and $u_{n-1}$ are adjacent in~$H$. 
Since $\delta(H)\geq n-4$, 
$N_H(u_0)\cap N_H(u_{n-2})\neq \emptyset$, 
so assume without loss of generality that $u_1$ is adjacent to $u_{n-2}$ in $H$. 
Furthermore, $u_1$ must be adjacent to at least $n-7$ vertices from $\{u_2,\ldots,u_{n-3}\}$ in $H$. 
Without loss of generality, assume that $u_1$ is adjacent to $u_2,\ldots,u_{n-6}$ in $H$. 
Now, if any vertex of $\{u_2,\ldots,u_{n-6}\}$ is adjacent to $u_{n-5}$, $u_{n-4}$ or $u_{n-3}$ in $H$, then we have $T_E(n)$ in $H$. 
Suppose that is not the case; 
then each vertex of $\{u_2,\ldots,u_{n-6}\}$ must be adjacent to each other and to $u_0$, $u_1$, $u_{n-2}$ and $u_{n-1}$ in $H$. 
Since $d_H(u_{n-3})\ge n-4$,
$u_{n-3}$ is adjacent to at least one of $u_1$, $u_{n-2}$ and $u_{n-1}$ in $H$, 
so $H$ contains $T_E(n)$, a contradiction.

Suppose that $u_{n-2}$ is not adjacent to $u_{n-1}$ in $H$. 
Since $\delta(H)\geq n-4$, 
$u_{n-2}$ and $u_{n-1}$ are each adjacent to at least $n-5$ vertices in $N_H(u_0)$, 
so at least one vertex of $N_H(u_0)$, say $u_1$, is adjacent in $H$ to both $u_{n-2}$ and $u_{n-1}$. 
If $H[\{u_2,\ldots,u_{n-3}\}]$ contains subgraph $2K_2$, then H contains subgraph $T_E(n)$. Note that this will always happens for $n\ge 11$, since $\delta(H)\ge n-4$. 

Suppose that $n=10$.
Since $\delta(H)\ge 6$, $u_2$ must be adjacent in $H$ to at least two vertices of $u_3,\ldots,u_7$, without loss of generality say $u_3$ and $u_4$. 
If $H[\{u_4,\ldots,u_7\}]$ contains any edge, then $H$ contains $T_E(10)$. 
Otherwise, $\{u_4,\ldots,u_7\}$ must be independent in $H$ 
and each of these vertices must be adjacent to $u_0$, $u_1$, $u_2$, $u_3$, $u_8$ and $u_9$; 
this also gives a subgraph $T_E(10)$ in $H$. 

Similarly, for $n=9$, $u_2$ must be adjacent to at least one of $u_3,\ldots,u_6$, say $u_3$, in $H$. 
If $H[\{u_4,u_5,u_6\}]$ contains any edge, then $H$ contains $T_E(9)$. 
Otherwise, $\{u_4,u_5,u_6\}$ is independent in $H$ and since $\delta(H)\ge 5$, 
$u_4$ is adjacent to at least one of $u_2$ and $u_3$, and 
$u_5$ is adjacent to at least one of $u_7$ and $u_8$. 
Again, this gives a subgraph $T_E(9)$ in $H$. 

For $n=8$, if $u_2,\ldots,u_5$ are independent in $H$, 
then they are each adjacent to $u_0$, $u_1$, $u_6$ and $u_7$ in $H$,
which gives $T_E(8)$ in $H$. 
Otherwise, we can assume that $u_2$ is adjacent to $u_3$ in $H$. 
If $u_4$ is adjacent to $u_5$ in $H$, we will have $T_E(8)$ in $H$; 
otherwise, assume that $u_4$ is not adjacent to $u_5$. 
Now, suppose that $u_4$ is adjacent to $u_2$ or $u_3$ in $H$. 
If $u_5$ is adjacent to $u_6$ or $u_7$ in $H$, then $H$ contains $T_E(8)$. 
Otherwise, $u_5$ must be adjacent to $u_0$, $u_1$, $u_2$ and $u_3$ since $\delta(H)\ge 4$. 
However, this also gives $T_E(8)$ in $H$. 
On the other hand, suppose that $u_4$ is adjacent to neither $u_2$ nor $u_3$ in $H$. 
Similarly, $u_5$ is not adjacent to $u_2$ or to $u_3$ in $H$. 
Since $\delta(H)\ge 4$, both $u_4$ and $u_5$ must be adjacent to $u_0$, $u_1$, $u_6$ and $u_7$ in $H$, and this also gives $T_E(8)$ in $H$. 

Suppose that $H$ is $(n-4)$-regular and that $N_H(u_0)=\{u_1,\ldots,u_{n-4}\}$. 
By the Handshaking Lemma, this only happens when $n$ is even. 

Suppose that $n\geq 10$.
Note that $u_{n-3}$, $u_{n-2}$ and $u_{n-1}$ are each adjacent to at least $n-6$ vertices of $N_H(u_0)$ in $H$. 
By the Inclusion-Exclusion Principle, 
at least one of $u_1,\ldots,u_{n-4}$ is adjacent to two of $u_{n-3},u_{n-2},u_{n-1}$ in $H$,
say $u_1$ to $u_{n-3}$ and $u_{n-2}$, 
and there must be another vertex, say $u_2$, that is adjacent to $u_{n-1}$ in $H$. 
Now, if there is any edge in $H[\{u_3,\ldots,u_{n-4}\}]$, 
then $T_E(n)\subseteq H$, and this always happens for $n\ge 12$. 
For $n=10$, since $d_H(u_1)=6$, $u_1$ is non-adjacent in $H$
to at least one of $u_3,\ldots,u_6$, say $u_3$. 
Since $d_H(u_3)=6$, $u_3$ is adjacent to one of $u_4,u_5,u_6$, 
giving $T_E(10)$ in $H$. 

Now suppose that $n=8$.
If $u_5$, $u_6$ and $u_7$ are independent in $H$, then $H=K_{4,4}$. 
Otherwise, we can assume that $u_5$ is adjacent to $u_6$ in $H$. 
If $u_5$ is also adjacent to $u_7$ in $H$, 
then $u_5$ is adjacent in $H$ to two vertices of $N_H(u_0)$, say $u_1$ and~$u_2$. 
Suppose that $u_6$ is adjacent to $u_1$ or $u_2$, say $u_1$, in $H$. 
Since $d_H(u_6)=4$, $u_6$ is also adjacent to at least one of $u_2,u_3,u_4,u_7$, so $T_E(8)\subseteq H$. 
Otherwise, suppose that neither $u_6$ nor $u_7$ is adjacent to $u_1$ or $u_2$ in $H$. 
Since $H$ is a $4$-regular graph, 
$u_6$ and $u_7$ are both adjacent to $u_3$ and $u_4$ in $H$, 
and $u_1$ is adjacent to at least one of $u_3$ and $u_4$ in $H$. 
This gives $T_E(8)$ in $H$. 
On the other hand, suppose that $u_5$ is not adjacent to $u_7$ in $H$. 
Then, similarly, $u_6$ is not adjacent to $u_7$ in $H$, 
so $u_7$ is adjacent to $u_1$, $u_2$, $u_3$ and $u_4$ in $H$, and $H$ contains $T_E(8)$. 
\end{proof}

\begin{theorem}
For $n\geq 8$,  
\[
  R(T_E(n),W_8)=\begin{cases}
    2n-1 & \text{if $n\geq 9$}\,;\\
    16   & \text{if $n=8$}\,.
  \end{cases}
\]  
\end{theorem}

\begin{proof}
Lemma~\ref{lem:lower-bound-on-S_n(1,4)-T_D(n)-...-T_L(n)}
provides the lower bound, so it remains to prove the upper bound.
Let $G$ be any graph of order $2n-1$ if $n\ge 9$ and of order $16$ if $n = 8$.
Assume that $G$ does not contain $T_E(n)$ and that $\overline{G}$ does not contain $W_8$.

By Theorem~\ref{thm:R(Sn(3,1),W8)}, $G$ has a subgraph $T=S_n(3,1)$.
Let $V(T)=\{v_0,\ldots,v_{n-4},w_1,w_2,w_3\}$
and $E(T)=\{v_0v_1,\ldots,v_0v_{n-4},v_1w_1,v_2w_2,v_3w_3\}$.
Set $V=\{v_4,\ldots,v_{n-4}\}$ and $U=V(G)-V(T)$.
Then $|V|=n-7$ and $|U|\ge n-1$.
Since $T_E(n)\nsubseteq G$, 
each of $v_1,v_2,v_3$ is not adjacent to any vertex of $V\cup U$ in $G$, 
and each vertex of $V$ is adjacent to at most one vertex of $U$ in~$G$.
Let $W$ be a set of $n-2$ vertices of $U$ that are not adjacent to $v_4$ in~$G$.
By Lemma~\ref{lm4:R(Sn(1,2),W8)}, $G[W]$ is $K_{n-2}$ or $K_{n-2}-e$.
Since $T_E(n)\nsubseteq G$, 
no vertex of $T$ is adjacent to any vertex of $W$, 
and so $\delta(G[V(T)])\geq n-4$ by Observation~\ref{obs2}.

Lemma~\ref{lm:R(TF(n),W8)} implies that 
$G[V(T)]$ contains $T_E(n)$ if $n\ge 9$, a contradiction, 
and so $n=8$ and $G[V(T)] = K_{4,4}$.
Note that $|U| = 8$, and as $T_E(8)\nsubseteq G$, 
no vertex of $U$ is adjacent to any vertex of $G[V(T)]$.
By Lemma~\ref{lm4:R(Sn(1,2),W8)}, 
$G[U]$ is $K_8$ or $K_8-e$, and thus contains $T_E(8)$, a contradiction.

Therefore, $R(T_E(n),W_8)\leq 2n-1$ when $n\geq 9$ and $R(T_E(n),W_8)\leq 16$ when $n=8$.
\end{proof}

\begin{lemma}\label{lm:R(TG(n),W8)} 
Each graph $H$ of order $n\geq 8$ with minimal degree at least $n-4$
contains $T_F(n)$ unless $n=8$ and $H=K_{4,4}$.
\end{lemma}

\begin{proof}
Let $V(H) = \{u_0, u_1\ldots, u_{n-1}\}$ so that 
$d(u_0) = \delta(H)$ and $V = \{u_1,\ldots, u_{n-4}\}\subseteq N(u_0)$.
Set $U = \{u_{n-3}, u_{n-2}, u_{n-1}\}$.
By the minimum degree condition, every vertex of $U$ is adjacent to at least $n-6$ vertices of $V$.
It is straightforward to see that some pair of vertices in $U$ has a common neighbour in $V$, 
and moreover for $n\ge 9$ that every pair of vertices in $U$ has a common neighbour in $V$.

Assume without loss of generality that $u_1$ is adjacent to both $u_{n-3}$ and $u_{n-2}$, 
and that $u_2$ is adjacent to $u_{n-1}$.
If $u_2$ is adjacent to a vertex of $V\setminus\{u_1\}$, which is the case when $n\ge 10$, 
then $H$ contains $T_F(n)$.
Assume now that $n\le 9$ and that $u_2$ is not adjacent to any vertex of $V\setminus\{u_1\}$.

For the case when $n=9$, 
$u_{n-1}$ is adjacent to at least $n-6=3$ vertices of $V$, 
and so it is adjacent to another vertex, say to $u_3$.
As above, assume that $u_3$ is not adjacent to any vertex of $V\setminus\{u_1\}$.
By the minimum degree condition, 
each of $u_2$ and $u_3$ is adjacent to every vertex of $\{u_1\} \cup U$, giving $T_F(9)$ in $H$.

For the final case when $n=8$, 
the minimum degree condition implies that $u_2$ is adjacent to at least two vertices of $\{u_1,u_5,u_6\}$.
If $u_2$ is adjacent to $u_1$, then $H$ contains $T_F(8)$.
Remaining is the case when $u_2$ is not adjacent to $u_1$ but is adjacent to both $u_5$ and $u_6$.
Exchanging the roles of $u_1$ and $u_2$, 
we may assume that $u_1$ is adjacent to $u_7$ but not adjacent to any vertex of $V$.
From the minimum degree condition on $u_3$ and $u_4$, 
it is easy to see that either $H$ contains $T_F(8)$ or $H = K_{4,4}$.
\end{proof}

\begin{theorem}
For $n\geq 8$,  
\[
  R(T_F(n),W_8)=\begin{cases}
    2n-1 & \text{if $n\geq 9$}\,;\\
    16   & \text{if $n=8$}\,.
  \end{cases}
\]  
\end{theorem}

\begin{proof}
Lemma~\ref{lem:lower-bound-on-S_n(1,4)-T_D(n)-...-T_L(n)}
provides the lower bound, so it remains to prove the upper bound.
Let $G$ be a graph with no $T_F(n)$ subgraph whose complement $\overline{G}$ does not contain $W_8$. 
Suppose that $n=8$ and that $G$ has order $16$. 
By Theorem~\ref{thm:R(TC,W8)}, $G$ has a subgraph $T=T_C(8)$. 
Let $V(T)=\{v_0,\ldots,v_4,w_1,w_2,w_3\}$ 
and $E(T)=\{v_0v_1,\ldots,v_0v_4,v_1w_1,v_2w_2,v_2w_3\}$. 
Set $U=V(G)-V(T)=\{u_1,\ldots,u_8\}$; 
then $|U|=8$. 
Since $T_F(8)\nsubseteq G$, 
$v_1$ is not adjacent in $G$ to $v_2,v_3,v_4$ or any vertex of $U$, 
and $d_{G[U]}(v)\leq 1$ for $v=v_3,v_4,w_2,w_3$.  

Suppose that $v_1$ is adjacent to $w_2$ or $w_3$, without loss of generality say $w_2$.
Since $T_F(8)\nsubseteq G$, $v_2$ is not adjacent to $\{v_3,v_4\}\cup U$.
If neither $v_3$ nor $v_4$ are adjacent to $U$, 
then by Lemma~\ref{lm4:R(Sn(1,2),W8)}, $G[U]$ is $K_8$ or $K_8-e$, so $G[U]$ contains $T_F(8)$, a contradiction. 
Suppose that only one of the vertices $v_3$ and $v_4$ is adjacent to $U$ in $G$, say $v_3$. 
By Lemma~\ref{lm4:R(Sn(1,2),W8)}, $G[U\setminus \{u_1\}]$ is $K_7$ or $K_7-e$, 
and $G[V(T)\cup \{u_1\}]$ is not adjacent to $G[U\setminus \{u_1\}]$.
By Observation~\ref{obs2}, $\delta(G[V(T)\cup \{u_1\}])\geq 5$, 
and by Lemma~\ref{lm:R(TG(n),W8)}, $G[V(T)\cup \{u_1\}]$ contains $T_F(9)$ and hence $T_F(8)$, a contradiction. 
Suppose that both $v_3$ and $v_4$ are adjacent to $U$ in $G$
and assume that $v_3$ is adjacent to $u_1$ and that $v_4$ is adjacent to $u_2$.
By Lemma~\ref{lm4:R(Sn(1,2),W8)}, $G[U\setminus \{u_1,u_2\}]$ is $K_6$ or $K_6-e$.
At most one vertex from $G[V(T)\cup \{u_1,u_2\}]$ is adjacent to $G[U\setminus \{u_1,u_2\}]$ or else $G$ contains $T_F(8)$. 
Therefore, $9$ vertices from $G[V(T)\cup \{u_1,u_2\}]$ form a vertex set $W$ that is not adjacent to $U\setminus \{u_1,u_2\}$. 
By Observation~\ref{obs2}, $\delta(G[W])\geq 5$, 
and by Lemma~\ref{lm:R(TG(n),W8)}, $G[W]$ contains $T_F(9)$ and hence $T_F(8)$, a contradiction. 

Suppose then that $v_1$ is not adjacent to $w_2$ or $w_3$. 
Since $d_{G[U]}(v)\leq 1$ for $v=v_3,v_4,w_2,w_3$, 
there are $4$ vertices from $U$ that are not adjacent to $\{v_3,v_4,w_2,w_3\}$.
These $8$ vertices form $C_8$ in $\overline{G}$ and thus, with $v_1$ as hub, $W_8$, a contradiction. 

Thus, $R(T_F(8),W_8)\leq 16$. 

Now, suppose that $n\geq 9$ and that $G$ has order $2n-1$. 
By Theorem~\ref{thm:R(TC,W8)}, $G$ has a subgraph $T=T_C(n)$. 
Let $V(T)=\{v_0,\ldots,v_{n-4},v_4,w_1,w_2,w_3\}$ 
and $E(T)=\{v_0v_1,\ldots,v_0v_{n-4},v_1w_1,v_2w_2,v_2w_3\}$. 
Set $V=\{v_3,\ldots,v_{n-4}\}$ and $U=V(G)-V(T)=\{u_1,\ldots,u_{n-1}\}$; 
then $|V|=n-6$ and $|U|=n-1$. 
Since $T_F(n)\nsubseteq G$, $v_1$ is not adjacent in $G$ to any vertex of $U\cup V$, 
and $d_{G[U]}(v)\leq 1$ for $v\in V$. 
Since $n\geq 10$, there are $4$ vertices from $U$, $4$ vertices from $V$ and $v_1$ that form $W_8$ in $\overline{G}$, 
a contradiction. 
Thus, $R(T_F(n),W_8)\leq 2n-1$ for $n\geq 10$.

Suppose that $n=9$ and let $m$ be the number of vertices of $U$ that are adjacent in $G$ to at least one vertex of $V$. 
Since $d_{G[U]}(v)\leq 1$ for $v\in V$, $0\leq m\leq 3$. 
If $m=0$, then $G[U]$ is $K_8$ or $K_8-e$ by Lemma~\ref{lm4:R(Sn(1,2),W8)}, 
so $G[V(T)]$ is not adjacent to $G[U]$. 
By Observation~\ref{obs2}, $\delta(G[V(T)])\geq 5$, 
and $G[V(T)]$ contains $T_F(9)$ by Lemma~\ref{lm:R(TG(n),W8)}, a contradiction. 
Suppose that $m=1$. 
Assume without loss of generality that 
$u_1$ is adjacent to some vertex of $V$, and that $E_G(V,U\setminus \{u_1\}) = \emptyset$.
By Lemma~\ref{lm4:R(Sn(1,2),W8)}, $G[U\setminus \{u_1\}]$ is $K_7$ or $K_7-e$, 
and at most one vertex from $G[V(T)\cup \{u_1\}]$ is adjacent to $G[U\setminus \{u_1\}]$ or else $G$ contains $T_F(9)$. 
There are then $9$ vertices from $G[V(T)\cup \{u_1\}]$ 
that form a vertex set $W_1$ that is not adjacent to $U\setminus \{u_1\}$. 
By Observation~\ref{obs2}, $\delta(G[W_1])\geq 5$, and $G[W_1]$ contains $T_F(9)$ by Lemma~\ref{lm:R(TG(n),W8)}, 
a contradiction. 
Suppose that $m=2$. 
Assume that $u_1$ and $u_2$ are adjacent to some vertices of $V$ and that $E_G(V,U\setminus \{u_1,u_2\})=\emptyset$. 
By Lemma~\ref{lm4:R(Sn(1,2),W8)}, $G[U\setminus \{u_1,u_2\}]$ is $K_6$ or $K_6-e$. 
If at least three vertices in $U\setminus \{u_1,u_2\}$ are adjacent to $V(T)\cup \{u_1\}$, 
then $T_F(9)\subseteq G$. 
If at most two vertices in $U\setminus \{u_1,u_2\}$ are adjacent to $V(T)\cup \{u_1\}$, 
then there are $4$ vertices in $U\setminus \{u_1,u_2\}$ that are not adjacent to $V(T)$. 
Then Observation~\ref{obs2} gives $\delta(G[V(T)])\geq 5$, 
and $G[V(T)]$ contains $T_F(9)$ by Lemma~\ref{lm:R(TG(n),W8)}, a contradiction. 
Suppose that $m=3$. 
Assume that $u_1,u_2,u_3$ are each adjacent to some vertex of $V$ and that $E_G(V,U\setminus \{u_1,u_2,u_3\})=\emptyset$.
Without loss of generality, assume that $u_i$ is adjacent to $v_{i+2}$ for $i=1,2,3$. 
By Lemma~\ref{lm4:R(Sn(1,2),W8)}, $G[U\setminus \{u_1,u_2,u_3\}]$ is $K_5$ or $K_5-e$. 
Since $T_F(9)\nsubseteq G$, 
$\{v_1,v_3,v_4,v_5\}$ is independent and $V(T)\setminus \{w_1\}$ is not adjacent to $U\setminus \{u_1,u_2,u_3\}$. 
Then by Observation~\ref{obs2}, $\delta(G[V(T)\setminus \{w_1\}])\geq 4$, 
and $v_1$, $v_3$, $v_4$ and $v_5$ are each adjacent to $v_2$, $w_2$ and $w_3$ in~$G$. 
This gives $T_F(9)$ in~$G$. 
Therefore, $T_F(9)\leq 17=2n-1$. 
\end{proof}

\begin{theorem}
\label{thm:R(TH(n),W8)}
If $n\geq 8$, then $R(T_G(n),W_8)=2n-1$. 
\end{theorem}

\begin{proof}
Lemma~\ref{lem:lower-bound-on-S_n(1,4)-T_D(n)-...-T_L(n)}
provides the lower bound, so it remains to prove the upper bound.
Let $G$ be any graph of order $2n-1$.
Assume that $G$ does not contain $T_G(n)$ and that $\overline{G}$ does not contain $W_8$. 
By Theorem~\ref{thm:R(Sn(3,1),W8)}, $G$ has a subgraph $T=S_n(3,1)$. 
Let $V(T)=\{v_0,\ldots,v_{n-4},w_1,w_2,w_3\}$ 
and $E(T)=\{v_0v_1,\ldots,v_0v_{n-4},v_1w_1,v_2w_2,v_3w_3\}$. 
Set $V=\{v_4,v_5,\ldots,v_{n-4}\}$ and $U=V(G)-V(T)$; 
then $|V|=n-7$ and $|U|=n-1$. 
Since $T_G(n)\nsubseteq G$, 
$w_1,w_2,w_3$ are not adjacent to $U\cup V$ in $G$, 
and $v_1,v_2,v_3$ are not adjacent to $V$. 

Suppose that $n\geq 9$; then $|U|\geq 8$. 
If $\delta(\overline{G}[U])\geq \frac{n-1}{2}$,
then $\overline{G}[U]$ contains $C_8$ by Lemma~\ref{lem:pancyclic} which, with $w_2$ as hub, forms $W_8$, a contradiction. 
Therefore, $\delta(\overline{G}[U])<\frac{n-1}{2}$, 
and $\Delta(G[U\cup V])\geq \frac{n-1}{2}\geq 4$. 
Therefore, some vertex $u\in U$ satisfies $|N_{G[U]}(u)|\geq 4$. 
Since $T_G(n)\nsubseteq G$, $N_{G[U]}(u)$ is not adjacent in $G$ to $N_{G[V(T)]}(v_0)$. 
Hence, $4$ vertices from $N_{G[U]}(u)$, $v_1,v_2,v_3,w_1$ and any vertex from $V$ form $W_8$ in $\overline{G}$, a contradiction. 
Thus, $R(T_G(n),W_8)\leq 2n-1$ for $n\geq 9$. 

Suppose that $n=8$ and let $U=\{u_1,\ldots,u_7\}$ and $W=\{v_4\}\cup U$. 
If $\delta(\overline{G}[W])\geq 4$, 
then $\overline{G}$ contains $C_8$ by Lemma~\ref{lem:pancyclic} and thus $W_8$, with $w_1$ as hub, a contradiction. 
Therefore, $\delta(\overline{G}[W])\leq 3$, and $\Delta(G[W])\geq 4$. 
Now, suppose that $d_{G[W]}(v_4)\geq 4$.
Then without loss of generality, assume that $u_1,\ldots,u_4\in N_G(v_4)$.
Then $u_1,\ldots,u_4,w_1,w_2,w_3$ are independent and are not adjacent to $u_5$, $u_6$ or $u_7$, 
giving $W_8$, a contradiction. 
On the other hand, suppose that some vertex in $U$, say $u_1$, satisfies $d_{G[W]}(u_1)\geq 4$. 
Then $v_4$ is not adjacent to $u_1$; therefore, assume that $u_2,\ldots,u_5\in N_G(u_1)$. 
Then $v_1,\ldots,v_4$ are not adjacent to $\{u_1,\ldots,u_5\}$, 
so $v_1u_1v_2u_2v_3u_3w_1u_4v_1$ and $v_4$ form $W_8$ in $\overline{G}$, a contradiction. 
Thus, $R(T_L(8),W_8)\leq 15$.  
\end{proof}

\begin{lemma}\label{lm:R(TJLM(n),W8)} 
Each graph $H$ of order $n\geq 8$ with minimal degree at least $n-4$ contains $T_H(n)$, $T_K(n)$ and $T_L(n)$. 
\end{lemma}

\begin{proof}
Let $V(H)=\{u_0,\ldots,u_{n-1}\}$ where $u_1,\ldots,u_{n-4}\in N_H(u_0)$. 
Suppose that $u_{n-3}$, $u_{n-2}$ or $u_{n-1}$, 
say $u_{n-3}$, is adjacent in $H$ to the two others. 

Since $\delta(H)\geq n-4$, 
$u_{n-3}$ is adjacent to at least one of  $u_1,\ldots,u_{n-4}$, say $u_1$. 
If $u_1$ is adjacent to another vertex in $\{u_2,\ldots,u_{n-4}\}$, 
then $H$ contains $T_K(n)$.
Note that this always happens for $n\geq 9$. 
Suppose that $n=8$ and that $u_1$ is not adjacent to any of $u_2,u_3,u_4$.
Then $u_1$ is adjacent to $u_6$ and $u_7$.
Since $\delta(H)\geq n-4$, 
$u_2$ is adjacent to at least one of $u_5, u_6, u_7$, 
giving $T_K(n)$ in $H$.

Similarly, since $\delta(H)\geq n-4$, 
$u_{n-2}$ is adjacent to at least $n-7$ vertices of $\{u_1,\ldots,u_{n-4}\}$. 
Suppose that $u_{n-2}$ is adjacent to $u_1$. 
If $n\geq 10$, 
then at least two of $u_2,\ldots,u_{n-4}$ are adjacent, 
so $H$ contains $T_H(n)$. 
If $n\geq 9$, 
then $u_1$ is adjacent to at least one of $u_2,\ldots,u_{n-4}$, 
so $H$ contains $T_L(n)$.
Now suppose that $n=9$. 
If any of $u_2,\ldots,u_5$ are adjacent to each other, 
then $H$ contains $T_H(9)$. 
Otherwise, $u_2,\ldots,u_5$ are each adjacent to $u_6$, $u_7$ and $u_8$, 
and so $H$ contains $T_H(9)$. 
Finally, suppose that $n=8$.
If any two of $u_2,u_3,u_4$ are adjacent, 
then $H$ contains $T_H(8)$; 
otherwise, they are each adjacent to $u_6$ or $u_7$. 
Now, if $u_1$ is adjacent to any of $u_2,u_3,u_4$, 
then $H$ contains $T_H(8)$. 
Otherwise, $u_1,\ldots,u_4$ are each adjacent to $u_5$, $u_6$ and $u_7$, 
and $H$ also contains $T_H(8)$. 
Furthermore, if $u_1$ is adjacent to $u_2$, $u_3$ or $u_4$, 
then $H$ contains $T_L(8)$.
If $u_1$ is not adjacent to $u_2$, $u_3$ or $u_4$, 
then $u_6,u_7,u_8$ are adjacent to $u_2,u_3,u_4$, 
and 
then $H$ contains $T_L(8)$. 
Now if $u_{n-2}$ is adjacent to some $u_2,\ldots,u_{n-4}$, say $u_2$, 
then similar arguments apply by interchanging $u_1$ and $u_2$.

Suppose now that none of $u_{n-3},u_{n-2},u_{n-1}$ is adjacent to both of the others.
Then one of these, say $u_{n-3}$, is adjacent to neither of the others. 
Since $\delta(H)\geq n-4$, 
$u_{n-3}$ is adjacent to at least $n-5$ of the vertices $u_1,\ldots,u_{n-4}$. 
Without loss of generality, assume that $u_1,\ldots,u_{n-5}\in N_H(u_{n-3})$. 
Then $u_{n-2}$ is adjacent to at least $n-7$ of the vertices $u_1,\ldots,u_{n-5}$ 
including, without loss of generality, the vertex $u_1$. 
Also, $u_{n-1}$ is adjacent to at least one of $u_2,\ldots,u_{n-4}$, 
so $H$ contains $T_H(n)$.
If $u_{n-2}$ is adjacent to $u_{n-1}$, then $H$ also contains $T_L(n)$.
If $u_{n-2}$ is not adjacent to $u_{n-1}$, 
then $u_{n-2}$ is adjacent to at least $n-6$ vertices of $u_1,\ldots,u_{n-5}$, 
so $H$ contains $T_L(n)$.
Now, suppose that $n\geq 9$.
Then $u_{n-2}$ and $u_{n-1}$ are each adjacent to at least $3$ of $u_1,\ldots,u_5$, 
and one of those vertices must be adjacent to both $u_{n-2}$ and $u_{n-1}$;
thus, $H$ contains $T_K(n)$.
Finally, suppose that $n = 8$.
If $u_6$ and $u_7$ are each adjacent to at least two of the vertices $u_1,u_2,u_3$, 
then one of those vertices must be adjacent to both $u_6$ and $u_7$;
thus, $H$ contains $T_K(8)$.
Otherwise, $u_6$ or $u_7$, say $u_6$, is non-adjacent to at least two of $u_1,u_2,u_3$, say $u_1$ and $u_2$.
Then $u_6$ is adjacent to $u_0$, $u_3$, $u_4$ and $u_7$,
and so $H$ contains $T_K(8)$.
\end{proof}

\begin{theorem}
\label{thm:R(TJ(n),W8)}
If $n\geq 8$, then $R(T_H(n),W_8)=2n-1$. 
\end{theorem}

\begin{proof}
Lemma~\ref{lem:lower-bound-on-S_n(1,4)-T_D(n)-...-T_L(n)}
provides the lower bound, so it remains to prove the upper bound.
Let $G$ be any graph of order $2n-1$
and assume that $G$ does not contain $T_H(n)$ and that $\overline{G}$ does not contain $W_8$. 
By Theorem~\ref{thm:R(TH(n),W8)}, $G$ has a subgraph $T=T_G(n)$. 
Let $V(T)=\{v_0,\ldots,v_{n-5},w_1,\ldots,w_4\}$ 
and $E(T)=\{v_0v_1,\ldots,v_0v_{n-5},v_1w_1,v_2w_2,v_3w_3,w_3w_4\}$. 
Set $U=\{u_1,\ldots,u_{n-1}\}=V(G)-V(T)$; 
then $|U|=n-1$. 
Since $T_G(n)\nsubseteq G$, $E_G(\{w_1,w_2\},\{w_3,w_4\})=\emptyset$ and $w_4$ is not adjacent to $U$. 
Now, let $W=\{w_1\}\cup U$; then $|W|=n$.  
If $\delta(\overline{G}[W])\geq \frac{n}{2}$, 
then $\overline{G}[W]$ contains $C_8$ by Lemma~\ref{lem:pancyclic} 
which, with $w_4$ as hub, forms $W_8$, a contradiction. 
Therefore, $\delta(\overline{G}[W])<\frac{n}{2}$, 
and $\Delta(G[W])\geq \lfloor \frac{n}{2} \rfloor \geq 4$. 

First, suppose that $w_1$ is a vertex with degree at least $\frac{n}{2}$ in $G[W]$. 
Assume without loss of generality that $u_1,\ldots,u_4\in N_{G[W]}(w_1)$. 
Since $T_H(n)\nsubseteq G$, $u_1,\ldots,u_4$ are independent and are not adjacent to $\{w_2,u_5,\ldots,u_{n-1}\}$ in~$G$. 
Then $w_2,u_1,\ldots,u_4,w_4$ and any $3$ vertices from $\{u_5,\ldots,u_{n-1}\}$ form $W_8$ in $\overline{G}$, a contradiction. 
Hence, $d_{G[W]}(u')\geq \frac{n}{2}$ for some vertex $u'\in U$, say $u'=u_1$.
Note that $w_1$ is not adjacent to $u_1$, or else $G$ contains $T_H(n)$. 
Without loss of generality, suppose that $u_2,\ldots,u_5\in N_{G[W]}(u_1)$. 
Since $T_H(n)\nsubseteq G$, 
$u_2,\ldots,u_5$ are not adjacent to $V(T)\setminus \{v_0\}$ in~$G$. 
Now, if $v_0$ is not adjacent to $\{u_2,\ldots,u_5\}$ in $G$, 
then by Observation~\ref{obs2}, $\delta(G[V(T)])\geq n-4$, or else $\overline{G}$ contains $W_8$. 
By Lemma~\ref{lm:R(TJLM(n),W8)}, $G[V(T)]$ contains $T_H(n)$, a contradiction. 
On the other hand, suppose that $v_0$ is adjacent to at least one of $u_2,\ldots,u_5$, say $u_2$.
Then $u_3,u_4,u_5$ are independent in $G$ and are not adjacent to $u_6$ and $u_7$ in~$G$. 
Furthermore, $w_4$ is not adjacent to $v_1$ or $v_2$. 
Then $v_1u_3v_2u_4u_6w_1u_7u_5v_1$ and $w_4$ form $W_8$ in $\overline{G}$, a contradiction. 
Thus, $R(T_H(n),W_8)\leq 2n-1$. 
\end{proof}

\begin{theorem}
\label{thm:R(TK(n),W8)}
If $n\geq 8$, then $R(T_J(n),W_8)=2n-1$. 
\end{theorem}

\begin{proof}
Lemma~\ref{lem:lower-bound-on-S_n(1,4)-T_D(n)-...-T_L(n)}
provides the lower bound, so it remains to prove the upper bound.
Let $G$ be any graph of order $2n-1$
and assume that $G$ does not contain $T_J(n)$ 
and that $\overline{G}$ does not contain $W_8$. 
By Theorem~\ref{thm:R(TC,W8)}, $G$ has a subgraph $T=T_C(n)$. 
Let $V(T)=\{v_0,\ldots,v_{n-4},w_1,w_2,w_3\}$ 
and $E(T)=\{v_0v_1,\ldots,v_0v_{n-4},v_1w_1,v_1w_2,v_2w_3\}$. 
Set $V=\{v_3,\ldots,v_{n-4}\}$ and $U=V(G)-V(T)$;
then $|U|=n-1$. 
Let $U = \{u_1,\ldots,u_{n-1}\}$. 
Since $T_J(n)\nsubseteq G$, 
neither $w_1$ nor $w_2$ is adjacent in $G$ to any vertex from $U\cup V$. 

Let $W=\{v_3\}\cup U$; then $|W|=n$.  
If $\delta(\overline{G}[W])\geq \lceil \frac{n}{2} \rceil\geq \frac{n}{2}$, 
then $\overline{G}[W]$ contains $C_8$ by Lemma~\ref{lem:pancyclic} 
which with $w_1$ forms $W_8$, a contradiction. 
Thus, $\delta(\overline{G}[W])<\lceil \frac{n}{2} \rceil$,
and 
$\Delta(G[W])\geq 
\lfloor \frac{n}{2} \rfloor \geq 4$. 

Suppose that $d_{G[W]}(v_3)\geq \lfloor \frac{n}{2} \rfloor\geq 4$. 
Without loss of generality, assume that $u_1,\ldots,u_4\in N_G(v_3)$.
Since $T_J(n)\nsubseteq G$, 
$u_1,\ldots,u_4$ is independent in $G$ 
and is not adjacent to any remaining vertices from $U$ in~$G$. 
Then $u_2w_1u_3u_5u_4u_6w_2u_7u_2$ and $u_1$ form $W_8$ in $\overline{G}$, a contradiction. 
Hence, there is a vertex in $U$, say $u_1$, 
such that $d_{G[W]}(u_1)\geq \lfloor \frac{n}{2} \rfloor\geq 4$. 

Now, suppose that $v_3$ is adjacent to $u_1$ in $G[W]$.
Then $u_1$ is adjacent to at least $3$ other vertices of $U$ in~$G$,
say $u_2$, $u_3$ and $u_4$.
Since $T_J(n)\nsubseteq G$, 
$v_3$ is not adjacent to $v_1,v_2,v_4,\ldots,v_{n-4},w_1,w_2,w_3,u_2,u_3,u_4$ 
and neither $v_1$ nor $v_2$ is adjacent to $u_2$, $u_3$ or $u_4$ in~$G$. 
Then $v_2u_2v_1u_3w_1v_4w_2u_4v_2$ and $v_3$ form $W_8$ in $\overline{G}$, a contradiction. 

Thus, $v_3$ is not adjacent to $u_1$ in~$G$. 
Note that $u_1$ is not adjacent to any other vertices of $V$ in $G$ or else previous arguments apply. 
Similarly, $v_0$ is not adjacent to $N_{G[W]}(u_1)$ in~$G$. 
Since $T_J(n)\nsubseteq G$, neither $v_1$ nor $v_2$ is adjacent to $u_1$ or $N_{G[W]}(u_1)$ in $G$, 
and so $d_{N_{G[W]}(u_1)}(v)\leq 1$ for all $v\in V$. 

Suppose that $n\geq 10$; then $|V|\geq 4$ and $|N_{G[W]}(u_1)|\geq 5$. 
If $d_{G[V]}(u)\leq 2$ for each $u\in N_{G[W]}(u_1)$, 
then $\overline{G}[V\cup N_{G[W]}(u_1)]$ contains $C_8$ by Lemma~\ref{lem:1-1relation} 
which, with $w_1$ as hub, forms $W_8$ in $\overline{G}$, a contradiction.
Thus, $d_V(u')\geq 3$ for some vertex $u'\in N_{G[W]}(u_1)$.
Then any $4$ vertices from $V$, of which at least $3$ are in $N_{G[V]}(u')$, 
and any $4$ vertices from $N_{G[W]}(u_1)\setminus \{u'\}$
satisfy the condition in Lemma~\ref{lem:1-1relation}, 
so $\overline{G}[V\cup N_{G[W]}(u_1)]$ contains $C_8$ which with $w_1$ forms $W_8$, a contradiction. 

Suppose that $n=9$; then $V=\{v_3,v_4,v_5\}$.
Assume that $u_2,\ldots,u_5\in N_{G[W]}(u_1)$. 
Suppose that $w_1$ is not adjacent to $w_2$ in~$G$.
Let $X=\{v_3,v_4,v_5,w_2\}$ and $Y=\{u_2,\ldots,u_5\}$
and note that $d_{G[Y]}(x)\leq 1$ for each $x\in X$.
If $d_{G[X]}(y)\leq 2$ for each $y\in Y$, 
then $\overline{G}[X\cup Y]$ contains $C_8$ by Lemma~\ref{lem:1-1relation} 
which, with $w_1$ as hub, forms $W_8$, a contradiction. 
Thus, $d_{G[X]}(u')\geq 3$ for some $u'\in Y$, say $u'=u_2$, 
so $X$ is not adjacent to $Y\setminus \{u_2\}$. 
Hence, $v_3u_1v_4u_3v_5u_4w_2u_5v_3$ and $w_1$ form $W_8$ in $\overline{G}$, a contradiction. 

Thus, $w_1$ is adjacent to $w_2$ in~$G$.
Then $v_1$ is not adjacent to $\{v_3,v_4,v_5\}\cup U$. 
Suppose that $v_1$ is not adjacent to $v_2$.
Then set $X=\{v_2,\ldots,v_5\}$ and $Y=\{u_2,\ldots,u_5\}$.
If $d_{G[X]}(y)\leq 2$ for each $y\in Y$, 
then $\overline{G}[X\cup Y]$ contains $C_8$ by Lemma~\ref{lem:1-1relation} 
which, with $v_1$ as hub, forms $W_8$, a contradiction. 
Thus, $d_{G[X]}(u')\geq 3$ for some $u'\in Y$, say $u'=u_2$, 
so $X$ is not adjacent to $Y\setminus \{u_2\}$, 
and $v_2u_1v_3u_3v_4u_4v_5u_5v_2$ and $v_1$ form $W_8$ in $\overline{G}$, a contradiction. 
Thus, $v_1$ is adjacent to $v_2$ in~$G$.
Then $V$ is independent and is not adjacent to $U$ in~$G$. 
Since $W_8\nsubseteq \overline{G}$, 
$G[U]$ is $K_{n-1}$ or $K_{n-1}-e$ by Lemma~\ref{lm4:R(Sn(1,2),W8)}. 
Since $T_J(9)\nsubseteq G$, 
$T$ is not adjacent to $U$ and, by Observation~\ref{obs2}, $\delta(G[V(T)])\geq 5$. 
However, this is impossible since $V$ is independent and is not adjacent to $v_1$, $w_1$ or $w_2$. 

Finally, suppose that $n=8$; then $V=\{v_3,v_4\}$.
Assume that $u_2,\ldots,u_5\in N_{G[W]}(u_1)$. 
If $v_3$ is adjacent to any vertex of $\{u_2,\ldots,u_5\}$, say $u_2$, 
then $v_3$ is not adjacent to $\{v_1,v_2,v_4,w_3\}\cup U\setminus \{u_2\}$,
so $v_1u_1v_2u_3w_1u_4w_2u_5v_1$ and $v_3$ form $W_8$ in $\overline{G}$, a contradiction. 
Thus, $v_3$ is not adjacent to $\{u_2,\ldots,u_5\}$. 
Similarly, $v_4$ is not adjacent to $\{u_2,\ldots,u_5\}$.  
Now, if $w_3$ is adjacent to any of the vertices $u_2,\ldots,u_5$, say $u_2$, 
then $v_2$ is not adjacent to $\{w_1, w_2, v_3, v_4\}$,
so $v_3u_1v_4u_2w_1u_3w_2u_4v_3$ and $v_2$ form $W_8$ in $\overline{G}$, a contradiction.
Thus, $w_3$ is not adjacent to $\{u_2,\ldots,u_5\}$.
By Observation~\ref{obs2}, $\delta(G[V(T)])\geq 4$. 
Suppose that $v_2$ is adjacent to $w_1$.
Since $T_J(8)\nsubseteq G$, neither $v_3$ nor $v_4$ is adjacent to $w_3$.
Since $\delta(G[V(T)])\geq 4$, 
$v_3$ and $v_4$ are adjacent to $v_1$ and $v_2$, 
and $\{w_1,w_2,w_3\}$ is not independent. 
However, then $T_J(8)\subseteq G[V(T)]$, a contradiction. 
Thus, $v_2$ is not adjacent to $w_1$ and, similarly, $v_2$ is not adjacent to $w_2$. 
Since $\delta(G[V(T)])\geq 4$, 
$w_1$ and $w_2$ are adjacent to each other and to $w_3$. 
Since $T_J(8)\nsubseteq G$, neither $v_3$ nor $v_4$ is adjacent to $v_1$ or $v_2$;
however, this contradicts $\delta(G[V(T)])\geq 4$. 

In each case, $R(T_J(8),W_8)\leq 2n-1$ which completes the proof of the theorem. 
\end{proof}

\begin{theorem}
\label{thm:R(TL(n),W8)}
If $n\geq 8$, then $R(T_K(n),W_8)=2n-1$. 
\end{theorem}

\begin{proof}
Lemma~\ref{lem:lower-bound-on-S_n(1,4)-T_D(n)-...-T_L(n)}
provides the lower bound, so it remains to prove the upper bound.
Let $G$ be a graph of order $2n-1$ and assume that $G$ does not contain $T_K(n)$ 
and that $\overline{G}$ does not contain $W_8$. 

Suppose that $n\not\equiv 0 \pmod{4}$.
By Theorem~\ref{thm:R(Sn(1,3),W8)}, $G$ has a subgraph $T=S_n(1,3)$. 
Let $V(T)=\{v_0,\ldots,v_{n-4},w_1,w_2,w_3\}$ 
and $E(T)=\{v_0v_1,\ldots,v_0v_{n-4},v_1w_1,w_1w_2,w_2w_3\}$. 
Set $V=\{v_2,\ldots,v_{n-4}\}$ and $U=V(G)-V(T)$;
then $|V|=n-5$ and $|U|=n-1$. 
Since $T_K(n)\nsubseteq G$, $w_2$ is not adjacent in $G$ to any vertex of $U\cup V$. 
Now, if $\delta(G[U])\geq \frac{n-1}{2}$, 
then $\overline{G}[U]$ contains $C_8$ by Lemma~\ref{lem:pancyclic} 
which, with $v_1$ as hub, forms $W_8$, a contradiction. 
Therefore, $\delta(\overline{G}[U])<\frac{n-1}{2}$, 
and $\Delta(G[U])\geq \lfloor \frac{n-1}{2} \rfloor$. 
Let $U=\{u_1,\ldots,u_{n-1}\}$ 
and assume without loss of generality that $d_{G[U]}(u_1)\geq \lfloor \frac{n-1}{2} \rfloor \geq 4$. 
Since $T_K(n)\nsubseteq G$, $E_G(V,N_{G[U]}(u_1))=\emptyset$, 
so any $4$ vertices from $V$, any $4$ vertices from $N_{G[U]}(u_1)$ and $w_2$ form $W_8$ in $\overline{G}$, a contradiction. 
Therefore, $R(T_K(n),W_8)\leq 2n-1$ for $n\not\equiv 0 \pmod{4}$. 

Let $n=8$.
By Theorem~\ref{thm:R(TJ(n),W8)}, $G$ has a subgraph $T=T_H(8)$.
Let $V(T)=\{v_0,v_1,v_2,v_3,w_1,\ldots,w_4\}$ 
and $E(T)=\{v_0v_1,\ldots,v_0v_3,v_1w_1,w_1w_2,w_2w_3,v_2w_4\}$. 
Set $U=V(G)-V(T)=\{u_1,\ldots,u_7\}$; 
then $|U|=7$. 
Since $T_K(8)\nsubseteq G$, 
$w_2$ is not adjacent to $\{w_4\}\cup U$. 
Let $W=\{w_4\}\cup U$; then $|W|=8$.  
If $\delta(\overline{G}[W])\geq 4$, 
then $\overline{G}[W]$ contains $C_8$ by Lemma~\ref{lem:pancyclic} which, with $w_2$ as hub, forms $W_8$, a contradiction. 
Therefore, $\delta(\overline{G}[W])<3$, 
and $\Delta(G[W])\geq 4$. 

Now, suppose that $d_{G[W]}(w_4)\geq 4$ and assume without loss of generality 
that $w_4$ is adjacent to $u_1$, $u_2$, $u_3$ and $u_4$. 
Then $v_1$ is not adjacent to $\{v_3,w_2,w_3\}\cup U$ 
and neither $v_2$ nor $v_3$ is adjacent to $\{u_1,\ldots,u_4\}$, since $T_K(8)\nsubseteq G$. 
Now, suppose that $E_G(\{u_1,\ldots,u_4\},\{u_5,u_6,u_7\})\neq \emptyset$ 
and assume that $u_1$ is adjacent to $u_5$. 
Then $u_1$ is not adjacent to $\{w_1,w_2,w_3,u_2,\ldots,u_7\}$ in $G$, 
and $v_1u_2v_2u_3v_3u_4w_2u_6v_1$ and $u_1$ form $W_8$ in $\overline{G}$, a contradiction.
Thus, $E_G(\{u_1,\ldots,u_4\},\{u_5,u_6,u_7\})=\emptyset$, 
so $u_1u_5u_2u_6u_3u_7u_4v_3u_1$ and $v_1$ form $W_8$ in $\overline{G}$, a contradiction.

Now suppose that $d_{G[W]}(u')\geq 4$ for some vertex $u'\in U$, say $u'=u_1$.
Since, $T_K(8)\nsubseteq G$, $w_4$ is not adjacent to $u_1$.
Then without loss of generality, suppose that $u_2,\ldots,u_5\in N_G(u_1)$. 
Since $T_K(8)\nsubseteq G$, $E_G(\{v_1,v_2,v_3\},\{u_2,\ldots,u_5\})=\emptyset$. 
If $u_2$ is adjacent to $w_1$, then $u_2$ is not adjacent to $\{u_3,\ldots,u_7\}$ and $v_1$ is not adjacent to $u_6$.
Then $w_2u_3v_2u_4v_3u_5v_1u_6w_2$ and $u_2$ form $W_8$ in $\overline{G}$, a contradiction.
Thus, $u_2$ is not adjacent to $w_1$. 
Similarly, $u_3$, $u_4$ and $u_5$ are not adjacent to $w_1$. 
If $u_2$ is adjacent to $v_0$, 
then $v_2$ is not adjacent to $\{v_1,v_3,w_1,w_2,w_3,u_2,\ldots,u_7\}$, 
and $v_1u_2v_3u_3w_1u_4w_2u_5v_1$ and $v_2$ form $W_8$ in $\overline{G}$, a contradiction.
Thus, $u_2$ is not adjacent to $v_0$. 
Similarly, $u_3$, $u_4$ and $u_5$ are not adjacent to $v_0$. 
By similar arguments, $u_3$, $u_4$ and $u_5$ are not adjacent to $w_3$ or $w_4$.

Hence, $u_2,\ldots,u_5$ are not adjacent to $V(T)$ in $G$,
so $\delta(G[V(T)])\geq 4$ by Observation~\ref{obs2}. 
By Lemma~\ref{lm:R(TJLM(n),W8)}, 
$G[V(T)]$ contains $T_K(8)$, a contradiction. 
Thus, $R(T_K(8),W_8)\leq 15$. 

Now suppose that $n\equiv 0 \pmod{4}$ and that $n\geq 12$.
If $G$ has an $S_n(1,3)$ subgraph, 
then the arguments above lead to contradictions.
Thus, $G$ does not contain $S_n(1,3)$ as a subgraph. 
Now, by Theorem~\ref{thm:R(TJ(n),W8)}, $G$ has a subgraph $T=T_H(n)$. 
Let $V(T)=\{v_0,\ldots,v_{n-5},w_1,\ldots,w_4\}$ 
and $E(T)=\{v_0v_1,\ldots,v_0v_{n-5},v_1w_1,w_1w_2,w_2w_3,v_2w_4\}$. 
Set $V=\{v_3,\ldots,v_{n-5}\}$ and let $U=V(G)-V(T)=\{u_1,\ldots,u_{n-1}\}$. 
Then $|V|=n-7$ and $|U|=n-1$. 
Since $T_K(n)\nsubseteq G$,   $w_2$ is not adjacent in $G$ to $\{w_4\}\cup U$.
Since $S_n(1,3)\nsubseteq G$, $v_0$ is not adjacent to $\{w_4\}\cup U$. 

If $\delta(\overline{G}[U])\geq \frac{n-1}{2}$, 
then $\overline{G}[U]$ contains $C_8$ by Lemma~\ref{lem:pancyclic} which, with $w_2$, forms $W_8$, a contradiction. 
Thus, $\delta(\overline{G}[U])<\frac{n-1}{2}$, 
and $\Delta(G[U])\geq \lfloor \frac{n-1}{2} \rfloor \geq 5$. 
Without loss of generality, assume that $u_2,\ldots,u_6\in N_G(u_1)$. 
Since $T_K(n)\nsubseteq G$, 
$v_1$, $v_2$ and $V$ are not adjacent to $\{u_2,\ldots,u_6\}$,
and $w_1$ and $w_2$ are not adjacent to $u_1$. 

Now, if $u_2$ is adjacent to $w_1$, 
then $u_2$ is not adjacent to $\{w_3, w_4\}\cup U\setminus \{u_1\}$, since $T_K(n)\nsubseteq G$, 
so $v_0u_3v_1u_4v_2u_5v_3u_6v_0$ and $u_2$ form $W_8$ in $\overline{G}$, a contradiction.
Thus, $u_2$ is not adjacent to $w_1$. 
Similarly, $u_3,\ldots,u_6$ are not adjacent to $w_1$. 
If $u_2$ is adjacent to $w_3$ in $G$, 
then $v_0$ is not adjacent to $w_1,w_2,w_3$, 
and $d_{G[U\setminus \{u_1,u_2\}]}(u_i)\leq n-6$ for $i=3,\ldots,6$, since $S_n(1,3)\nsubseteq G$. 
Since $T_K(n)\nsubseteq G$, $w_3$ is not adjacent to $w_1$ or $w_4$. 
Since $d_{G[U\setminus \{u_1,u_2\}]}(u_3)\leq n-6$ and $d_{G[U\setminus \{u_1,u_2\}]}(u_4)\leq n-6$, 
$u_3$ and $u_4$ are adjacent in $\overline{G}$ to at least $2$ vertices in $\{u_7,\ldots,u_{n-1}\}$.
Without loss of generality, 
assume that $u_3$ is adjacent in $\overline{G}$ to $u_7$ and that $u_4$ is adjacent to $u_8$. 
Then $u_3u_7w_2u_8u_4w_1w_3w_4u_3$ and $v_0$ form $W_8$ in $\overline{G}$, a contradiction. 
Thus, $u_2$ is not adjacent to $w_3$. 
Similarly, $u_3,\ldots,u_6$ are not adjacent to $w_4$. 

Hence, $u_2,\ldots,u_6$ are not adjacent to $V(T)$. 
By Observation~\ref{obs2}, $\delta(G[V(T)])\geq 4$, 
so $G[V(T)]$ contains $T_K(n)$ by Lemma~\ref{lm:R(TJLM(n),W8)}, a contradiction. 
Thus, $R(T_K(n),W_8)\leq 2n-1$ for $n\equiv 0 \pmod{4}$. 
This completes the proof. 
\end{proof}

\begin{theorem}
If $n\geq 8$, then $R(T_L(n),W_8)=2n-1$. 
\end{theorem}

\begin{proof}
Lemma~\ref{lem:lower-bound-on-S_n(1,4)-T_D(n)-...-T_L(n)}
provides the lower bound, so it remains to prove the upper bound.
Let $G$ be a graph with no $T_L(n)$ subgraph whose complement $\overline{G}$ does not contain $W_8$. 
Suppose that $n\not\equiv 0 \pmod{4}$ and that $G$ has order $2n-1$.
By Theorem~\ref{thm:R(Sn(1,3),W8)}, $G$ has a subgraph $T=S_n(1,3)$. 
Let $V(T)=\{v_0,\ldots,v_{n-4},w_1,w_2,w_3\}$ 
and $E(T)=\{v_0v_1,\ldots,v_0v_{n-4},v_1w_1,w_1w_2,w_2w_3\}$. 
Set $V=\{v_2,\ldots,v_{n-4}\}$ and $U=V(G)-V(T)$;
then $|V|=n-5$ and $|U|=n-1$. 
Since $T_L(n)\nsubseteq G$, 
$v_1$ is not adjacent to $U\cup V$,
and $d_{G[U]}(v_i)\leq n-7$ for each $v_i\in V$. 
Now, if $\delta(G[U])\geq \frac{n-1}{2}$, 
then $\overline{G}[U]$ contains $C_8$ by Lemma~\ref{lem:pancyclic} which, with $v_1$, forms $W_8$, a contradiction. 
Thus, $\delta(\overline{G}[U])<\frac{n-1}{2}$, 
and $\Delta(G[U])\geq \lfloor \frac{n-1}{2} \rfloor$. 

Let $U=\{u_1,\ldots,u_{n-1}\}$ and without loss of generality 
assume that $d_{G[U]}(u_1)\geq \lfloor \frac{n-1}{2} \rfloor \geq 4$ and that $u_2,\ldots,u_5\in N_{G[U]}(u_1)$. 
Now if $E_G(V,N_{G[U]}(u_1))=\emptyset$, 
then $4$ vertices from $V$, $4$ vertices from $N_{G[U]}(u_1)$ and $v_1$ form $W_8$ in $\overline{G}$, a contradiction. 
Thus, $E_G(V,N_{G[U]}(u_1))\neq \emptyset$.
Assume without loss of generality that $v_2$ is adjacent to $u_2$. 
Since $T_L(n)\nsubseteq G$, $v_2$ is not adjacent to $U\setminus \{u_1,u_2\}$. 
Since $d_{G[U]}(v_i)\leq n-7$ for each $v_i\in V$, 
$v_5$ is non-adjacent to at least one of $u_6,\ldots,u_{n-1}$, say $u_6$. 
Now if $E_G(\{v_3,v_4,v_5\},\{u_3,u_4,u_5\})=\emptyset$, 
then $v_2u_3v_3u_4v_4u_5v_5u_6v_2$ and $v_1$ form $W_8$ in $\overline{G}$, a contradiction.
Thus assume, say, that $v_3$ is adjacent to $u_3$ in $G$; 
then $v_3$ is not adjacent to $U\setminus \{u_1,u_3\}$. 
Again, if $E_G(\{v_4,v_5\},\{u_4,u_5\})=\emptyset$, 
then $v_2u_7v_3u_4v_4u_5v_5u_6v_2$ and $v_1$ form $W_8$ in $\overline{G}$, a contradiction.
Thus assume, say, that $v_4$ is adjacent to $u_4$, 
then $v_4$ is not adjacent to $U\setminus \{u_1,u_4\}$. 
If $v_5$ is not adjacent to $u_5$,
then $v_2u_7v_3u_2v_4u_5v_5u_6v_2$ and $v_1$ form $W_8$ in $\overline{G}$, a contradiction.
Thus, $v_5$ is adjacent to $u_5$, 
so $v_5$ is not adjacent to $U\setminus \{u_1,u_5\}$, 
and $v_2u_7v_3u_2v_4u_3v_5u_6v_2$ and $v_1$ form $W_8$ in $\overline{G}$, a contradiction.

Hence, $R(T_L(n),W_8)\leq 2n-1$ for $n\not\equiv 0 \pmod{4}$. 

Now, suppose that $n\equiv 0 \pmod{4}$ and that $G$ has order $2n-1$.
Suppose first that $n=8$.
By Theorem~\ref{thm:R(TJ(n),W8)}, $G$ has a subgraph $T=T_H(8)$. 
Let $V(T)=\{v_0,\ldots,v_3,w_1,\ldots,w_4\}$ 
and $E(T)=\{v_0v_1,\ldots,v_0v_3,v_1w_1,w_1w_2,w_2w_3,v_2w_4\}$. 
Set $U=V(G)-V(T)=\{u_1,\ldots,u_7\}$; then $|U|=7$. 
Since $T_L(8)\nsubseteq G$, 
neither $v_1$ nor $v_2$ are adjacent to $U$, 
and $d_{G[U]}(v_3)\leq 1$. 
Furthermore, $v_1$ is not adjacent to $w_4$, 
and $v_2$ is not adjacent to $w_1$ or $w_3$. 
Let $W=\{w_4\}\cup U$; then $|W|=8$.  
If $\delta(\overline{G}[W])\geq 4$, 
then $\overline{G}[W]$ contains $C_8$ by Lemma~\ref{lem:pancyclic} 
which with $v_1$ forms $W_8$, a contradiction. 
Thus, $\delta(\overline{G}[W])<3$ and $\Delta(G[W])\geq 4$. 

Now, suppose that $d_{G[W]}(w_4)\geq 4$ 
and assume without loss of generality that $u_1,\ldots,u_4\in N_G(w_4)$.
Then $v_2$ is not adjacent to $v_1, v_3, w_1, w_2$ 
and $d_{G[U]}(u_i)\leq 1$ for $1\leq i\leq 4$, since $T_L(8)\nsubseteq G$. 
Since $d_{G[U]}(v_3)\leq 1$, assume without loss of generality that $v_3$ is not adjacent to $u_3$ or $u_4$.
Now, suppose that $E_G(\{u_1,\ldots,u_4\},\{u_5,u_6,u_7\})\neq \emptyset$
and assume, say, that $u_1$ is adjacent to $u_5$. 
Then $u_1$ is not adjacent to $\{v_3,w_1,w_2,w_3,u_2,\ldots,u_7\}$. 
Since $T_L(8)\nsubseteq G$, 
at least one of $w_1$ and $w_2$ is adjacent in $\overline{G}$ to $u_2$, $u_3$ and $u_4$, 
say $w_1$, 
so $v_1u_2w_1u_3v_3u_4v_2u_6v_1$ and $u_1$ form $W_8$ in $\overline{G}$, a contradiction.
Thus, $E_G(\{u_1,\ldots,u_4\},\{u_5,u_6,u_7\})=\emptyset$. 
Then $u_1u_5u_2u_6u_3u_7u_4v_2u_1$ and $v_1$ form $W_8$ in $\overline{G}$, a contradiction.
Therefore, $d_{G[W]}(u')\geq 4$ for some vertex of $u'\in U$, say $u'=u_1$.

Suppose that $w_4$ is adjacent to $u_1$.
Then without loss of generality, assume that $u_1$ is adjacent to $u_2$, $u_3$ and $u_4$. 
Since $T_L(8)\nsubseteq G$, 
neither $v_0$ nor $w_4$ is adjacent to $w_1$ or $w_2$, 
and $w_4$ is not adjacent to $\{v_1, v_3\}\cup U\setminus \{u_1\}$. 
If $E_G(\{u_2,u_3,u_4\},\{u_5,u_6,u_7\})\neq \emptyset$, 
then, say, $u_2$ is adjacent to $u_5$ and is thus not adjacent to $\{v_0,v_3,w_1,w_2,w_3,u_3,u_4,u_6,u_7\}$, 
so $w_1v_0w_2w_4u_3v_1u_4v_2w_1$ and $u_2$ form $W_8$ in $\overline{G}$, a contradiction.
Thus $E_G(\{u_1,\ldots,u_4\},\{u_5,u_6,u_7\}=\emptyset$.
Let $X=\{v_1,u_2,u_3,u_4\}$ and $Y=\{v_3,u_5,u_6,u_7\}$.
Since $d_{G[U]}(v_3)\leq 1$, 
$\overline{G}[X\cup Y]$ contains $C_8$ by Lemma~\ref{lem:1-1relation} which, with $w_4$, forms $W_8$, a contradiction. 

Thus, $u_1$ is not adjacent to $w_4$
so assume without loss of generality that $u_2,\ldots,u_5\in N_G(u_1)$. 
Since $G$ does not contain $T_L(8)$, 
$d_{G[V(T)]}(u_i)\leq 1$ for $2\leq i\leq 5$. 
If $u_2$ is adjacent to $w_4$, 
then $u_2$ is not adjacent to $V(G)\setminus \{u_1,w_4\}$ in~$G$. 
Since $d_{G[U]}(v_3)\leq 1$, 
that $v_3$ is not adjacent to, say, $u_3$ or $u_4$.
Since $d_{G[V(T)]}(u_i)\leq 1$ for $2\leq i\leq 5$, 
$u_4$ and $u_5$ are each adjacent in $\overline{G}$ to at least $2$ of $w_1,w_2,w_3$, 
so some $w_i\in\{w_1,w_2,w_3\}$ is adjacent in $\overline{G}$ to both $u_4$ and $u_5$.
Therefore, $u_3v_3u_4w_iu_5v_2u_6v_1u_3$ and $u_2$ form $W_8$ in $\overline{G}$, a contradiction.
Thus, $u_2$ is not adjacent to $w_4$. 
Similarly, $u_3, u_4, u_5$ are not adjacent to $w_4$. 
Similar arguments show that $u_2,\ldots,u_5$ are not adjacent to $w_1$ or $w_2$.

Now, if $u_2$ is adjacent to any other vertex of $V(T)$, 
then $u_2$ is not adjacent to $\{u_3,u_4,u_5\}$, 
so $u_3w_1u_4w_4u_5v_2u_6v_1u_3$ and $u_2$ form $W_8$ in $\overline{G}$, a contradiction.
Hence, $u_2$ is not adjacent to $V(T)$ and, similarly, $u_3,u_4,u_5$ are not adjacent to $V(T)$. 
Therefore, by Observation~\ref{obs2}, $\delta(G[V(T)])\geq 4$.
By Lemma~\ref{lm:R(TJLM(n),W8)}, $G[V(T)]$ contains $T_L(8)$, a contradiction. 
Thus, $R(T_L(8),W_8)\leq 15$. 

Now suppose that $n\geq 12$.
If $G$ contains $S_n(1,3)$, 
then the previous arguments above lead to contradictions. 
Thus, $G$ does not contain $S_n(1,3)$.
By Theorem~\ref{thm:R(TC,W8)}, $G$ has a subgraph $T=T_C(n)$. 
Let $V(T)=\{v_0,\ldots,v_{n-4},w_1,w_2,w_3\}$ 
and $E(T)=\{v_0v_1,\ldots,v_0v_{n-4},v_1w_1,v_2w_2,v_2w_3\}$. 
Set $U=V(G)-V(T)=\{u_1,\ldots,u_{n-1}\}$; then $|U|=n-1$. 

Suppose that $w_2$ is not adjacent to $U$. 
If $\delta(\overline{G}[U])\geq \frac{n-1}{2}$, 
then $G$ contains $C_8$ by Lemma~\ref{lem:pancyclic} and, with $w_2$ as hub, forms $W_8$, a contradiction. 
Therefore, $\delta(\overline{G}[U])<\frac{n-1}{2}$ 
and so $\Delta(G[U])\geq \lfloor \frac{n-1}{2} \rfloor \geq 5$. 
Without loss of generality, assume that $u_2,\ldots,u_6\in N_G(u_1)$. 
Since $S_n(1,3)\nsubseteq G$, 
$u_2,\ldots,u_6$ are not adjacent to $V(T)\setminus \{v_0\}$. 
If $u_2$ is adjacent to $v_0$, 
then since $S_n(1,3)\nsubseteq G$, 
$u_3,\ldots,u_6$ are not adjacent to $\{u_7,\ldots,u_{n-1}\}$, 
so $u_3u_7u_4u_8u_5u_9u_6u_{10}u_3$ and $w_2$ form $W_8$ in $\overline{G}$, a contradiction.
Thus, $u_2$ is not adjacent to $v_0$ and, similarly, $u_3,\ldots,u_6$ are also not adjacent to $v_0$. 
Hence, $u_2,\ldots,u_6$ are not adjacent to $V(T)$. 
Therefore, by Observation~\ref{obs2}, $\delta(G[V(T)])\geq n-4$, 
so $G[V(T)]$ contains $T_L(n)$ by Lemma~\ref{lm:R(TJLM(n),W8)}, a contradiction. 

Thus some vertex of $U$, say $u_{n-1}$, is adjacent to $w_2$. 
Set $U'=U\setminus \{u_{n-1}\}$; then $|U'|=n-2$. 
Since $T_L(n)\nsubseteq G$, 
$u_{n-1}$ is not adjacent to $U'$ in~$G$. 
Now, if $\delta(\overline{G}[U'])\geq \frac{n-2}{2}$, 
then $\overline{G}[U']$ contains $C_8$ by Lemma~\ref{lem:pancyclic} which, with $u_{n-1}$, forms $W_8$, a contradiction. 
Thus, $\delta(\overline{G}[U'])\leq \frac{n-2}{2}-1$, 
and $\Delta(G[U'])\geq \frac{n-2}{2}\geq 5$. 
Without loss of generality, assume that $u_2,\ldots,u_6\in N_G(u_1)$
and repeat the above arguments to prove that $u_2,\ldots,u_6$ are not adjacent to $V(T)$ . 
Therefore, $\delta(G[V(T)])\geq n-4$ by Observation~\ref{obs2}, 
so $G[V(T)]$ contains $T_L(n)$ by Lemma~\ref{lm:R(TJLM(n),W8)}, 
a contradiction. 

Therefore, $R(T_L(n),W_8)\leq 2n-1$ for $n\equiv 0 \pmod{4}$
which completes the proof. 
\end{proof}

\begin{theorem}\label{thm:R(TN(n),W8)}
If $n\geq 9$, then $R(T_M(n),W_8)=2n-1$.
\end{theorem}

\begin{proof}
Lemma~\ref{lem:lower-bound-on-S_n(1,4)-T_D(n)-...-T_L(n)} provides the lower bound, 
so it remains to prove the upper bound. 
Let $G$ be any graph of order $2n-1$. 
Assume that $G$ does not contain $T_M(n)$ and that $\overline{G}$ does not contain~$W_8$. 
By Theorem~\ref{thm:R(Sn(4),W8)}, $G$ has a subgraph $T=S_n(4)$. 
Let $V(T)=\{v_0,\ldots,v_{n-4},w_1,w_2,w_3\}$ 
and $E(T)=\{v_0v_1,\ldots,v_0v_{n-4},v_1w_1,v_1w_2,v_1w_3\}$. 
Set $V=\{v_2,v_3,\ldots,v_{n-4}\}$ and $U=V(G)-V(T)=\{u_1,\ldots,u_{n-1}\}$; 
then $|V|=n-5$ and $|U|=n-1$. 
Since $T_M(n)\nsubseteq G$, 
$w_1$, $w_2$ and $w_3$ are not adjacent to any vertex of $U\cup V$ in~$G$. 

Now, suppose that some vertex in $V$ is adjacent to at least $4$ vertices of $U$ in $G$, 
say $v_2$ to $u_1,\ldots,u_4$. 
Then $u_1,\ldots,u_4$ are not adjacent to other vertices in $U$. 
Then $u_1u_5u_2u_6u_3u_7u_4u_8u_1$ and $w_1$ form $W_8$ in $\overline{G}$, a contradiction.
Therefore, each vertex in $V$ is adjacent to at most three vertices of $U$ in~$G$. 
Choose any $8$ vertices of $U$.
By Corollary~\ref{cor:bipartite4-8}, 
$\overline{G}[U\cup V]$ contains $C_8$ which together with $w_1$ gives $W_8$ 
in $\overline{G}$, a contradiction. 

Thus, $R(T_M(n),W_8)\leq 2n-1$ for $n\geq 9$. 
This completes the proof.
\end{proof}

\begin{theorem}
If $n\geq 9$, then 
\[
  R(T_N(n),W_8)=\begin{cases}
    2n-1 & \text{if $n\not\equiv 0 \pmod{4}$}\,;\\
    2n   & \text{otherwise}\,.
  \end{cases}
\]  
\end{theorem}

\begin{proof}
Lemma~\ref{lem:lower-bound-on-S_n(1,4)-T_D(n)-...-T_L(n)} provides the lower bound, 
so it remains to prove the upper bound. 
Let $G$ be any graph of order $2n$ if $n\equiv 0 \pmod{4}$ 
and of order $2n-1$ if $n\not\equiv 0 \pmod{4}$. 
Assume that $G$ does not contain $T_N(n)$ and that $\overline{G}$ does not contain $W_8$. 
By Theorem~\ref{thm:R(TA,W8)}, $G$ has a subgraph $T=T_A(n)$. 
Let $V(T)=\{v_0,\ldots,v_{n-4},w_1,w_2,w_3\}$ 
and $E(T)=\{v_0v_1,\ldots,v_0v_{n-4},v_1w_1,v_1w_2,w_1w_3\}$. 
Set $V=\{v_2,v_3,\ldots,v_{n-4}\}$ and $U=V(G)-V(T)=\{u_1,\ldots,u_j\}$, 
where $j=n-1$ if $n\not\equiv 0 \pmod{4}$ and $j=n$ otherwise.
Since $T_N(n)\nsubseteq G$, 
$w_2$ is not adjacent to $U\cup V$ in~$G$. 
If each $v_i\in V$ is adjacent to at most three vertices of $U$ in $G$, 
then by Corollary~\ref{cor:bipartite4-8}, 
$\overline{G}[U\cup V]$ contains $C_8$ which with $w_2$ gives $W_8$ in $\overline{G}$, a contradiction. 
Therefore, some vertex in $V$, say $v_2$, 
is adjacent to at least four vertices of $U$ in~$G$, 
say $u_1,\ldots,u_4$. 
If none of these is adjacent to other vertices of $U$ in $G$, 
then $u_1u_5u_2u_6u_3u_7u_4u_8u_1$ and $w_2$ form $W_8$ in $\overline{G}$, 
a contradiction. 

Therefore, assume that $u_1$ is adjacent to $u_5$ in~$G$. 
Since $T_N(n)\nsubseteq G$, 
$u_2,u_3,u_4$ are not adjacent to $\{u_6,\ldots,u_j\}$ in~$G$. 
For $n=9$ and $n=10$, $\{v_3,\ldots,v_{n-4}\}$ is not adjacent to $\{u_5,\ldots,u_{n-1}\}$
or else $G$ will contain $T_N(n)$ 
with $v_2$ and $v_0$ being the vertices of degree $n-5$ and $3$, respectively. 
However, $v_3u_5v_4u_6u_2u_7u_3u_8v_3$ and $w_2$ form $W_8$ in $\overline{G}$, 
a contradiction. 
For $n\geq 11$, if $v_2$ is not adjacent to $\{u_6,\ldots,u_j\}$ in $G$, 
then $v_2u_6u_2u_7u_3u_8u_4u_9v_2$ and $w_2$ form $W_8$ in $\overline{G}$, a contradiction. 
Therefore, assume that $v_2$ is adjacent to $u_6$ in $G$. 
Then $u_6$ is not adjacent to $\{u_7,\ldots,u_j\}$ in~$G$, 
and $u_2u_7u_3u_8u_4u_9u_6u_{10}u_2$ and $w_2$ form $W_8$ in $\overline{G}$, 
again a contradiction. 

Thus, $R(T_N(n),W_8)\leq 2n$   for $n\equiv 0 \pmod{4}$ 
and   $R(T_N(n),W_8)\leq 2n-1$ for $n\not\equiv 0 \pmod{4}$. 
\end{proof}

\begin{theorem}
If $n\geq 9$, then $R(T_P(n),W_8)=2n-1$.
\end{theorem}

\begin{proof}
Lemma~\ref{lem:lower-bound-on-S_n(1,4)-T_D(n)-...-T_L(n)} provides the lower bound, 
so it remains to prove the upper bound. 
Let $G$ be any graph of order $2n-1$. 
Assume that $G$ does not contain $T_P(n)$ and that $\overline{G}$ does not contain $W_8$. 
Suppose $n\not\equiv 0 \pmod{4}$.
By Theorem~\ref{thm:R(TA,W8)}, $G$ has a subgraph $T=T_A(n)$. 
Let $V(T)=\{v_0,\ldots,v_{n-4},w_1,w_2,w_3\}$ 
and $E(T)=\{v_0v_1,\ldots,v_0v_{n-4},v_1w_1,v_1w_2,w_1w_3\}$. 
Set $V=\{v_2,v_3,\ldots,v_{n-4}\}$ and $U=V(G)-V(T)$; 
then $|V|=n-5$ and $|U|=n-1$. 
Since $T_P(n)\nsubseteq G$, 
$w_1$ is not adjacent to any vertex of $U\cup V$ in~$G$. 
If each $v_i$ in $V$ is adjacent to at most three vertices of $U$ in $G$, 
then by Corollary~\ref{cor:bipartite4-8}, 
$\overline{G}[U\cup V]$ contains $C_8$ which with $w_1$ gives $W_8$ in $\overline{G}$, 
a contradiction. 
Therefore, some vertex in $V$, say $v_2$, 
is adjacent to at least four vertices of $U$ in~$G$, 
say $u_1,\ldots,u_4$. 
For $n=9$ and $n=10$, 
$G$ contains $T_P(9)$ and $T_P(10)$ with edge set 
$\{u_1v_2,u_2v_2,u_3v_2,v_2v_0,v_0v_1,v_0v_3,v_1w_1,v_1w_2\}$ and
$\{u_1v_2,u_2v_2,u_3v_2,u_4v_2,v_2v_0,v_0v_1,v_0v_3,v_1w_1,v_1w_2\}$, respectively. 
For $n\geq 11$, each of $u_1,\ldots,u_4$ is adjacent to 
at most two remaining vertices in~$U$. 
Then by Corollary~\ref{cor:bipartite4-6}, 
$\overline{G}[U]$ contains $C_8$ which with $w_1$ gives $W_8$ in $\overline{G}$, 
a contradiction. 

Suppose that $n\equiv 0 \pmod{4}$.
By Theorem~\ref{thm:R(TN(n),W8)}, 
$G$ contains a subgraph $T=T_M(n)$. 
Let $V(T)=\{v_0,\ldots,v_{n-5},w_1,\ldots,w_4\}$ 
and $E(T)=\{v_0v_1,\ldots,v_0v_{n-5},v_1w_1,v_1w_2,v_1w_3,w_1w_4\}$. 
Let $V=\{v_2,v_3,\ldots,v_{n-5}\}$ and $U=V(G)-V(T)$; 
then $|V|=n-6$ and $|U|=n-1$. 
Since $T_P(n)\nsubseteq G$, 
$w_1$ is not adjacent to $\{v_0,w_2,w_3\}\cup U$ in $G$, 
and so $d_{G[U]}(w_2)\leq 1$, 
       $d_{G[U]}(w_3)\leq 1$ 
   and $d_{G[U]}(v)  \leq n-7$ for any vertex $v\in V$. 
Now, if $G$ contains a subgraph $T_A(n)$, 
then arguments similar to those used for the case $n\not\equiv 0 \pmod{4}$ above can be used.
Therefore, $G$ contains no $T_A(n)$. 
Then $v_0$ is not adjacent to $\{w_2,w_3\}\cup U$ in~$G$.

Suppose that some vertex $v\in V$ is not adjacent to $w_1$ in $G$.
Let $X$ be any four vertices in $U$ that are not adjacent to $v$ in $G$ 
and set $Y=\{v,v_0,w_2,w_3\}$. 
By Lemma~\ref{lem:1-1relation}, 
$\overline{G}[X\cup Y]$ contains $C_8$ which with $w_1$ gives $W_8$ in $\overline{G}$, 
a contradiction. 
Therefore, each vertex of $V$ is adjacent to $w_1$ in~$G$. 
Since $T_P(n)\nsubseteq G$, 
$w_4$ is adjacent to at most $n-7$ vertices of $U$ in $G$.
Since $T_A(n)\nsubseteq G$, $w_2$ and $w_3$ are not adjacent in~$G$. 
Now, if $w_4$ is adjacent to both $w_2$ and $w_3$ in $G$, 
then $w_4$ is not adjacent to $v_0$ in $G$ since $T_P(n)\nsubseteq G$. 
Let $X$ be any four vertices of $U$ that are not adjacent to $w_4$ in $G$ 
and let $V=\{w_1,\ldots,w_4\}$. 
By Lemma~\ref{lem:1-1relation}, 
$\overline{G}[X\cup Y]$ contains $C_8$ which with $w_1$ gives $W_8$ in $\overline{G}$, 
a contradiction. 
Therefore, $w_4$ is non-adjacent to either $w_2$ or $w_3$ in $G$, say $w_2$. 
Since $d_{G[U]}(w_2)\leq 1$ and $d_{G[U]}(w_4)\leq n-7$, 
there is a set $X$ of four vertices in $U$ 
that are not adjacent to both $w_2$ and $w_4$ in $G$.
Let $Y=\{v_0,w_1,w_3,w_4\}$.
By Lemma~\ref{lem:1-1relation}, $\overline{G}[X\cup Y]$ contains $C_8$ 
which with $w_1$ gives $W_8$ in $\overline{G}$, again a contradiction. 

In either case, $R(T_P(n),W_8)\leq 2n-1$ for $n\geq 9$ and this completes the proof.
\end{proof}

\begin{theorem}
If $n\geq 9$, then $R(T_Q(n),W_8)=2n-1$.
\end{theorem}

\begin{proof}
Lemma~\ref{lem:lower-bound-on-S_n(1,4)-T_D(n)-...-T_L(n)} provides the lower bound, 
so it remains to prove the upper bound. 
Let $G$ be any graph of order $2n-1$. 
Assume that $G$ does not contain $T_Q(n)$ and that $\overline{G}$ does not contain $W_8$. 
By Theorem~\ref{thm:R(Sn(4),W8)}, $G$ has a subgraph $T=S_n(4)$. 
Let $V(T)=\{v_0,\ldots,v_{n-4},w_1,w_2,w_3\}$ 
and $E(T)=\{v_0v_1,\ldots,v_0v_{n-4},v_1w_1,v_1w_2,v_1w_3\}$. 
Set $V=\{v_2,v_3,\ldots,v_{n-4}\}$ and $U=V(G)-V(T)$; 
then $|V|=n-5$ and $|U|=n-1$. 
Since $T_Q(n)\nsubseteq G$, 
$G[V]$ are independent vertices and not adjacent to~$U$. 

Suppose that $n\geq 10$.
Then $|V|\geq 5$ and $|U|\geq 9$, 
so by Observation~\ref{obs2}, 
$\overline{G}$ contains $W_8$, a contradiction. 
If $n=9$, then $|V|=4$ and $|U|=8$. 
By Lemma~\ref{lm4:R(Sn(1,2),W8)}, $G[U]$ is $K_8$ or $K_8-e$. 
Since $T_Q(9) \nsubseteq G$, 
$T$ is not adjacent to $U$, and $\delta(G[V(T)]\geq 5$. 
As $v_2,\ldots,v_5$ are independent in $G$, 
they are each adjacent to all other vertices in $G[V(T)]$, 
Therefore, $G[V(T)]$ contains $T_Q(9)$ with $v_2$ and $v_0$ as the vertices of degree $4$,
a contradiction. 

Thus, $R(T_Q(n),W_8)\leq 2n-1$ for $n\geq 9$ which completes the proof.
\end{proof}

\begin{theorem}
If $n\geq 9$, then $R(T_R(n),W_8)=2n-1$.
\end{theorem}

\begin{proof}
Lemma~\ref{lem:lower-bound-on-S_n(1,4)-T_D(n)-...-T_L(n)} provides the lower bound, 
so it remains to prove the upper bound. 
Let $G$ be any graph of order $2n-1$. 
Assume that $G$ does not contain $T_R(n)$ and that $\overline{G}$ does not contain $W_8$. 
By Theorem~\ref{thm:R(TC,W8)}, $G$ has a subgraph $T=T_C(n)$. 
Let $V(T)=\{v_0,\ldots,v_{n-4},w_1,w_2,w_3\}$ 
and $E(T)=\{v_0v_1,\ldots,v_0v_{n-4},v_1w_1,v_2w_2,v_2w_3\}$. 
Set $V=\{v_3,\ldots,v_{n-4}\}$ and $U=V(G)-V(T)=\{u_1,\ldots,u_{n-1}\}$; 
then $|V|=n-6$ and $|U|=n-1$. 
Since $T_R(n)\nsubseteq G$, $w_1$ is not adjacent in $G$ to any vertex of $U\cup V$. 
If $\delta(\overline{G}[U\cup V])\geq \lceil\frac{2n-7}{2}\rceil$, 
then $\overline{G}[U\cup V]$ contains $C_8$ by Lemma~\ref{lem:pancyclic} 
which with $w_3$ forms $W_8$, a contradiction. 
Therefore, $\delta(\overline{G}[U\cup V])\leq \lceil\frac{2n-7}{2}\rceil-1$, 
and $\Delta(G[U\cup V])\geq \lfloor\frac{2n-7}{2}\rfloor=n-4$. 
Now, there are two cases to be considered. 

\smallskip
\noindent {\bf Case 1}: 
One of the vertices of $V$, say $v_3$, is a vertex of degree at least $n-4$ in $G[U\cup V]$. 

Note that in this case, there are at least $3$ vertices from $U$, say $u_1,u_2,u_3$, 
that are adjacent to $v_3$ in~$G$. 
Suppose that $v_3$ is also adjacent to $a$ in $G$, where $a$ is a vertex in $U\cup V$.
Since $T_R(n)\nsubseteq G$, 
these $4$ vertices are independent and are not adjacent to any other vertices of $U$. 
Since $n\geq 9$, $U$ contains at least $4$ other vertices, say $u_5,\ldots,u_8$, 
so $u_1u_5u_2u_6u_3u_7au_8u_1$ and $w_3$ form $W_8$ in $\overline{G}$, a contradiction.

\smallskip
\noindent {\bf Case 2}: Some vertex $u\in U$ has degree at least $n-4$ in $G[U\cup V]$.

Since $T_R(n)\nsubseteq G$, $u$ is not adjacent to any vertex of $V$ in~$G$. 
Therefore, $u$ must be adjacent to at least $n-4$ vertices of $U$ in~$G$. 
Without loss of generality, suppose that $u_1,\ldots,u_{n-4}\in N_{G[U]}(u)$. 
Note that $V$ is not adjacent to $N_{G[U]}(u)$, 
or else it will form $T_R(n)$ in $G$, a contradiction. 
If $n\geq 10$, 
then any $4$ vertices from $N_{G[U]}(u)$ and any $4$ vertices from $V$ 
form $C_8$ in $\overline{G}$ which with $w_3$ forms $W_8$, a contradiction. 
Suppose that $n=9$ and let the remaining two vertices be $u_6$ and $u_7$. 
If either $u_6$ or $u_7$ is non-adjacent to any two vertices of $\{u_1,\ldots,u_5\}$ in $G$, 
say $u_6$ is not adjacent to $u_1$ and $u_2$ in $G$, 
then $u_1u_6u_2v_3u_3v_4u_4v_5u_1$ and $w_3$ form $W_8$ in $\overline{G}$, 
a contradiction. 
So, both $u_6$ and $u_7$ are adjacent to at least $4$ vertices of $\{u_1,\ldots,u_5\}$ in~$G$. 
Since $T_R(9)\nsubseteq G$, $T$ cannot be adjacent to $U$, and $\delta(G[V(T)]\geq 5$. 
As both $v_2$ and $w_3$ are not adjacent to $v_3$, $v_4$ or $v_5$ in $G$, 
they are adjacent to all other vertices in $G[V(T)]$. 
Similarly, since $v_3$ is not adjacent to $v_2$ or $w_3$ in $G$, 
$v_3$ is adjacent to $w_1$ or $w_2$ in $G$.
Without loss of generality, assume that $v_3$ is adjacent to $w_1$. 
Then $G[V(T)]$ contains $T_R(9)$ with edge set 
$\{v_2w_2,v_2v_1,v_2v_0,v_0v_4,v_0v_5,v_2w_3,v_2w_1,w_1v_3\}$, a contradiction.

In either case, $R(T_R(n),W_8)\leq 2n-1$. 
\end{proof}

\begin{theorem}
If $n\geq 9$, then $R(T_S(n),W_8)=2n-1$.
\end{theorem}

\begin{proof}
Lemma~\ref{lem:lower-bound-on-S_n(1,4)-T_D(n)-...-T_L(n)} provides the lower bound, 
so it remains to prove the upper bound. 
Let $G$ be any graph of order $2n-1$. 
Assume that $G$ does not contain $T_S(n)$ and that $\overline{G}$ does not contain $W_8$. 
Suppose $n\not\equiv 0 \pmod{4}$.
By Theorem~\ref{thm:R(Sn[4],W8)}, $G$ has a subgraph $T=S_n[4]$. 
Let $V(T)=\{v_0,\ldots,v_{n-4},w_1,w_2,w_3\}$ 
and $E(T)=\{v_0v_1,\ldots,v_0v_{n-4},v_1w_1,w_1w_2,w_1w_3\}$. 
Set $V=\{v_2,\ldots,v_{n-4}\}$ and $U=V(G)-V(T)$; 
then $|V|=n-5$ and $|U|=n-1$. 
Since $T_S(n)\nsubseteq G$, 
$G[V]$ are independent vertices and are not adjacent to~$U$. 
If $n\geq 10$, 
then $|V|\geq 5$ and $|U|\geq 9$, 
so by Observation~\ref{obs2}, $\overline{G}$ contains $W_8$, a contradiction. 
Suppose that $n=9$.
Then $|V|=4$ and $|U|=8$. 
By Lemma~\ref{lm4:R(Sn(1,2),W8)}, 
$G[U]$ is $K_8$ or $K_8-e$.
Since $T_S(9) \nsubseteq G$, $T$ is not adjacent to $U$, and $\delta(G[V(T)]\geq 5$. 
As $v_2,\ldots,v_5$ are independent in $G$, 
they are adjacent to all other vertices in $G[V(T)]$,
and so $G[V(T)]$ contains $T_S(9)$ with edge set 
$\{v_0v_1,v_0v_2,v_1v_4,v_1v_5,v_2w_1,v_2w_2,v_2w_3,v_3w_1\}$. 

Now suppose that $n\equiv 0 \pmod{4}$. 
By Theorem~\ref{thm:R(Sn[4],W8)}, $G$ has a subgraph $T=S_{n-1}[4]$. 
Let $V(T)=\{v_0,\ldots,v_{n-5},w_1,w_2,w_3\}$ 
and $E(T)=\{v_0v_1,\ldots,v_0v_{n-5},v_1w_1,w_1w_2,w_1w_3\}$. 
Set $V=\{v_2,\ldots,v_{n-5}\}$ and $U=V(G)-V(T)$; 
then $|V|=n-6$ and $|U|=n$. 
Since $T_S(n)\nsubseteq G$, 
$G[V]$ is not adjacent to~$U$. 
Since $|V|=n-6>4$, by Observation~\ref{obs2}, 
$\Delta(\overline{G}[U])\leq 3$ and $\delta(G[U])\geq n-4$ 
since $W_8\nsubseteq\overline{G}$. 
By Lemma~\ref{lm:R(Sn[4],W8)}, 
either $G[U]$ is $K_{4,\ldots,4}$ and contains $T_S(n)$ 
or $G[U]$ contains $S_n[4]$ 
and the arguments from the $n\not\equiv 0\pmod{4}$ case above lead to a contradiction. 

Thus, $R(T_S(n),W_8)\leq 2n-1$ for $n\geq 9$ which completes the proof.
\end{proof}

\section*{Declarations}

The authors have no relevant financial or non-financial interests to disclose.

\end{document}